\documentclass[12pt, reqno]{amsart}
\usepackage{}
\usepackage{amsfonts}
\usepackage{bbm}
\usepackage{amscd,amsfonts}
\usepackage{amssymb, eucal, amsfonts, amsmath, xypic, latexsym}
\usepackage{pifont}
\usepackage{mathrsfs,color}
\usepackage{amsthm,indentfirst,bm,fancyhdr,dsfont}
\usepackage{graphicx}
\usepackage[all]{xy}
\usepackage[CJKbookmarks=true]{hyperref}

\newtheorem{theorem}{Theorem}[section]
\newtheorem{lemma}[theorem]{Lemma}
\newtheorem{prop}[theorem]{Proposition}
\newtheorem{corollary}[theorem]{Corollary}
\theoremstyle{definition}
\newtheorem{conj}[theorem]{Conjecture}
\newtheorem{defn}[theorem]{Definition}

\newtheorem{rem}[theorem]{Remark}

\numberwithin{equation}{section}

\def\ggg{\mathfrak{g}}

\def\ggg{\mathfrak{g}}
\def\mmm{\mathfrak{m}}
\def\ppp{\mathfrak{p}}

\def\bbc{\mathbb{C}}
\def\bbf{\mathbb{F}}
\def\bbz{\mathbb{Z}}

\def\bbk{{\mathds{k}}}

\def\bo{{\bar 1}}
\def\bz{{\bar 0}}
\def\ev{{\text{ev}}}

\def\sfr{\textsf{r}}

\def\Lie{\text{Lie}}
\def\ad{\text{ad}}

{\vskip-\lastskip\medskip
  \noindent
  {\em #1.}\enspace
  }%
{\qed\par\medskip
  }

\begin{document}
\title[Minimal $W$-superalgebras and modular representations]{Minimal $W$-superalgebras and the modular representations of basic Lie superalgebras}
\author{Yang Zeng and Bin Shu}
\thanks{\nonumber{{\it{Mathematics Subject Classification}} (2010):
Primary 17B50, Secondary 17B05, 17B35 and 17B81.
 {\it{The Key words}}: finite $W$-(super)algebras, basic (classical) Lie superalgebras, minimal nilpotent elements, modular representations of Lie (super)algebras, Kac-Weisfeiler conjecture (property) for modular Lie (super)algebras.
 This work is  supported partially by the NSFC (Nos. 11701284; 11671138; 11401312), Shanghai Key Laboratory of PMMP (No. 13dz2260400),  Natural Science Foundation for Colleges and Universities in Jiangsu Province (No. 17KJB110004). }}
\address{School of Science, Nanjing Audit University, Nanjing, Jiangsu Province 211815, China}
\email{zengyang214@163.com}
\address{Department of Mathematics, East China Normal University, Shanghai 200241, China}
\email{bshu@math.ecnu.edu.cn}
\begin{abstract}
Let $\ggg=\ggg_\bz+\ggg_\bo$ be a basic Lie superalgebra over $\mathbb{C}$, and $e$ a minimal nilpotent element in $\ggg_\bz$. Set $W_\chi'$ to be the refined $W$-superalgebra associated with the pair $(\ggg,e)$, which is called a minimal $W$-superalgebra. In this paper we present a set of explicit generators of minimal $W$-superalgebras and the commutators between them. In virtue of this, we show that over an algebraically closed field $\bbk$ of characteristic $p\gg0$, the lower bounds of dimensions in the modular representations of basic Lie superalgebras with minimal nilpotent $p$-characters are attainable. Such lower bounds are  indicated  in \cite{WZ} as the super Kac-Weisfeiler property.
\end{abstract}
\maketitle
\setcounter{tocdepth}{1}
\tableofcontents
\section{Introduction}

This work  is a sequel to \cite{ZS2}-\cite{ZS4}. In \cite{ZS4}, we asserted that the lower bounds of modular representation dimensions stated in the abstract are attainable under the assumption that the associated complex finite $W$-superalgebras admit one- or two-dimensional representations. In the present paper, we certify the assumption in the case of minimal nilpotent elements.

\subsection{} A finite $W$-algebra $U(\ggg,e)$ is a certain associative algebra associated with a complex semisimple Lie algebra $\mathfrak{g}$ and a nilpotent element $e\in{\ggg}$. The study of finite $W$-algebras can be traced back to Kostant's work in the case when $e$ is regular \cite{Ko}, then a further study was done by Lynch in the case when $e$ is arbitrary even nilpotent element (cf. \cite{Ly}). Premet developed finite $W$-algebras in full generality in \cite{P2}. On his way to proving the celebrated Kac-Weisfeiler conjecture for Lie algebras of reductive groups in \cite{P1}, Premet first constructed the modular version of finite $W$-algebras in \cite{P2}. By means of a complicated but natural ``admissible'' procedure, the finite $W$-algebras over the field of complex numbers were introduced in \cite{P2}, which shows that they are filtered deformations of the coordinate rings of Slodowy slices.

Aside from the advances in finite $W$-algebras over complex numbers, the modular theory of finite $W$-algebras has also developed excitingly. It is worth noting  that in \cite{P7} Premet proved that if the $\bbc$-algebra $U(\ggg,e)$ has a one-dimensional representation, then under the assumption $p\gg 0$ for the ground algebraically closed field $\bbk$ of positive characteristic $p$, the reduced enveloping algebra $U_\chi(\ggg_\bbk)$ of the modular counterpart $\ggg_\bbk$ of $\ggg$ possesses an irreducible module of dimension $d(e)$ (where $\chi$ is the linear function on $\ggg_{{\bbk}}$ corresponding to $e$, and $d(e)$ is half of the dimension of the orbit $G_\bbk\cdot \chi$ for the simple, simply connected algebraic group $G_\bbk$ with $\ggg_\bbk=\Lie(G_\bbk)$), which is a lower bound predicted by the Kac-Weisfeiler conjecture mentioned above.

The existence of one-dimensional representations for $U(\ggg,e)$ associated with $\ggg=\text{Lie}(G)$ of a simple algebraic group $G$ over $\mathbb{C}$ was conjectured by Premet, and confirmed in the classical cases by Losev in \cite[Theorem 1.2.3(1)]{L3} (see also \cite[\S6]{L1}). Goodwin-R\"{o}hrle-Ubly \cite{GRU} proved that the finite $W$-algebras associated with exceptional Lie algebras $E_6,E_7,F_4,G_2$, or $E_8$ with $e$ not rigid, admit one-dimensional representations (see also \cite{P7}). Finally Premet solved this problem completely in \cite{P9}.

\subsection{}
The theory of finite $W$-superalgebras was developed in the same time.  In the work of Sole and Kac \cite{SK}, finite $W$-superalgebras were defined in terms of BRST cohomology under the background of vertex algebras and quantum reduction. The theory of finite $W$-superalgebras for the queer Lie superalgebras over an algebraically closed field of characteristic $p>2$ was first introduced and discussed by Wang and Zhao in \cite{WZ2}, then studied by Zhao over the field of complex numbers in \cite{Z2}. The topics on finite $W$-superalgebras attracted many researchers, and the structure theory of $W$-superalgebras is developed in various articles (cf.  \cite{BBG2},  \cite{BBG3}, \cite{Peng2}, \cite{Peng3}, \cite{PS} and \cite{PS2}, $\sl{etc.}$).

In mathematical physics, $W$-(super)algebras are divided into four types: classical affine, classical finite, quantum affine, and quantum finite $W$-(super)algebras. These types of algebras are endowed with Poisson vertex algebras, Poisson algebras, vertex algebras, and associative algebras structures, respectively. In the present paper, finite $W$-superalgebras will be referred to the so-called quantum finite $W$-(super)algebras.

Apart from the ones associated with principal nilpotent elements, the most elementary examples of finite $W$-superalgebras are those ones corresponding to (even) minimal nilpotent elements of a given Lie superalgebra $\ggg$. For the counterpart associated with Lie algebras, Premet described generators of finite $W$-algebras associated with minimal nilpotent elements in \cite{P3}. Under the background of vertex algebras and quantum reduction, similar results for classical affine $W$-algebras can be found in \cite{SKV} and \cite{suh2}.

For the case of basic Lie superalgebras, the related study associated with (even) minimal nilpotent elements is made  mainly in the context of vertex  operators and quantum reduction (see \cite{AKFPP}, \cite{KRW}, \cite{KW}, {\sl{etc}}.).
Recently, Suh described the generators of a classical affine $W$-(super)algebra associated with a minimal nilpotent element explicitly in \cite{suh}, and also the ones of quantum finite $W$-superalgebras analogously. It is worthwhile reminding  that   there are some errors in the presentation of generators and their relations of finite $W$-superalgebras in \cite{suh}. In the present paper, we will rewrite these generators and their commutators in an analogue of Premet's strategy of finite $W$-algebras case in \cite{P3}.

\subsection{}
In \cite{WZ}, the authors initiated the study of modular representations of basic Lie superalgebras over an algebraically closed field of positive characteristic, formulating the super Kac-Weisfeiler property for those Lie superalgebras as well as presenting the definition of modular $W$-superalgebras. 

\subsection{}
Based on Premet's and Wang-Zhao's work as mentioned above, in \cite{ZS2} we presented the PBW  theorem for the finite $W$-superalgebras over $\bbf$ ($\bbf=\bbc$ or $\bbk$ with characteristic $p\gg0$), which shows that the construction of finite $W$-superalgebras can be divided into two cases by virtue of the so-called judging  parity of the dimension of certain specific subspace for  basic Lie superalgebra $\mathfrak{g_\bbf}$. The situation of finite $W$-superalgebras is significantly different from that of finite $W$-algebras at the odd judging parity.

To be explicit, for a given complex basic classical Lie superalgebra  ${\ggg}={\ggg}_{\bar0}+{\ggg}_{\bar1}$ and a nilpotent element $e\in\ggg_\bz$ (thereby a linear function $\chi$ in $\ggg^*_\bz$. see \S\ref{0.1.1}),
one has a so-called $\chi$-admissible algebra $\mmm$ (see (\ref{admissible alg})).
Consider a generalized Gelfand-Graev ${\ggg}$-module associated with $\chi$
 $$Q_\chi:=U({\ggg})\otimes_{U(\mathfrak{m})}{\bbc}_\chi,$$
where ${\bbc}_\chi={\bbc}1_\chi$ is a one-dimensional  $\mathfrak{m}$-module such that $x.1_\chi=\chi(x)1_\chi$ for all $x\in\mathfrak{m}$.
A finite $W$-superalgebra $U(\ggg,e)$ is by definition equal to    $(\text{End}_{\ggg}Q_{\chi})^{\text{op}}$ which is isomorphic to $Q_{\chi}^{\text{ad}\,{\mmm}}$ (see \cite[Theorem 2.12]{ZS2}),
where $Q_{\chi}^{\text{ad}\,{\mmm}}$ is the invariant subalgebra of $Q_{\chi}$ under the adjoint action of ${\mmm}$.  A PBW theorem of $U(\ggg,e)$ (see \cite[Theorem 0.1]{ZS2}) shows that the structure of $U(\ggg,e)$ is crucially dependent on the parity of a discriminant number $\mathsf{r}$ (the meaning of this notation can be seen in the above of (\ref{extend admiss alg})). The parity of $\mathsf{r}$ is therefore called the judging parity.

\subsection{} When we turn to finite-dimensional representations of finite $W$-superalgebras over complex numbers, their minimal dimensions will be crucial to small representations of modular Lie superalgebras.
Under an assumption on the minimal dimensions of representations for complex finite $W$-superalgebrs, we proved in \cite[Theorem 1.6]{ZS4}  the accessibility of lower-bounds of dimensions for modular representations of basic Lie superalgebras (see the next subsection). Such an assumption is also predicted to be true, as a conjecture listed below (as an analogy of Premet's work on finite $W$-algebras).

\begin{conj}\label{conjectureold}(\cite{ZS4})
Let ${\ggg}$ be a basic Lie superalgebra over ${\bbc}$. Then the following statements hold:
\begin{itemize}
\item[(1)] when $\sfr$ is even, the finite $W$-superalgebra $U({\ggg},e)$ affords a one-dimensional representation;
\item[(2)] when $\sfr$ is odd, the finite $W$-superalgebra $U({\ggg},e)$ affords a two-dimensional representation.
\end{itemize}
\end{conj}
For the case ${\ggg}$ is of type $A(m,n)$, Conjecture \ref{conjectureold} was confirmed in \cite[Proposition 4.7]{ZS4}, which was accomplished by conversion from the verification of the attainableness of lower-bounds of modular dimensions for basic Lie superalgebras of the same type by some direct computation; see \cite{ZS3} for more details. For the case ${\ggg}$ is of type $B(0,n)$ with $e$ being a regular nilpotent element in ${\ggg}$, we certified  Conjecture \ref{conjectureold} in \cite[Proposition 5.8]{ZS4}. In the present paper, we will certify this conjecture for minimal nilpotent elements.


\subsection{}\label{1.5} Let us recall the lower bounds of dimensions in  modular representation for  basic Lie superalgebras. Let $(\cdot,\cdot)$ be a bilinear form on ${\ggg}_{\bbk}$ which is induced from that on ${\ggg}$, and $\chi\in({\ggg}_{\bbk})^*_{\bar0}$ be the nilpotent $p$-character of ${\ggg}_{\bbk}$ corresponding to $\bar e\in ({\ggg}_{\bbk})_{\bar0}$ such that $\chi(\bar y)=(\bar e,\bar y)$ for any $\bar y\in{\ggg}_{\bbk}$; where $\bar{e}=e\otimes1$ is obtained from $e\in\ggg$ by ``reduction modulo $p$" .

Set $d_0=\text{dim}\,({\ggg}_{\bbk})_{\bar 0}-\text{dim}\,({\ggg}_{\bbk}^{\bar e})_{\bar 0}$ and $d_1=\text{dim}\,({\ggg}_{\bbk})_{\bar 1}-\text{dim}\,({\ggg}_{\bbk}^{\bar e})_{\bar 1}$, where ${\ggg}_{\bbk}^{\bar e}$ denotes the centralizer of $\bar e$ in ${\ggg}_{\bbk}$. For any real number $a\in\mathbb{R}$, let $\lfloor a\rfloor$ denote the least integer upper bound of $a$. In \cite[Theorem 5.6]{WZ}, Wang-Zhao showed that the dimension of any irreducible representation of $\ggg_{\bbk}$ is divisible by the number $p^{\frac{d_0}{2}}2^{\lfloor\frac{d_1}{2}\rfloor}$. With this number, we partially answered the question whether there exist modules of dimensions equal to such a number (see \cite[Theorems 1.5 and 1.6]{ZS4})

\vskip0.3cm
The main purpose of the present paper is to certify Conjecture \ref{conjectureold} for the case when $e$ is an (even) minimal nilpotent element, and thereby to show the accessibility of the lower bounds of dimensions in the modular representations of basic Lie superalgebras in this case (see the forthcoming Theorem  \ref{intromainminnimalf}).

\subsection{} Recall that in \cite[Remark 70]{W}, Wang introduced another definition of finite $W$-superalgebra $W'_\chi:=Q_\chi^{\text{ad}\,{\mmm}'}$, where $\mmm'$ is the so-called extended $\chi$-admissible algebra which is either a one-dimensional extension of $\mmm$, or $\mmm$ itself, dependent on the judging parity (see (\ref{extend admiss alg})).
 We will call it a refined $W$-superalgebra in the present paper. By definition, the refined $W$-superalgebra $W'_\chi$ is a subalgebra of the finite $W$-superalgebra $U(\ggg,e)$.

In this paper, we will take use of the refined $W$-superalgebras, instead of our original finite $W$-superalgebras, which enables us to unify the related statements.

\subsection{} Let us introduce the main results in the present paper.

We first establish an isomorphism (Proposition \ref{important}) between the refined $W$-superalgebra $W_\chi'$ and the quantum finite $W$-superalgebra $W^{\text{fin}}(\ggg,e)$ introduced by Suh in \cite[Definition 4.3]{suh}. This enables us to take $W^{\text{fin}}(\ggg,e)$ as the refined $W$-superalgebra $W_\chi'$ for the following discussions.

 Then  we focus  on the case with $e$ being a minimal nilpotent element in basic Lie superalgebra $\ggg$. We call a refined $W$-superalgebra  minimal when the defining nilpotent element associated with $W_\chi'$ is  minimal in $\ggg_\bz$. We can present a set of generators  of minimal $W$-superalgebras as below, correcting and reformulating the ones in   \cite[Propositions 5.3 and 5.4]{suh}.

\begin{prop}\label{ge}
Let $e$ be a minimal nilpotent element in $\ggg$. Suppose $v\in\ggg^e(0)$, $w\in\ggg^e(1)$, $C$ is a central element of $W_\chi'$, and set $s=\text{dim}\,\ggg(-1)_{\bar0}$ and $\sfr=\text{dim}\,\ggg(-1)_{\bar1}$. Then the followings are free generators of the refined $W$-superalgebra $W_\chi'$:
\begin{equation*}
\begin{split}
\Theta_v=&(v-\frac{1}{2}\sum\limits_{\alpha\in S(-1)}z_\alpha[z_\alpha^*,v])\otimes1_\chi,\\
\Theta_w=&(w-\sum\limits_{\alpha\in S(-1)}z_\alpha[z_\alpha^*,w]+\frac{1}{3}(\sum\limits_{\alpha,\beta\in S(-1)}z_\alpha z_\beta[z_\beta^*,[z_\alpha^*,w]]-2[w,f]))\otimes1_\chi,\\
C=&(2e+\frac{h^2}{2}-(1+\frac{s-\sfr}{2})h+\sum\limits_{i\in I}(-1)^{|i|}a_ib_i
+2\sum\limits_{\alpha\in S(-1)}(-1)^{|\alpha|}[e,z_\alpha^*]z_\alpha)\otimes1_\chi,
\end{split}
\end{equation*}where $\{z_\alpha^*\mid\alpha\in S(-1)\}$ and $\{z_\alpha\mid\alpha\in S(-1)\}$ are dual bases of $\ggg(-1)$ with respect to $\langle\cdot,\cdot\rangle=(e,[\cdot,\cdot])$, and $\{a_i\mid i\in I\}$ and $\{b_i\mid i\in I\}$ are dual bases of $\ggg^e(0)$ with respect to $(\cdot,\cdot)$.
\end{prop}
The proof of Proposition \ref{ge} will be given in \S\ref{3.1.3}.  Then the commutators between the generators are presented as below.
\begin{theorem}\label{maiin1}
The minimal $W$-superalgebra is generated by the Casimir element $C$ and the subspaces $\Theta_{\ggg^e(i)}$ for $i=0,1$, as described in Proposition \ref{ge}, subject to the following relations:
\begin{itemize}
\item[(1)] $[\Theta_{v_1},\Theta_{v_2}]=\Theta_{[v_1,v_2]}$ for all $v_1, v_2\in\ggg^e(0)$;
\item[(2)] $[\Theta_{v},\Theta_{w}]=\Theta_{[v,w]}$ for all $v\in\ggg^e(0)$ and $w\in\ggg^e(1)$;
\item[(3)] $[\Theta_{w_1},\Theta_{w_2}]=\frac{1}{2}([w_1,w_2],f)(C-\Theta_{\text{Cas}}-c_0)-\frac{1}{2}\sum\limits_{\alpha\in S(-1)}(\Theta_{[w_1,z_\alpha]^{\sharp}}\Theta_{[z_\alpha^*,w_2]^{\sharp}}\\ \hspace*{2.2cm} -(-1)^{|w_1||w_2|}\Theta_{[w_2,z_\alpha]^{\sharp}}\Theta_{[z_\alpha^*, w_1]^{\sharp}})$
    for all $w_1, w_2\in\ggg^e(1)$;
\item[(4)]  $[C,W_\chi']=0$.
\end{itemize}
In (3), the meaning of the notation $\sharp$ will be explained in \eqref{xh}, and the constant $c_0$ is decided by the following equation:
\begin{equation*}
\begin{split}
c_0([w_1,w_2],f)=&\frac{1}{12}\sum\limits_{\alpha,\beta\in S(-1)}(-1)^{|\alpha||w_1|+|\beta||w_1|+|\alpha||\beta|}\otimes[[z_\beta,[z_\alpha,w_1]],[z_\beta^*,[z_\alpha^*,w_2]]]\\
&-\frac{3(s-\sfr)+4}{12}([w_1,w_2],f),
\end{split}
\end{equation*}
where $s=\text{dim}\,\ggg(-1)_{\bar0}$ and $\sfr=\text{dim}\,\ggg(-1)_{\bar1}$.
\end{theorem}
In virtue of these results, Proposition \ref{minimal nilpotent} says that minimal $W$-superalgebra affords a two-sided ideal of codimension one.  As an immediate consequence, we obtain the main result  as below
\begin{theorem}\label{intromainminnimalf}
Let ${\ggg}_{\bbk}$ be a basic Lie superalgebra over ${\bbk}=\overline{\mathbb{F}}_p$, and let $\chi\in({\ggg}_{\bbk})^*_{\bar0}$ be a nilpotent $p$-character, with respect to a minimal nilpotent element $e\in(\ggg_\bbk)_\bz$. If $p\gg0$, then the reduced enveloping algebra $U_\chi({\ggg}_{\bbk})$ admits irreducible representations of dimension $p^{\frac{d_0}{2}}2^{\lfloor\frac{d_1}{2}\rfloor}$.
\end{theorem}

The proof of the above theorems will be fulfilled in \S\ref{5.2.4} and \S\ref{5.2.5} respectively.
\subsection{}
The paper is organized as follows. In \S\ref{Backgrounds}, some basics on  Lie superalgebras and finite $W$-superalgebras are recalled. In \S\ref{crw}, we first study the construction of refined reduced $W$-superalgebra $(Q_\chi^\chi)^{\text{ad}\,{\mmm}'_\mathds{k}}$ over $\bbk$, and then reformulate  the PBW theorem for refined $W$-superalgebra $W_\chi'$ over $\mathbb{C}$. In the new setting-up of refined $W$-superalgebras, we refine the conjecture \cite[Conjecture 1.3]{ZS4} in \S\ref{Refined W-superalgebras and the dimensional lower bounds}.  We first introduce Conjecture \ref{conjecture22}, which is irrelevant to the judging parity we mentioned above. Under the assumption of Conjecture \ref{conjecture22}, we show that the lower bounds of dimensions in the modular representations of basic Lie superalgebras are attainable. In the end of \S\ref{Refined W-superalgebras and the dimensional lower bounds}, we introduce a variation on the definition of refined $W$-superalgebras, which will be applied in the next section. \S\ref{structure} is the main body of the present paper. In the first part of \S\ref{structure}, we introduce the explicit expression of the generators of lower Kazhdan degree for finite $W$-superalgebras associate with arbitrary nilpotent elements over $\mathbb{C}$, and also the commutators between them. Then in the second part of \S\ref{structure}, we  completely determine the structure of minimal $W$-superalgebras. In the meantime, we complete the proof of  Proposition \ref{ge} in \S\ref{3.1.3}, Theorem \ref{maiin1} in \S\ref{5.2.4} (modulo Proposition \ref{1commutator}),  and Theorem \ref{intromainminnimalf} in \S\ref{5.2.5}. The concluding and lengthy section \S\ref{proof}  will be devoted to the proof of Proposition \ref{1commutator} leading to Theorem \ref{maiin1} (3) by lots of computation,
which is postponed there from \S\ref{structure}.
\subsection{}
Throughout we work with the field of complex numbers ${\bbc}$, or the algebraically closed field ${\bbk}=\overline{\mathbb{F}}_p$ of positive characteristic $p$ as the ground field.

Let ${\bbz}_+$ be the set of all the non-negative integers in ${\bbz}$, and denote by ${\bbz}_2$ the residue class ring modulo $2$ in ${\bbz}$. A superspace is a ${\bbz}_2$-graded vector space $V=V_{\bar0}\oplus V_{\bar1}$, in which we call elements in $V_{\bar0}$ and $V_{\bar1}$ even and odd, respectively. Write $|v|\in{\bbz}_2$ for the parity (or degree) of $v\in V$, which is implicitly assumed to be ${\bbz}_2$-homogeneous. We will use the notation $\text{dim}V=\text{dim}V_{\bar0}+\text{dim}V_{\bar1}$. All Lie superalgebras ${\ggg}$ will be assumed to be finite-dimensional.

We consider vector spaces, subalgebras, ideals, modules, and submodules $etc.$ in the super sense throughout the paper.

\section{Preliminaries}\label{Backgrounds}
In this section, we will recall some knowledge on basic classical Lie superalgebras and finite $W$-(super)algebras for use in the sequel. We refer the readers to  \cite{CW}, \cite{K} and \cite{K2} for Lie superalgebras, and \cite{P2}, \cite{P3}, \cite{P7}, \cite{W}, \cite{ZS2} and \cite{ZS4} for finite $W$-(super)algebras.

\subsection{Basic Lie superalgebras}
Following \cite[\S1]{CW}, \cite[\S2.3-\S2.4]{K}, \cite[\S1]{K2} and \cite[\S2]{WZ},  we recall the list of basic classical Lie superalgebras over $\bbf$ for $\bbf=\bbc$ or $\bbf=\bbk$.
These Lie superalgebras, with even parts being Lie algebras of reductive algebraic groups, are simple over $\bbf$ (the
general linear Lie superalgebras, though not simple, are also included), and they admit an even non-degenerate supersymmetric invariant bilinear form in the following sense.
\begin{defn}\label{form}
Let $V=V_{\bar0}\oplus V_{\bar1}$ be a $\mathbb{Z}_2$-graded space and $(\cdot,\cdot)$ be a bilinear form on $V$.
\begin{itemize}
\item[(1)] If $(a,b)=0$ for any $a\in V_{\bar0}, b\in V_{\bar1}$, then $(\cdot,\cdot)$ is called even.
\item[(2)] If $(a,b)=(-1)^{|a||b|}(b,a)$ for any homogeneous elements $a,b\in V$, then $(\cdot,\cdot)$ is called supersymmetric.
\item[(3)] If $([a,b],c)=(a,[b,c])$ for any homogeneous elements $a,b,c\in V$, then $(\cdot,\cdot)$ is called invariant.
\item[(4)] If one can conclude from $(a,V)=0$ that $a=0$, then $(\cdot,\cdot)$ is called non-degenerate.
\end{itemize}
\end{defn}

Note that when $\bbf=\bbk$ is a field of characteristic $p>0$, there are restrictions on $p$, as shown for example in \cite[Table 1]{WZ}. So we have the following list

\vskip0.3cm
\begin{center}\label{Table 1}
({\sl{Table 1}}): basic classical Lie superalgebras over $\bbk$
\vskip0.3cm
\begin{tabular}{ccc}
\hline
 $\frak{g}_\bbk$ & $\ggg_{\bar 0}$  & Restriction of $p$ when $\bbf=\bbk$\\
\hline
$\frak{gl}(m|n$) &  $\frak{gl}(m)\oplus \frak{gl}(n)$                &$p>2$            \\
$\frak{sl}(m|n)$ &  $\frak{sl}(m)\oplus \frak{sl}(n)\oplus \bbk$    & $p>2, p\nmid (m-n)$   \\
$\frak{osp}(m|n)$ & $\frak{so}(m)\oplus \frak{sp}(n)$                  & $p>2$ \\
$\text{D}(2,1,\bar a)$   & $\frak{sl}(2)\oplus \frak{sl}(2)\oplus \frak{sl}(2)$        & $p>3$   \\
$\text{F}(4)$            & $\frak{sl}(2)\oplus \frak{so}(7)$                  & $p>15$  \\
$\text{G}(3)$            & $\frak{sl}(2)\oplus \text{G}_2$                    & $p>15$     \\
\hline
\end{tabular}

\end{center}

\vskip0.3cm

Throughout the paper, we will simply call all $\ggg_\bbf$ listed above {\sl{``basic Lie superalgebras"}} for both $\bbf=\bbc$ and $\bbf=\bbk$.
\subsection{Finite $W$-superalgebras over the field of complex numbers}\label{background}
\subsubsection{}\label{0.1.1}
Let ${\ggg}$ be a basic Lie superalgebra over ${\bbc}$, and $\mathfrak{h}$ be a typical Cartan subalgebra of ${\ggg}$. Let $\Phi$ be a root system of ${\ggg}$ relative to $\mathfrak{h}$ whose simple root system $\Delta=\{\alpha_1,\cdots,\alpha_l\}$ is distinguished (cf. \cite[Proposition 1.5]{K2}). By \cite[\S3.3]{FG} we can choose a Chevalley basis $B=\{e_\gamma\mid\gamma\in\Phi\}\cup\{h_\alpha\mid\alpha\in\Delta\}$ of ${\ggg}$ excluding the case $D(2,1;a)$ with $a\notin{\bbz}$ (in the case $D(2,1;a)$ with $a\notin\bbz$ being an algebraic number, one needs to adjust the definition of Chevalley basis by changing $\bbz$ to the $\bbz$-algebra generated by $(a)$, in the range of construction constants; see \cite[\S3.1]{Gav}). Let ${\ggg}_{\bbz}$ denote the Chevalley ${\bbz}$-form in ${\ggg}$ and $U_{\bbz}$ the Kostant ${\bbz}$-form of $U({\ggg})$ associated with $B$. Given a ${\bbz}$-module $V$ and a ${\bbz}$-algebra $A$, we write $V_A:=V\otimes_{\bbz}A$.

Let $G$ be an algebraic supergroup with $\Lie(G)=\ggg$, and let $G_\ev$ be a subgroup scheme of $G$ such that $G_\ev$ is an ordinary connected reductive group with $\Lie(G_\ev)=\ggg_\bz$. Denote the corresponding super Harish-Chandra pair by $(G_\ev,\ggg)$. For a given nilpotent element $e\in{\ggg}_{\bar0}$, by Dynkin-Kostant theory one can further assume that $e$ is in $({\ggg}_{\bbz})_{\bar{0}}$ up to an $\text{Ad}\,G_{\text{ev}}$-action. Choose $f,h\in({\ggg}_\mathbb{Q})_{\bar{0}}$ such that $(e,h,f)$ is an $\mathfrak{sl}_2$-triple in ${\ggg}$. Let $(\cdot,\cdot)$ be an even nondegenerate supersymmetric invariant bilinear form, under which the Chevalley basis $B$ of ${\ggg}$ take values in $\mathbb{Q}$, and $(e,f)=1$. Define $\chi\in{\ggg}^{*}$ by letting $\chi(x)=(e,x)$ for all $x\in{\ggg}$.

A commutative ring $A$ is called {\sl admissible} if $A$ is a finitely generated ${\bbz}$-subalgebra of ${\bbc}$, $(e,f)\in A^{\times}(=A\backslash \{0\})$ and all bad primes of the root system of ${\ggg}$ and the determinant of the Gram matrix of ($\cdot,\cdot$) relative to a Chevalley basis of ${\ggg}$ are invertible in $A$. It is clear by the definition that every admissible ring is a Noetherian domain. Given a finitely generated ${\bbz}$-subalgebra $A$ of ${\bbc}$, denote by $\text{Specm}\,A$ the maximal spectrum of $A$. It is well known that for every element $\mathfrak{P}\in\text{Specm}\,A$, the residue field $A/\mathfrak{P}$ is isomorphic to $\mathbb{F}_{q}$, where $q$ is a $p$-power depending on $\mathfrak{P}$. We denote by $\Pi(A)$ the set of all primes $p\in\mathbb{N}$ that occur in this way, and the set $\Pi(A)$ contains almost all primes in $\mathbb{N}$. We denote by ${\ggg}_A$ the $A$-submodule of ${\ggg}$ generated by the Chevalley basis $B$.

Let ${\ggg}(i)=\{x\in{\ggg}\mid[h,x]=ix\}$, then ${\ggg}=\bigoplus_{i\in{\bbz}}{\ggg}(i)$. By $\mathfrak{sl}_2$-theory, all subspaces ${\ggg}(i)$ are defined over $\mathbb{Q}$. Also, $e\in{\ggg}(2)_{\bar{0}}$ and $f\in{\ggg}(-2)_{\bar{0}}$. Define a symplectic (resp. symmetric) bilinear form $\langle\cdot,\cdot\rangle$ on the ${\bbz}_2$-graded subspace ${\ggg}(-1)_{\bar{0}}$ (resp. ${\ggg}(-1)_{\bar{1}}$) by $\langle x,y\rangle:=(e,[x,y])=\chi([x,y])$ for all $x,y\in{\ggg}(-1)_{\bar0}~(\text{resp.}\,x,y\in{\ggg}(-1)_{\bar1})$.
There exist bases $\{u_1,\cdots,u_{s}\}$ of ${\ggg}(-1)_{\bar0}$ and $\{v_1,\cdots,v_\sfr\}$ of ${\ggg}(-1)_{\bar1}$ contained in ${\ggg}_\mathbb{Q}:={\ggg}_A\otimes_{A}\mathbb{Q}$ such that $\langle u_i, u_j\rangle =i^*\delta_{i+j,s+1}$ for $1\leqslant i,j\leqslant s$, where $i^*=\left\{\begin{array}{ll}-1&\text{if}~1\leqslant i\leqslant \frac{s}{2};\\ 1&\text{if}~\frac{s}{2}+1\leqslant i\leqslant s\end{array}\right.$, and $\langle v_i,v_j\rangle=\delta_{i+j,\sfr+1}$ for $1\leqslant i,j\leqslant \sfr$.
We can introduce the so-called ``$\chi$-admissible algebra" as below
\begin{align}\label{admissible alg}
\mathfrak{m}:=\bigoplus_{i\leqslant -2}{\ggg}(i)\oplus{\ggg}(-1)^{\prime}
 \end{align}
 with ${\ggg}(-1)^{\prime}={\ggg}(-1)^{\prime}_{\bar0}\oplus{\ggg}(-1)^{\prime}_{\bar1}$, where ${\ggg}(-1)^{\prime}_{\bar0}$ is the ${\bbc}$-span of $u_{{s\over 2}+1},\cdots,u_{s}$ and
${\ggg}(-1)^{\prime}_{\bar1}$ is the ${\bbc}$-span of $v_{\frac{\sfr}{2}+1},\cdots,v_\sfr$ (resp. $v_{\frac{\sfr+3}{2}},\cdots,v_\sfr$) when $\sfr:=\text{dim}\,{\ggg}(-1)_{\bar{1}}$ is even (resp. odd), then $\chi$ vanishes on the derived subalgebra of $\mathfrak{m}$. Define ${\ppp}:=\bigoplus_{i\geqslant 0}{\ggg}(i)$. We also have an extended $\chi$-admissible algebra as below
\begin{align}\label{extend admiss alg}
\mathfrak{m}^{\prime}:=\left\{\begin{array}{ll}\mathfrak{m}&\text{if}~\sfr~\text{is even;}\\
\mathfrak{m}\oplus {\bbc}v_{\frac{\sfr+1}{2}}&\text{if}~\sfr~\text{is odd.}\end{array}\right.
\end{align}
Write ${\ggg}^e$ for the centralizer of $e$ in ${\ggg}$ and denote by $d_i:=\text{dim}\,{\ggg}_i-\text{dim}\,{\ggg}^e_i$ for $i\in{\bbz}_2$, then \cite[Theorem 4.3]{WZ} shows that $\sfr$ and $d_1$ always have the same parity. This parity is a crucial factor deciding the structure of finite $W$-superalgebras (cf. \cite[Theorem 4.5]{ZS2}).
After enlarging $A$ one can assume that ${\ggg}_A=\bigoplus_{i\in{\bbz}}{\ggg}_A(i)$, and each ${\ggg}_A(i):={\ggg}_A\cap{\ggg}(i)$ is freely generated over $A$ by a basis of the vector space ${\ggg}(i)$. Then $\{u_1,\cdots,u_{s}\}$ and $\{v_1,\cdots,v_\sfr\}$ are free basis of $A$-modules ${\ggg}_A(-1)_{\bar0}$ and  ${\ggg}_A(-1)_{\bar1}$, respectively. It is obvious that
$\mathfrak{m}_A:={\ggg}_A\cap\mathfrak{m}$, $\mathfrak{m}^{\prime}_A:={\ggg}_A\cap\mathfrak{m}^{\prime}$ and ${\ppp}_A:={\ggg}_A\cap{\ppp}$ are free $A$-modules and direct summands of ${\ggg}_A$. Moreover, one can assume $e,f\in({\ggg}_A)_{\bar0}$ after enlarging $A$ possibly; $[e,{\ggg}_A(i)]$ and $[f,{\ggg}_A(i)]$ are direct summands of ${\ggg}_A(i+2)$ and ${\ggg}_A(i-2)$ respectively, and ${\ggg}_A(i+2)=[e,{\ggg}_A(i)]$ for each $i\geqslant -1$ by $\mathfrak{sl}_2$-theory.

\subsubsection{}\label{2.2.2}
Define a generalized Gelfand-Graev ${\ggg}$-module associated with $\chi$ by $$Q_\chi:=U({\ggg})\otimes_{U(\mathfrak{m})}{\bbc}_\chi,$$
where ${\bbc}_\chi={\bbc}1_\chi$ is a one-dimensional  $\mathfrak{m}$-module such that $x.1_\chi=\chi(x)1_\chi$ for all $x\in\mathfrak{m}$. The super structure of $Q_\chi$ is dependent on the parity of $\bbc_\chi$, which is indicated to be even hereafter.
Define the {\sl finite $W$-superalgebra over $\mathbb{C}$} by $$U({\ggg},e):=(\text{End}_{\ggg}Q_{\chi})^{\text{op}},$$
where $(\text{End}_{\ggg}Q_{\chi})^{\text{op}}$ denotes the opposite algebra of the endomorphism algebra of ${\ggg}$-module $Q_{\chi}$.

Let $I_\chi$ denote the ${\bbz}_2$-graded ideal in $U({\ggg})$ generated by all $x-\chi(x)$ with $x\in\mathfrak{m}$. The fixed point space $(U({\ggg})/I_\chi)^{\ad\,\mmm}$ carries a natural algebra structure given by $(x+I_\chi)\cdot(y+I_\chi):=(xy+I_\chi)$ for all $x,y\in U({\ggg})$. Then $Q_\chi\cong U({\ggg})/I_\chi$ as ${\ggg}$-modules via the ${\ggg}$-module map sending $1+I_\chi$ to $1_\chi$, and $Q_{\chi}^{\ad\,\mmm}\cong U({\ggg},e)$ as $\bbc$-algebras. Any element of $U({\ggg},e)$ is uniquely determined by its effect on the generator $1_\chi\in Q_\chi$, and the canonical isomorphism between $U({\ggg},e)$ and $Q_{\chi}^{\ad\,\mmm}$ is given by $u\mapsto u(1_\chi)$ for any $u\in U({\ggg},e)$. In what follows we will often identify $Q_\chi$ with $U({\ggg})/I_\chi$ and $U({\ggg},e)$ with $Q_{\chi}^{\ad\,\mmm}$.

Let $w_1,\cdots, w_c$ be a basis of $\ggg$ over $\bbc$. Let $U({\ggg})=\bigcup_{i\in{\bbz}}\text{F}_iU({\ggg})$ be a filtration of $U({\ggg})$, where $\text{F}_iU({\ggg})$ is the ${\bbc}$-span of all $w_1\cdots w_c$ with $w_1\in{\ggg}(j_1),\cdots,w_c\in{\ggg}(j_c)$ and $(j_1+2)+\cdots+(j_c+2)\leqslant  i$. This filtration is called {\sl Kazhdan filtration}.  The Kazhdan filtration on $Q_{\chi}$ is defined by $\text{F}_iQ_{\chi}:=\pi(\text{F}_iU({\ggg}))$ with $\pi:U({\ggg})\twoheadrightarrow U({\ggg})/I_\chi$ being the canonical homomorphism, which makes $Q_{\chi}$ into a filtered $U({\ggg})$-module. Then there is an induced Kazhdan filtration
$\text{F}_i U({\ggg},e)$ on the subspace $U({\ggg},e)=Q_{\chi}^{\ad\,\mmm}$ of $Q_{\chi}$ such that $\text{F}_j U({\ggg},e)=0$
unless $j\geqslant0$.

Choose a basis $x_1,\cdots,x_l,x_{l+1},\cdots,x_m\in({\ppp}_A)
_{\bar{0}}, y_1,\cdots, y_q, y_{q+1}, \cdots,y_n\in({\ppp}_A)_{\bar{1}}$ of the free $A$-module ${\ppp}_A=\bigoplus_{i\geqslant 0}{\ggg}_A(i)$ such that

(a) $x_i\in{\ggg}_A(k_i)_{\bar{0}}, y_j\in{\ggg}_A(k'_j)_{\bar{1}}$, where $k_i,k'_j\in{\bbz}_+$ with $1\leqslant i\leqslant m$ and $1\leqslant j\leqslant n$;

(b) $x_1,\cdots,x_l$ is a basis of $({\ggg}_A)^e_{\bar{0}}$ and $y_1,\cdots,y_q$ is a basis of $({\ggg}_A)^e_{\bar{1}}$;

(c) $x_{l+1},\cdots,x_m\in[f,({\ggg}_A)_{\bar{0}}]$ and $ y_{q+1},\cdots,y_n\in[f,({\ggg}_A)_{\bar{1}}]$.

For $k\in\mathbb{Z}_+$, define
\begin{equation*}
\begin{split}
\mathbb{Z}_+^k:=&\{(i_1,\cdots,i_k)\mid i_j\in\mathbb{Z}_+\},\\
\Lambda_k:=&\{(i_1,\cdots,i_k)\mid i_j\in{\bbz}_+,~0\leqslant  i_j\leqslant  p-1\},\\
\Lambda'_k:=&\{(i_1,\cdots,i_k)\mid i_j\in\{0,1\}\}
\end{split}
\end{equation*}
with $1\leqslant j\leqslant k$. For $\mathbf{i}=(i_1,\cdots,i_k)$ in $\mathbb{Z}_+^k$, $\Lambda_k$ or $\Lambda'_k$, set $|\mathbf{i}|=i_1+\cdots+i_k$. For any real number $a\in\mathbb{R}$, let $\lceil a\rceil$ denote the largest integer lower bound of $a$, and $\lfloor a\rfloor$ the least integer upper bound of $a$. Given $(\mathbf{a},\mathbf{b},\mathbf{c},\mathbf{d})\in{\bbz}^m_+\times\Lambda'_n\times{\bbz}^{\frac{s}{2}}_+\times\Lambda'_{\lfloor\frac{\sfr}{2}\rfloor}$, let $x^\mathbf{a}y^\mathbf{b}u^\mathbf{c}v^\mathbf{d}$ denote the monomial $x_1^{a_1}\cdots x_m^{a_m}y_1^{b_1}\cdots y_n^{b_n}u_1^{c_1}\cdots u_{\frac{s}{2}}^{c_{\frac{s}{2}}}v_1^{d_1}\cdots v_{\lfloor\frac{\sfr}{2}\rfloor}^{d_{\lfloor\frac{\sfr}{2}\rfloor}}$ in $U({\ggg})$. Set $Q_{\chi,A}:=U({\ggg}_A)\otimes_{U(\mathfrak{m}_A)}A_\chi$, where $A_\chi=A1_\chi$. It is obvious that $Q_{\chi,A}$ is a ${\ggg}_A$-stable $A$-lattice in $Q_{\chi}$ with $\{x^\mathbf{a}y^\mathbf{b}u^\mathbf{c}v^\mathbf{d}\otimes1_\chi\mid(\mathbf{a},\mathbf{b},\mathbf{c},\mathbf{d})\in{\bbz}^m_+\times\Lambda'_n\times{\bbz}^{\frac{s}{2}}_+\times\Lambda'_{\lfloor\frac{\sfr}{2}\rfloor}\}$
being a free basis. Given $(\mathbf{a},\mathbf{b},\mathbf{c},\mathbf{d})\in{\bbz}_+^m\times\Lambda'_n\times{\bbz}_+^{\frac{s}{2}}\times\Lambda'_{\lfloor\frac{\sfr}{2}\rfloor}$, set
\begin{equation*}
\begin{split}
|(\mathbf{a},\mathbf{b},\mathbf{c},\mathbf{d})|_e:=&\sum_{i=1}^ma_i(k_i+2)+\sum_{i=1}^nb_i(k'_i+2)+\sum_{i=1}^{\frac{s}{2}}c_i+\sum_{i=1}^{\lfloor\frac{\sfr}{2}\rfloor}d_i,\\
\text{wt}(\bar x^{\mathbf{a}}\bar y^\mathbf{b}\bar u^\mathbf{c}\bar v^\mathbf{d}):=&(\sum\limits_{i=1}^mk_ia_i)+(\sum\limits_{i=1}^nk'_ib_i)-|\mathbf{c}|-|\mathbf{d}|,
\end{split}
\end{equation*} which are called the $e$-degree and the weight of $x^{\mathbf{a}}y^\mathbf{b}u^\mathbf{c}v^\mathbf{d}$.

For any non-zero element $h\in Q_{\chi}^{\ad\,\mmm}$, write$$h=(\sum\limits_{|(\mathbf{a},\mathbf{b},\mathbf{c},\mathbf{d})|_e\leqslant  n(h)}\lambda_{\mathbf{a},\mathbf{b},\mathbf{c},\mathbf{d}}
x^{\mathbf{a}}y^\mathbf{b}u^\mathbf{c}v^\mathbf{d})\otimes1_\chi,$$ where $n(h)$ is the highest $e$-degree of the terms in the linear expansion of $h$, and $\lambda_{\mathbf{a},\mathbf{b},\mathbf{c},\mathbf{d}}\neq0$ for at least one $(\mathbf{a},\mathbf{b},\mathbf{c},\mathbf{d})$ with $|(\mathbf{a},\mathbf{b},\mathbf{c},\mathbf{d})|_e=n(h)$.

For $k\in\mathbb{Z}_+$, put $\Lambda^{k}_{h}=\{(\mathbf{a},\mathbf{b},\mathbf{c},\mathbf{d})\mid\lambda_{\mathbf{a},\mathbf{b},\mathbf{c},\mathbf{d}}\neq0\mbox{ and }|(\mathbf{a},\mathbf{b},\mathbf{c},\mathbf{d})|_e=k\}$ and set
\begin{align}\label{MaxLambda}
\Lambda^{\text{max}}_{h}:=
\{ (\mathbf{a},\mathbf{b},\mathbf{c},\mathbf{d})\in\Lambda^{n( h)}_{h}\mid
    \text{wt}(x^{\mathbf{a}}y^\mathbf{b} u^\mathbf{c}v^\mathbf{d}) \mbox{ takes its maximum value}\}.
    \end{align}
This maximum value mentioned in (\ref{MaxLambda}) will be denoted by $N(h)$.

\subsection{Finite $W$-superalgebras in positive characteristic}\label{2.3}
\subsubsection{}\label{2.3.1}
Pick a prime $p\in\Pi(A)$ and denote by ${\bbk}=\overline{\mathbb{F}}_p$ the algebraic closure of $\mathbb{F}_p$. Since the bilinear form $(\cdot,\cdot)$ is $A$-valued on ${\ggg}_A$, it induces a bilinear form on the Lie superalgebra ${\ggg}_{\bbk}\cong{\ggg}_A\otimes_A{\bbk}$. For $x\in{\ggg}_A$, set $\bar{x}:=x\otimes1$, an element of ${\ggg}_{\bbk}$. To simplify notation we identify $e,f,h$ with the nilpotent elements $\bar{e}=e\otimes1,~\bar{f}=f\otimes1$ and $\bar{h}=h\otimes1$ in ${\ggg}_{\bbk}$, and $\chi$ with the linear function $(\bar e,\cdot)$ on ${\ggg}_{\bbk}$.

The Lie superalgebra ${\ggg}_{\bbk}$ carries a natural $p$-mapping $x\mapsto x^{[p]}$ for all $x\in({\ggg}_{\bbk})_{\bar0}$. For any $\xi\in({\ggg}_{\bbk})_{\bar{0}}^*$ we denote by $J_\xi$ the two-sided ideal of $U({\ggg}_{\bbk})$ generated by the even central elements $\{\bar x^p-\bar x^{[p]}-\xi(\bar x)^p\mid\bar x\in({\ggg}_{\bbk})_{\bar{0}}\}$. Then the quotient algebra $U_\xi({\ggg}_{\bbk}):=U({\ggg}_{\bbk})/J_\xi$ is called the reduced enveloping algebra with $p$-character $\xi$. We often regard $\xi\in{\ggg}_{\bbk}^*$ by letting $\xi(({\ggg}_{\bbk})_{\bar{1}})=0$.
\subsubsection{}\label{2.3.2}
For $i\in{\bbz}$, set ${\ggg}_{\bbk}(i):={\ggg}_A(i)\otimes_A{\bbk}$ and put $\mathfrak{m}_{\bbk}:=\mathfrak{m}_A\otimes_A{\bbk}$, then $\mathfrak{m}_{\bbk}$ is a restricted subalgebra of $\mathfrak{g}_{\bbk}$. Denote by $\mathfrak{m}'_{\bbk}:=\mathfrak{m}'_A\otimes_A{\bbk}$ and $\mathfrak{p}_{\bbk}:=\mathfrak{p}_A\otimes_A{\bbk}$. Due to our assumptions on $A$, the elements $\bar{x}_1,\cdots,\bar{x}_l$ and $\bar{y}_1,\cdots,\bar{y}_q$ form  bases of the centralizer $({\ggg}^e_{\bbk})_{\bar{0}}$ and $({\ggg}^e_{\bbk})_{\bar{1}}$ of $e$ in ${\ggg}_{\bbk}$, respectively. \cite[\S4.1]{WZ} showed that the subalgebra $\mathfrak{m}_{\bbk}$ is $p$-nilpotent, and the linear function $\chi$ vanishes on the $p$-closure of $[\mathfrak{m}_{\bbk},\mathfrak{m}_{\bbk}]$. Set $Q_{\chi,{\bbk}}:=U({\ggg}_{\bbk})\otimes_{U(\mathfrak{m}_{\bbk})}{\bbk}_\chi$, where ${\bbk}_\chi=A_\chi\otimes_{A}{\bbk}={\bbk}1_\chi$. Clearly, ${\bbk}1_\chi$ is a one-dimensional  $\mathfrak{m}_{\bbk}$-module with the property $\bar x.1_\chi=\chi(\bar x)1_\chi$ for all $\bar x\in\mathfrak{m}_{\bbk}$, and it is obvious that $Q_{\chi,{\bbk}}\cong Q_{\chi,A}\otimes_A{\bbk}$ as ${\ggg}_{\bbk}$-modules.

Set the ${\ggg}_{\bbk}$-module $Q_{\chi}^\chi:=Q_{\chi,{\bbk}}/J_\chi Q_{\chi,{\bbk}}$. Then we can define the {\sl reduced $W$-superalgebra} by $U_\chi({\ggg}_{\bbk},e):=(\text{End}_{{\ggg}_{\bbk}}Q_{\chi}^\chi)^{\text{op}}$. It follows from \cite[Proposition 2.21]{ZS2} that $U_\chi({\ggg}_{\bbk},e)\cong(Q_\chi^\chi)^{\text{ad}\,{\mmm}_{\bbk}}$ as ${\bbk}$-algebras.

Given $(\mathbf{a},\mathbf{b},\mathbf{c},\mathbf{d})\in\Lambda_m\times\Lambda'_n\times\Lambda_{\frac{s}{2}}\times\Lambda'_{\lfloor\frac{\sfr}{2}\rfloor}$, let $\bar x^\mathbf{a}\bar y^\mathbf{b}\bar u^\mathbf{c}\bar v^\mathbf{d}$ denote the monomial $\bar x_1^{a_1}\cdots\bar x_m^{a_m}\bar y_1^{b_1}\cdots\bar y_n^{b_n}\bar u_1^{c_1}\cdots\bar u_{\frac{s}{2}}^{c_{\frac{s}{2}}}\bar v_1^{d_1}\cdots\bar v_{\lfloor\frac{\sfr}{2}\rfloor}^{d_{\lfloor\frac{\sfr}{2}\rfloor}}$ in $U({\ggg}_{\bbk})$. Denote by  $|(\mathbf{a},\mathbf{b},\mathbf{c},\mathbf{d})|_e$ and $\text{wt}(\bar x^{\mathbf{a}}\bar y^\mathbf{b}\bar u^\mathbf{c}\bar v^\mathbf{d})$ be the $e$-degree and the weight of $\bar x^{\mathbf{a}}\bar y^\mathbf{b}\bar u^\mathbf{c}\bar v^\mathbf{d}$ as defined in \S\ref{2.2.2}, respectively.
For any non-zero element $\bar h\in (Q_\chi^\chi)^{\text{ad}\,{\mmm}_{\bbk}}$ we let $n(\bar h)$, $\Lambda^{k}_{\bar h}$, $\Lambda_{\bar h}^{\text{max}}$ and $N(\bar h)$ have the same meaning as in \S\ref{2.2.2}.

\section{The structure theory of refined $W$-superalgebras}\label{crw}
In this section we will first introduce the definition of refined reduced $W$-superalgebras associated with a
basic Lie superalgebra $\ggg$ over $\bbk$, an algebraically closed field of prime characteristic $p$; and also refined $W$-superalgebras over $\mathbb{C}$. Then the structure theory of these algebras is studied. We mainly follow Premet's strategy on finite $W$-algebras \cite[\S3-\S4]{P2}, and also the method applied by Zeng-Shu on finite $W$-superalgebras  \cite[\S3-\S4]{ZS2}, with a few modifications.
\subsection{The structure theory of refined reduced $W$-superalgebras over $\bbk$}\label{rrw}
This subsection is devoted to the refined reduced $W$-superalgebras over $\bbk$.
\begin{defn}\label{rewcc} Define the refined reduced $W$-superalgebra over $\bbk$ by
 $$(Q_\chi^{\chi})^{\text{ad}\,{\mmm}'_\mathds{k}}:=(Q_{\chi,{\bbk}}/J_\chi Q_{\chi,{\bbk}})^{\text{ad}\,{\mmm}'_\mathds{k}}
\equiv\{\bar{y}\in Q_{\chi,{\bbk}}/J_\chi Q_{\chi,{\bbk}} \mid [a,y]\in J_\chi Q_{\chi,{\bbk}}, \forall a\in{\mmm}'_\mathds{k}\},$$
and $\bar{y}_1\cdot\bar{y}_2:=\overline{y_1y_2}$ for all $\bar{y}_1,\bar{y}_2\in (Q_\chi^{\chi})^{\text{ad}\,{\mmm}'_\mathds{k}}$.
\end{defn}

Retain the notations as in \S\ref{2.3}.
Now we will discuss the structure of refined reduced $W$-superalgebra over $\bbk$. First note that
\begin{lemma} \label{hw}
Let $\bar h\in (Q_\chi^{\chi})^{\text{ad}\,{\mmm}'_\mathds{k}}\backslash\{0\}$ and $(\mathbf{a},\mathbf{b},\mathbf{c},\mathbf{d})\in \Lambda^{\text{max}}_{\bar h}$. Then $\mathbf{a}\in\Lambda_{l}\times\{\mathbf{0}\}, \mathbf{b}\in\Lambda'_{q}\times\{\mathbf{0}\}$, $\mathbf{c}=\mathbf{0}$  and  $\mathbf{d}=\mathbf{0}$.
\end{lemma}
\begin{proof}
Since $(Q_\chi^\chi)^{\text{ad}\,{\mmm}'_\mathds{k}}$ is a subalgebra of the reduced $W$-superalgebra $U_\chi(\ggg_\mathds{k},e)\cong(Q_\chi^\chi)^{\text{ad}\,{\mmm}_\mathds{k}}$, it follows from \cite[Lemma 3.3]{ZS2} that $\mathbf{a}\in\Lambda_{l}\times\{\mathbf{0}\},~\mathbf{b}\in\Lambda'_{q}\times\{\mathbf{0}\}$ and $\mathbf{c}=\mathbf{0}$. Moreover, the sequence $\mathbf{d}$ satisfies:

(1) $\mathbf{d}=\mathbf{0}$ when $\sfr=\dim\ggg_{\bbk}(-1)_{\bar{1}}$ is even;

(2) $\mathbf{d}\in\{\mathbf{0}\}_{\frac{\sfr-1}{2}}\times\Lambda'_{1}$ when $\sfr=\dim\ggg_{\bbk}(-1)_{\bar{1}}$ is odd.

For the case when $\sfr$ is even, we have ${\mmm}'_\mathds{k}={\mmm}_\mathds{k}$ by definition, then $(Q_\chi^\chi)^{\text{ad}\,{\mmm}'_\mathds{k}}=(Q_\chi^\chi)^{\text{ad}\,{\mmm}_\mathds{k}}\cong U_\chi(\ggg_\mathds{k},e)$,  and the lemma readily follows from (1). So we just need to consider the case when $\sfr$ is odd. Our arguments mainly follow the proof of \cite[Lemma 3.3]{ZS2}, with some modifications. Now we proceed by steps.

Step 1: Denote by $\mathbf{e}_{\frac{\sfr+1}{2}}:=(\{\mathbf{0}\}_{\frac{\sfr-1}{2}},1)$.
We claim that $(\Lambda_{l}\times\{\mathbf{0}\},\Lambda'_{q}\times\{\mathbf{0}\},\mathbf{0},\mathbf{e}_{\frac{\sfr+1}{2}})\notin\Lambda^{\text{max}}_{\bar h}$ for any $\bar h\in(Q_\chi^\chi)^{\text{ad}\,{\mmm}'_\mathds{k}}$.
Suppose the contrary. Let $(\mathbf{a},\mathbf{b},\mathbf{c},\mathbf{d})\in\Lambda^{d'}_{\bar h}$ where $d'\in\mathbb{Z}_+$.  It follows from \cite[Lemma 3.1]{ZS2} and \cite[(3.7)]{ZS2} that
\begin{equation}\label{rhow}
(\bar v_{\frac{\sfr+1}{2}}\otimes\bar1_\chi)\cdot(\bar x^{\mathbf{a}}\bar y^\mathbf{b}\bar u^\mathbf{c}\bar v^\mathbf{d}\otimes\bar 1_\chi)=\sum_{\mathbf{i}\in\Lambda_{m}}\sum_{\mathbf{j}\in\Lambda'_{n}}\left(\begin{array}{@{\hspace{0pt}}c@{\hspace{0pt}}} \mathbf{a}\\ \mathbf{i}\end{array}\right)\bar x^{\mathbf{a}-\mathbf{i}}\bar y^{\mathbf{b}-\mathbf{j}}\cdot[\bar v_{\frac{\sfr+1}{2}}\bar x^{\mathbf{i}}\bar y^{\mathbf{j}}]\cdot \bar u^\mathbf{c}\bar v^\mathbf{d}\otimes\bar 1_\chi,
\end{equation}
where $\mathbf{a}\choose\mathbf{i}$$=\prod\limits_{l'=1}^m$$a_{l'}\choose i_{l'}$ and $$[\bar v_{\frac{\sfr+1}{2}}\bar x^{\mathbf{i}}\bar y^{\mathbf{j}}]=k_{1,b_1,j_1}\cdots k_{n,b_n,j_n}(-1)^{|\mathbf{i}|}(\text{ad}\,\bar y_n)^{j_n}\cdots(\text{ad}\,\bar y_1)^{j_1}(\text{ad}\,\bar x_m)^{i_m}\cdots(\text{ad}\,\bar x_1)^{i_1}(\bar v_{\frac{\sfr+1}{2}}),$$
in which the coefficients $k_{1,b_1,j_1},\cdots,k_{n,b_n,j_n}\in {\bbk}$ (note that $b_1,\cdots,b_n,j_1,\cdots,j_n\in\{0,1\}$) are defined by
$$k_{t',0,0}=1, k_{t',0,1}=0, k_{t',1,0}=(-1)^{1+j_1+\cdots+j_{{t'}-1}}, k_{t',1,1}=(-1)^{j_1+\cdots+j_{{t'}-1}}$$
with $1\leqslant {t'}\leqslant n$ ($j_0$ is interpreted as $0$). Moreover, the summation on the right side of  \eqref{rhow} runs through all $(\mathbf{i},\mathbf{j})\in\Lambda_m\times\Lambda'_n$ such that $[\bar v_{\frac{\sfr+1}{2}}\bar x^{\mathbf{i}}\bar y^{\mathbf{j}}]$ is nonzero and $\text{wt}([\bar v_{\frac{\sfr+1}{2}}\bar x^{\mathbf{i}}\bar y^{\mathbf{j}}])\geqslant -1$.

Step 2: First note that
$[\bar v_{\frac{\sfr+1}{2}},\bar v_i]\otimes\bar1_\chi=1\otimes\chi([\bar v_{\frac{\sfr+1}{2}},\bar v_i])\bar1_\chi=0$ for all $1\leqslant  i\leqslant \frac{\sfr-1}{2}$.
We continue the arguments case by case, according to the values of $\text{wt}([\bar v_{\frac{\sfr+1}{2}}\bar x^{\mathbf{i}}\bar y^{\mathbf{j}}])$ with respect to  summation parameters  $(\mathbf{i},\mathbf{j})\in \Lambda_m\times\Lambda'_n$.

(Case 1)  wt$([\bar v_{\frac{\sfr+1}{2}}\bar x^{\mathbf{i}}\bar y^{\mathbf{j}}])\geqslant 0$.
   We have $|\mathbf{i}|+|\mathbf{j}|\geqslant 1$. By Step 2(Case 1) in the proof of \cite[Lemma 3.3]{ZS2} we know that $\bar x^{\mathbf{a}-\mathbf{i}}\bar y^{\mathbf{b}-\mathbf{j}}\cdot[\bar v_{\frac{\sfr+1}{2}}\bar x^{\mathbf{i}}\bar y^{\mathbf{j}}]\cdot\bar u^\mathbf{c}\bar v^\mathbf{d}\otimes\bar 1_\chi$ is a linear combination of $\bar x^{\mathbf{i'}}\bar y^{\mathbf{j'}}\bar u^\mathbf{c}\bar v^\mathbf{d}\otimes\bar 1_\chi$ with
$$\text{wt}(\bar x^{\mathbf{i'}}\bar y^{\mathbf{j'}}\bar u^\mathbf{c}\bar v^\mathbf{d})=-1+\text{wt}(\bar x^{\mathbf{a}}\bar y^{\mathbf{b}}\bar u^\mathbf{c}\bar v^\mathbf{d}),$$
and \[
\text{deg}_e(\bar x^{\mathbf{i'}}\bar y^{\mathbf{j'}}\bar u^\mathbf{c}\bar v^\mathbf{d})
\leqslant1+d'-2(|\mathbf{i}|+|\mathbf{j}|).
\]

(Case 2) wt$([\bar v_{\frac{\sfr+1}{2}}\bar x^{\mathbf{i}}\bar y^{\mathbf{j}}])=-1$. By Step 2(Case 2) in the proof of \cite[Lemma 3.3]{ZS2}, the vector $\bar x^{\mathbf{a}-\mathbf{i}}\bar y^{\mathbf{b}-\mathbf{j}}\cdot[\bar v_{\frac{\sfr+1}{2}}\bar x^{\mathbf{i}}\bar y^{\mathbf{j}}]\cdot \bar u^\mathbf{c}\bar v^\mathbf{d}\otimes\bar 1_\chi$ is a linear combination of $\bar x^{\mathbf{a}-\mathbf{i}}\bar y^{\mathbf{b}-\mathbf{j}}\bar u^\mathbf{i'}\bar v^\mathbf{j'}\otimes\bar 1_\chi$ with $|\mathbf{i'}|=|\mathbf{c}|\pm1, \mathbf{j'}=\mathbf{d}$, or $\mathbf{i'}=\mathbf{c}, |\mathbf{j'}|=|\mathbf{d}|\pm1$.

(a) If $|\mathbf{i'}|=|\mathbf{c}|+1, \mathbf{j'}=\mathbf{d}$, or $\mathbf{i'}=\mathbf{c}, |\mathbf{j'}|=|\mathbf{d}|+1$, then $|\mathbf{i}|+|\mathbf{j}|\geqslant 1$,
$$\text{wt}(\bar x^{\mathbf{a}-\mathbf{i}}\bar y^{\mathbf{b}-\mathbf{j}}\bar u^\mathbf{i'}\bar v^\mathbf{j'})=-1+\text{wt}(\bar x^{\mathbf{a}}\bar y^{\mathbf{b}}\bar u^\mathbf{c}\bar v^\mathbf{d}),$$
and\[
\text{deg}_e(\bar x^{\mathbf{a}-\mathbf{i}}\bar y^{\mathbf{b}-\mathbf{j}}\bar u^\mathbf{i'}\bar v^\mathbf{j'})=1+d'-2(|\mathbf{i}|+|\mathbf{j}|).
\]

(b) If $|\mathbf{i'}|=|\mathbf{c}|-1, \mathbf{j'}=\mathbf{d}$, or $\mathbf{i'}=\mathbf{c}, |\mathbf{j'}|=|\mathbf{d}|-1$, then
$$\text{wt}(\bar x^{\mathbf{a}-\mathbf{i}}\bar y^{\mathbf{b}-\mathbf{j}}\bar u^\mathbf{i'}\bar v^\mathbf{j'})
=1+\text{wt}(\bar x^{\mathbf{a}}\bar y^{\mathbf{b}}\bar u^\mathbf{c}\bar v^\mathbf{d}),$$
and\[\text{deg}_e(\bar x^{\mathbf{a}-\mathbf{i}}\bar y^{\mathbf{b}-\mathbf{j}}\bar u^\mathbf{i'}\bar v^\mathbf{j'})=-1+d'-2(|\mathbf{i}|+|\mathbf{j}|).
\]

For concluding  our arguments, we adopt an auxiliary endomorphism.
For $i,j\in\mathbb{Z}$, take $\pi_{ij}$ to be an endomorphism of $Q_\chi^\chi$ defined  via
\begin{align}\label{pi}
\pi_{ij}(\bar x^{\mathbf{a}}\bar y^{\mathbf{b}}\bar u^\mathbf{c}\bar v^\mathbf{d}\otimes\bar1_\chi)=\begin{cases}\bar x^{\mathbf{a}}\bar y^{\mathbf{b}}\bar u^\mathbf{c}\bar v^\mathbf{d}\otimes\bar 1_\chi &\mbox{ if }\text{deg}_e(\bar x^{\mathbf{a}}\bar y^{\mathbf{b}}\bar u^\mathbf{c}\bar v^\mathbf{d})=i\\ &\mbox{ and } \text{wt}(\bar x^{\mathbf{a}}\bar y^{\mathbf{b}}\bar u^\mathbf{c}\bar v^\mathbf{d})=j;\\0 &\mbox{ otherwise}.
\end{cases}
\end{align}

 Step 3:
Now we proceed to complete the arguments by reducing contradictions in Step 1. Since $\bar h\in(Q_\chi^\chi)^{\text{ad}\,{\mmm}'_\mathds{k}}$ and $\bar v_{\frac{\sfr+1}{2}}\in{\mmm}'_\mathds{k}$, then we have
\begin{equation}\label{mh}
[\bar v_{\frac{\sfr+1}{2}}\otimes\bar1_\chi,\bar h]=0
\end{equation}
by the definition. On the other hand, by \cite[(3.12)]{ZS2} we have
\begin{equation}\label{exten}
\begin{split}
&[\bar v_{\frac{\sfr+1}{2}}\otimes\bar1_\chi,\bar h]=(\bar v_{\frac{\sfr+1}{2}}\otimes\bar 1_\chi)\cdot\bar h-(-1)^{|\bar h|}\bar h\cdot(\bar v_{\frac{\sfr+1}{2}}\otimes\bar 1_\chi)\\
=&(\sum\limits_{(\mathbf{a},\mathbf{b},\mathbf{c},\mathbf{d})
\in\Lambda_{\bar h}^{n(\bar h)}}\lambda_{\mathbf{a},\mathbf{b},\mathbf{c},\mathbf{d}}(
\sum\limits_{i=1}^m(-1)^{\sum\limits_{j=1}^nb_j}\bar{a}_i\bar x^{\mathbf{a}-\mathbf{e}_i}\bar y^{\mathbf{b}}\cdot[\bar v_{\frac{\sfr+1}{2}},\bar x_i]\cdot \bar u^\mathbf{c}\bar v^\mathbf{d}\\
&+\sum\limits_{i=1}^n(-1)^{1+\sum\limits_{j=1}^ib_j}\bar x^{\mathbf{a}}\bar y^{\mathbf{b}-\mathbf{e}_i}\cdot[\bar v_{\frac{\sfr+1}{2}},\bar y_i]\cdot \bar u^\mathbf{c}\bar v^\mathbf{d})\\
&+\sum\limits_{(\mathbf{a},\mathbf{b},\mathbf{c},d_1,\cdots,d_{\frac{\sfr-1}{2}},0)
\in\Lambda_{\bar h}^{n(\bar h)}}\lambda_{\mathbf{a},\mathbf{b},\mathbf{c},d_1,\cdots,d_{\frac{\sfr-1}{2}},0}(-1)^{\sum\limits_{j=1}^nb_j+\sum\limits_{j=1}^{\frac{r-1}{2}}d_{j}} \bar x^{\mathbf{a}}\bar y^{\mathbf{b}}\cdot\bar u^{\mathbf{c}}\bar v_1^{d_1}\cdots\bar v_{\frac{\sfr-1}{2}}^{d_{\frac{\sfr-1}{2}}}\cdot v_{\frac{\sfr+1}{2}}\\
&+\sum\limits_{(\mathbf{a},\mathbf{b},\mathbf{c},d_1,\cdots,d_{\frac{\sfr-1}{2}},1)
\in\Lambda_{\bar h}^{n(\bar h)}}\lambda_{\mathbf{a},\mathbf{b},\mathbf{c},d_1,\cdots,d_{\frac{\sfr-1}{2}},1}
(-1)^{\sum\limits_{j=1}^nb_j+\sum\limits_{j=1}^{\frac{r-1}{2}}d_{j}} \bar x^{\mathbf{a}}\bar y^{\mathbf{b}}\cdot\bar u^{\mathbf{c}}\bar v_1^{d_1}\cdots\bar v_{\frac{\sfr-1}{2}}^{d_{\frac{\sfr-1}{2}}}v_{\frac{\sfr+1}{2}}\cdot v_{\frac{\sfr+1}{2}}\\&+
\sum\limits_{|(\mathbf{i},
\mathbf{j},\mathbf{k},\mathbf{l})|_e\leqslant  n(\bar h)-2}\beta_{\mathbf{i},\mathbf{j},\mathbf{k},\mathbf{l}}\cdot\bar x^{\mathbf{i}}\bar y^{\mathbf{j}}\bar u^\mathbf{k}\bar v^\mathbf{l})\otimes\bar1_\chi\\
&-(\sum\limits_{(\mathbf{a},\mathbf{b},\mathbf{c},d_1,\cdots,d_{\frac{\sfr-1}{2}},0)
\in\Lambda_{\bar h}^{n(\bar h)}}\lambda_{\mathbf{a},\mathbf{b},\mathbf{c},d_1,\cdots,d_{\frac{\sfr-1}{2}},0}
(-1)^{\sum\limits_{j=1}^nb_j+\sum\limits_{j=1}^{\frac{r-1}{2}}d_{j}}
\bar x^{\mathbf{a}}\bar y^{\mathbf{b}}\cdot\bar u^{\mathbf{c}}\bar v_1^{d_1}\cdots\bar v_{\frac{\sfr-1}{2}}^{d_{\frac{\sfr-1}{2}}}\cdot v_{\frac{\sfr+1}{2}}\\
&+\sum\limits_{(\mathbf{a},\mathbf{b},\mathbf{c},d_1,\cdots,d_{\frac{\sfr-1}{2}},1)
\in\Lambda_{\bar h}^{n(\bar h)}}\lambda_{\mathbf{a},\mathbf{b},\mathbf{c},d_1,\cdots,d_{\frac{\sfr-1}{2}},1}
(-1)^{\sum\limits_{j=1}^nb_j+\sum\limits_{j=1}^{\frac{r-1}{2}}d_{j}+1}
\bar x^{\mathbf{a}}\bar y^{\mathbf{b}}\cdot\bar u^{\mathbf{c}}\bar v_1^{d_1}\cdots\bar v_{\frac{\sfr-1}{2}}^{d_{\frac{\sfr-1}{2}}}v_{\frac{\sfr+1}{2}}\cdot
\\& v_{\frac{\sfr+1}{2}}+\sum\limits_{|(\mathbf{i},
\mathbf{j},\mathbf{k},\mathbf{l})|_e\leqslant  n(\bar h)-2}\beta_{\mathbf{i},\mathbf{j},\mathbf{k},\mathbf{l}}\cdot\bar x^{\mathbf{i}}\bar y^{\mathbf{j}}\bar u^\mathbf{k}\bar v^\mathbf{l})\otimes\bar1_\chi,
\end{split}
\end{equation}
where at least one $\lambda_{\mathbf{a},\mathbf{b},\mathbf{c},d_1,\cdots,d_{\frac{\sfr-1}{2}},1}\neq0$ by our assumption in Step 1.

It is obvious that
\begin{equation*}
2\bar v_{\frac{\sfr+1}{2}}^2\otimes\bar1_\chi=[\bar v_{\frac{\sfr+1}{2}},\bar v_{\frac{\sfr+1}{2}}]\otimes
\bar1_\chi=1\otimes\chi([\bar v_{\frac{\sfr+1}{2}},\bar v_{\frac{\sfr+1}{2}}])\bar1_\chi=1\otimes\bar1_\chi,
\end{equation*}thus \eqref{exten} equals
\begin{equation}\label{exten2}
\begin{split}
[\bar v_{\frac{\sfr+1}{2}}\otimes\bar1_\chi,\bar h]
=&(\sum\limits_{(\mathbf{a},\mathbf{b},\mathbf{c},\mathbf{d})
\in\Lambda_{\bar h}^{n(\bar h)}}\lambda_{\mathbf{a},\mathbf{b},\mathbf{c},\mathbf{d}}(
\sum\limits_{i=1}^m(-1)^{\sum\limits_{j=1}^nb_j}\bar{a}_i\bar x^{\mathbf{a}-\mathbf{e}_i}\bar y^{\mathbf{b}}\cdot[\bar v_{\frac{\sfr+1}{2}},\bar x_i]\\
&\cdot \bar u^\mathbf{c}\bar v^\mathbf{d}+\sum\limits_{i=1}^n(-1)^{1+\sum\limits_{j=1}^ib_j}\bar x^{\mathbf{a}}\bar y^{\mathbf{b}-\mathbf{e}_i}\cdot[\bar v_{\frac{\sfr+1}{2}},\bar y_i]\cdot \bar u^\mathbf{c}\bar v^\mathbf{d}
\\&+\sum\limits_{|(\mathbf{a},
\mathbf{b},\mathbf{c},\mathbf{d})|_e=n(\bar h)-1}\lambda_{\mathbf{a},\mathbf{b},\mathbf{c},\mathbf{d}}
(-1)^{\sum\limits_{j=1}^nb_j+\sum\limits_{j=1}^{\frac{r-1}{2}}d_{j}} \bar x^{\mathbf{a}}\bar y^{\mathbf{b}}\cdot\bar u^{\mathbf{c}}\bar v_1^{d_1}\cdots\bar v_{\frac{\sfr-1}{2}}^{d_{\frac{\sfr-1}{2}}}\\
&+\sum\limits_{|(\mathbf{i},
\mathbf{j},\mathbf{k},\mathbf{l})|_e\leqslant  n(\bar h)-2}\beta_{\mathbf{i},\mathbf{j},\mathbf{k},\mathbf{l}}\cdot\bar x^{\mathbf{i}}\bar y^{\mathbf{j}}\bar u^\mathbf{k}\bar v^\mathbf{l})\otimes\bar1_\chi.
\end{split}
\end{equation}

When $[\bar v_{\frac{\sfr+1}{2}}\otimes\bar1_\chi,\bar h]$ is written as a linear combination of the canonical basis of $Q_\chi^\chi$, it is immediate from \eqref{exten2} and the arguments in Step 2 that the terms with $e$-degree $n(\bar h)-1$ and weight $N(\bar h)+1$ only occur as in (Case 2)(b) with wt$([\bar  v_{\frac{\sfr+1}{2}}\bar x^{\mathbf{i}}\bar y^{\mathbf{j}}])=-1$. By \eqref{exten2} we have
\[\begin{array}{ll}
&\pi_{n(\bar h)-1,N(\bar h)+1}([\bar v_{\frac{\sfr+1}{2}}\otimes\bar1_\chi,\bar h])\\
=&\sum\limits_{|(\mathbf{a},
\mathbf{b},\mathbf{c},\mathbf{d})|_e=n(\bar h)-1}\lambda_{\mathbf{a},\mathbf{b},\mathbf{c},\mathbf{d}}
(-1)^{\sum\limits_{j=1}^nb_j+\sum\limits_{j=1}^{\frac{r-1}{2}}d_{j}} \bar x^{\mathbf{a}}\bar y^{\mathbf{b}}\cdot\bar u^{\mathbf{c}}\bar v_1^{d_1}\cdots\bar v_{\frac{\sfr-1}{2}}^{d_{\frac{\sfr-1}{2}}}\otimes\bar1_\chi\\
\neq&0,
\end{array}\]
which contradicts to \eqref{mh}. This is to say, $(\Lambda_{l}\times\{\mathbf{0}\},\Lambda'_{q}\times\{\mathbf{0}\},\mathbf{0},
\mathbf{e}_{\frac{\sfr+1}{2}})\in\Lambda^{\text{max}}_{\bar h}$ is not possible. The proof is completed.

\end{proof}

In the following we will introduce a basis of the refined reduced $W$-superalgebra $(Q_\chi^{\chi})^{\text{ad}\,{\mmm}'_\mathds{k}}$ over $\mathds{k}$. We mainly follow Premet's strategy on finite $W$-algebras in \cite[Proposition 3.3]{P2}, with a few modifications.

\begin{prop}\label{reduced basis} For any $(\mathbf{a},\mathbf{b})\in\Lambda_l\times\Lambda'_q$ there is $\bar h_{\mathbf{a},\mathbf{b}}\in (Q_\chi^{\chi})^{\text{ad}\,{\mmm}'_\mathds{k}}$ such that $\Lambda_{\bar h_{\mathbf{a},\mathbf{b}}}^{\text{max}}=\{(\mathbf{a},\mathbf{b})\}$. The vectors $\{\bar h_{\mathbf{a},\mathbf{b}}\mid(\mathbf{a},\mathbf{b})\in\Lambda_l\times\Lambda'_q\}$ form a basis of $(Q_\chi^{\chi})^{\text{ad}\,{\mmm}'_\mathds{k}}$ over $\mathds{k}$.
\end{prop}

\begin{proof}
As $(Q_\chi^\chi)^{\text{ad}\,{\mmm}'_\mathds{k}}\cong U_\chi(\ggg_\mathds{k},e)$ when $\sfr$ is even, the proposition follows from \cite[Proposition 3.5(1)]{ZS2} in this case.

Now we will consider the case when $\sfr$ is odd. For $k\in\mathbb{Z}_+$ let $H^k$ denote the ${\bbk}$-linear span of all $0\neq \bar h\in (Q_\chi^\chi)^{\text{ad}\,{\mmm}'_\mathds{k}}$ with $n(\bar h)\leqslant  k$.
Given $(a,b)\in\mathbb{Z}_+^2$, let $H^{a,b}$ denote the subspace of $(Q_\chi^\chi)^{\text{ad}\,{\mmm}'_\mathds{k}}$ spanned by $H^{a-1}$ and all $\bar h\in (Q_\chi^\chi)^{\text{ad}\,{\mmm}'_\mathds{k}}$ with $n(\bar h)=a,~N(\bar h)\leqslant  b$. Order the elements in $\mathbb{Z}_+^2$ lexicographically. By construction, $H^{a,b}\subseteq H^{c,d}$ whenever $(a,b)\prec(c,d)$. Note that $(Q_\chi^\chi)^{\text{ad}\,{\mmm}'_\mathds{k}}$ is finite-dimensional, then $(Q_\chi^\chi)^{\text{ad}\,{\mmm}'_\mathds{k}}$ has basis $B:=\bigcup_{(i,j)}B_{i,j}$ such that $n(\mu)=i,~N(\mu)=j$ whenever $\mu\in B_{i,j}$.

For $i,j\in\mathbb{Z}$, take $\pi_{ij}$ to be an endomorphism of $Q_\chi^\chi$ defined via
\begin{align*}\label{pi}
\pi_{ij}(\bar x^{\mathbf{a}}\bar y^{\mathbf{b}}\bar u^\mathbf{c}\bar v^\mathbf{d}\otimes\bar1_\chi)=\begin{cases}\bar x^{\mathbf{a}}\bar y^{\mathbf{b}}\bar u^\mathbf{c}\bar v^\mathbf{d}\otimes\bar 1_\chi &\mbox{ if }\text{deg}_e(\bar x^{\mathbf{a}}\bar y^{\mathbf{b}}\bar u^\mathbf{c}\bar v^\mathbf{d})=i\\ &\mbox{ and } \text{wt}(\bar x^{\mathbf{a}}\bar y^{\mathbf{b}}\bar u^\mathbf{c}\bar v^\mathbf{d})=j;\\0 &\mbox{otherwise}.
\end{cases}
\end{align*}
Define the linear map $\pi_B: (Q_\chi^\chi)^{\text{ad}\,{\mmm}'_\mathds{k}}\longrightarrow Q_\chi^\chi$ via $\pi_B(\mu)=\pi_{i,j}(\mu)$ for any $\mu\in B_{i,j}$. It follows  from Lemma \ref{hw} that $\pi_B$ maps $(Q_\chi^\chi)^{\text{ad}\,{\mmm}'_\mathds{k}}$ into the subspace of $U_\chi({\ggg}_{\bbk}^e)\otimes\bar1_\chi$ of $Q_\chi^\chi$. By construction, $\pi_B$ is injective.

By the same discussion as in \cite[Proposition 4.9]{ZS2}, one can verify that
\begin{equation}\label{mvm'}
(Q_\chi^\chi)^{\text{ad}\,{\mmm}'_\mathds{k}}=[\bar v_{\frac{\sfr+1}{2}}\otimes\bar1_\chi,(Q_\chi^\chi)^{\text{ad}\,{\mmm}_\mathds{k}}].
\end{equation}
Let $((Q_\chi^\chi)^{\text{ad}\,{\mmm}_\mathds{k}})^{\text{ad}\,\bar v_{\frac{\sfr+1}{2}}}$ denote the invariant subspace of $(Q_\chi^\chi)^{\text{ad}\,{\mmm}_\mathds{k}}$ under the adjoint action of $\bar v_{\frac{\sfr+1}{2}}$. Obviously we have $$((Q_\chi^\chi)^{\text{ad}\,{\mmm}_\mathds{k}})^{\text{ad}\,\bar v_{\frac{\sfr+1}{2}}}=(Q_\chi^\chi)^{\text{ad}\,{\mmm}'_\mathds{k}}$$ by definition, and we can conclude from \eqref{mvm'} that
\begin{equation*}
\begin{split}
\text{dim}\,(Q_\chi^\chi)^{\text{ad}\,{\mmm}'_\mathds{k}}=&\text{dim}\,[\bar v_{\frac{\sfr+1}{2}}\otimes\bar1_\chi,(Q_\chi^\chi)^{\text{ad}\,{\mmm}_\mathds{k}}]=\text{dim}\,(Q_\chi^\chi)^{\text{ad}\,{\mmm}_\mathds{k}}
-\text{dim}\,((Q_\chi^\chi)^{\text{ad}\,{\mmm}_\mathds{k}})^{\text{ad}\,\bar v_{\frac{\sfr+1}{2}}}\\
=&\text{dim}\,(Q_\chi^\chi)^{\text{ad}\,{\mmm}_\mathds{k}}
-\text{dim}\,(Q_\chi^\chi)^{\text{ad}\,{\mmm}'_\mathds{k}},
\end{split}
\end{equation*}i.e.,
\begin{equation}\label{half}
\text{dim}\,(Q_\chi^\chi)^{\text{ad}\,{\mmm}'_\mathds{k}}=\frac{1}{2}\text{dim}\,
(Q_\chi^\chi)^{\text{ad}\,{\mmm}_\mathds{k}}.
\end{equation}

Recall that in the proof of \cite[Proposition 3.5]{ZS2}, we showed that $$\text{dim}\,
(Q_\chi^\chi)^{\text{ad}\,{\mmm}_\mathds{k}}=p^{\text{dim}\,({\ggg}_{\bbk}^e)_{\bar{0}}}\cdot2^{\text{dim}\,({\ggg}_{\bbk}^e)_{\bar{1}}+1},$$ then it follows from \eqref{half} that $$\text{dim}\,(Q_\chi^\chi)^{\text{ad}\,{\mmm}'_\mathds{k}}=p^{\text{dim}\,({\ggg}_{\bbk}^e)_{\bar{0}}}\cdot2^{\text{dim}\,({\ggg}_{\bbk}^e)_{\bar{1}}}
=\text{dim}\,U_\chi({\ggg}_{\bbk}^e)\otimes\bar1_\chi.$$
Thus $\pi_B: (Q_\chi^\chi)^{\text{ad}\,{\mmm}'_\mathds{k}}\longrightarrow U_\chi({\ggg}_{\bbk}^e)\otimes\bar1_\chi$ is a linear isomorphism. For $(\mathbf{a},\mathbf{b})=(a_1,\cdots,a_l;b_1,\cdots,b_q)\in\Lambda_l\times\Lambda'_q$ set $$h_{\mathbf{a},\mathbf{b}}=\pi_B^{-1}(\bar x_1^{a_1}\cdots \bar x_l^{a_l}\bar y_1^{b_1}\cdots \bar y_q^{b_q}\otimes\bar1_\chi).$$ By the bijectivity of $\pi_B$ and the PBW theorem of $U_\chi({\ggg}_{\bbk}^e)$, the vectors $h_{\mathbf{a},\mathbf{b}}$ with $(\mathbf{a},\mathbf{b})\in\Lambda_l\times\Lambda'_q$ form a basis of $(Q_\chi^\chi)^{\text{ad}\,{\mmm}'_\mathds{k}}$, while from the definition of $\pi_B$ it follows that $\Lambda_{\bar h_{\mathbf{a},\mathbf{b}}}^{\text{max}}=\{(\mathbf{a},\mathbf{b})\}$ for any $(\mathbf{a},\mathbf{b})\in\Lambda_l\times\Lambda'_q$.
\end{proof}

As an immediate consequence of Proposition~\ref{reduced basis}, we have

\begin{corollary}\label{rg}
There exist even elements $\theta_1,\cdots,\theta_l\in (Q_\chi^\chi)^{\text{ad}\,{\mmm}'_\mathds{k}}_{\bar0}$ and odd elements
$\theta_{l+1},\cdots,\theta_{l+q}\in (Q_\chi^\chi)^{\text{ad}\,{\mmm}'_\mathds{k}}_{\bar1}$ such that
\begin{itemize}
\item[(a)]\begin{equation*}
\begin{split}
\theta_k=&(\bar x_k+\sum\limits_{\mbox{\tiny $\begin{array}{c}|\mathbf{a},\mathbf{b},\mathbf{c},\mathbf{d}|_e=m_k+2,\\|\mathbf{a}|
+|\mathbf{b}|+|\mathbf{c}|+|\mathbf{d}|\geqslant 2\end{array}$}}\lambda^k_{\mathbf{a},\mathbf{b},\mathbf{c},\mathbf{d}}\bar x^{\mathbf{a}}
\bar y^{\mathbf{b}}\bar u^{\mathbf{c}}\bar v^{\mathbf{d}}\\&+\sum\limits_{|\mathbf{a},\mathbf{b},\mathbf{c},\mathbf{d}|_e<m_k+2}
\lambda^k_{\mathbf{a},\mathbf{b},\mathbf{c},\mathbf{d}}\bar x^{\mathbf{a}}
\bar y^{\mathbf{b}}\bar u^{\mathbf{c}}\bar v^{\mathbf{d}})\otimes\bar 1_\chi,
\end{split}
\end{equation*}
where $\bar x_k\in{\ggg}^e_{\bbk}(m_k)_{\bar0}$ for $1\leqslant  k\leqslant  l$.
\item[(b)]\begin{equation*}
\begin{split}
\theta_{l+k}=&(\bar y_k+\sum\limits_{\mbox{\tiny $\begin{array}{c}|\mathbf{a},\mathbf{b},\mathbf{c},\mathbf{d}|_e=n_k+2,\\|\mathbf{a}|
+|\mathbf{b}|+|\mathbf{c}|+|\mathbf{d}|\geqslant 2\end{array}$}}\lambda^k_{\mathbf{a},\mathbf{b},\mathbf{c},\mathbf{d}}\bar x^{\mathbf{a}}
\bar y^{\mathbf{b}}\bar u^{\mathbf{c}}\bar v^{\mathbf{d}}\\&+\sum\limits_{|\mathbf{a},\mathbf{b},\mathbf{c},\mathbf{d}|_e<n_k+2}
\lambda^k_{\mathbf{a},\mathbf{b},\mathbf{c},\mathbf{d}}\bar x^{\mathbf{a}}
\bar y^{\mathbf{b}}\bar u^{\mathbf{c}}\bar v^{\mathbf{d}})\otimes\bar 1_\chi,
\end{split}
\end{equation*}
where $\bar y_k\in{\ggg}^e_{\bbk}(n_k)_{\bar1}$ for $1\leqslant  k\leqslant  q$.
\end{itemize}
All the coefficients $\lambda^k_{\mathbf{a},\mathbf{b},\mathbf{c},\mathbf{d}}$ above are in ${\bbk}$. Moreover, $\lambda^k_{\mathbf{a},\mathbf{b},\mathbf{c},\mathbf{d}}=0$ if $(\mathbf{a},\mathbf{b},\mathbf{c},\mathbf{d})$ is such that $a_{l+1}=\cdots=a_m=b_{q+1}=\cdots=b_n=c_1=\cdots=c_s=d_1=\cdots=d_{\lfloor\frac{\sfr}{2}\rfloor}=0$.
\end{corollary}

Set
\[\bar Y_i:=\left\{
\begin{array}{ll}
\bar x_i&\text{if}~1\leqslant  i\leqslant  l;\\
\bar y_{i-l}&\text{if}~l+1\leqslant  i\leqslant  l+q,
\end{array}
\right.
\]
and assume that the homogenous element $\bar Y_i$ is in $\ggg_{\bbk}(m_i)$ for  $1\leqslant  i\leqslant  l+q$.
Since $(Q_\chi^\chi)^{\text{ad}\,{\mmm}'_\mathds{k}}$ is a subalgebra of the reduced $W$-superalgebra $U_\chi(\ggg_\mathds{k},e)$,
we can substitute the generators $\{\theta_1,\cdots,\theta_{l+q}\}$ of $U_\chi(\ggg_\mathds{k},e)$ in \cite[Corollary 3.6(1)]{ZS2} for the ones in Corollary~\ref{rg}, and all the results in \cite[Theorem 3.7]{ZS2} remain true. What is more, the same discussion as in \cite[Remark 3.8(2)]{ZS2} shows that there exist super-polynomials $\bar F_{ij}$ of $l+q$ indeterminants over $\bbk$ ($1\leqslant  i<j\leqslant  l+q$, or $l+1\leqslant  i=j\leqslant  l+q$) with the first $l$ indeterminants being even, and the others being odd, such that
\begin{align*}\label{explainforeq}
[\theta_i,\theta_j]=\bar F_{i,j}(\theta_1,\cdots,\theta_{l+q}),\quad i,j=1,\cdots,l+q.
\end{align*}
Moreover, the Lie bracket relations in the Lie superalgebra $\ggg_\bbk^e$ with
$$[\bar Y_i,\bar Y_j]=\sum\limits_{k=1}^{l+q}\alpha_{ij}^k\bar Y_k$$ imply
$$\bar F_{ij}(\theta_1,\cdots,\theta_{l+q})\equiv\sum\limits_{k=1}^{l+q}\alpha_{ij}^k\theta_k+\bar q_{ij}(\theta_1,\cdots,\theta_{l+q})\quad(\mbox{mod }F_{m_i+m_j+1}(Q_\chi^\chi)^{\text{ad}\,{\mmm}'_\mathds{k}}),$$ where $\bar q_{ij}\in\mathds{k}[X_1,\cdots,X_l;X_{l+1},\cdots,X_{l+q}]$ is a super-polynomial in $l+q$ invariables whose constant term and linear part are both zero, and $F_{m_i+m_j+1}(Q_\chi^\chi)^{\text{ad}\,{\mmm}'_\mathds{k}}$ denotes the component of Kazhdan filtration of $(Q_\chi^\chi)^{\text{ad}\,{\mmm}'_\mathds{k}}$ with degree $m_i+m_j+1$.

\subsection{The PBW structure theory of refined $W$-superalgebras over $\mathbb{C}$}\label{PBWCC}
Recall that in \cite[Remark 70]{W} Wang introduced what we call refined $W$-superalgebras as follows.
\begin{defn}(\cite{W})\label{rewcc} Let $\ggg$ be a basic Lie superalgebra over $\mathbb{C}$. Define the refined $W$-superalgebra over $\bbc$ by$$W'_\chi:=(U({\ggg})/I_\chi)^{\text{ad}\,{\mmm}'}\cong Q_\chi^{\text{ad}\,{\mmm}'}
\equiv\{\bar{y}\in U({\ggg})/I_\chi \mid [a,y]\in I_\chi, \forall a\in{\mmm}'\},$$
and $\bar{y}_1\cdot\bar{y}_2:=\overline{y_1y_2}$ for all $\bar{y}_1,\bar{y}_2\in W'_\chi$.
\end{defn}

In this subsection we will concentrate on the structure of refined $W$-superalgebras over $\mathbb{C}$. First note that
\begin{lemma}\label{hwc2}
Let $h\in W_\chi'\backslash\{0\}$ and $(\mathbf{a},\mathbf{b},\mathbf{c},\mathbf{d})\in \Lambda^{\text{max}}_h$. Then $\mathbf{a}\in\mathbb{Z}_+^{l}\times\{\mathbf{0}\},~\mathbf{b}\in\Lambda'_{q}\times\{\mathbf{0}\}$, $\mathbf{c}=\mathbf{0}$,  $\mathbf{d}=\mathbf{0}$.
\end{lemma}
\begin{proof}
The proof is the same as proof of Lemma \ref{hw} but apply \cite[Lemma 4.3]{ZS2} in place of \cite[Lemma 3.3]{ZS2}, thus will be omitted.
\end{proof}

Recall that $\{x_1,\cdots,x_l\}$ and $\{y_1,\cdots,y_q\}$ are the $\mathbb{C}$-basis of ${\ggg}^e_{\bar{0}}$ and ${\ggg}^e_{\bar{1}}$, respectively. Set
\begin{equation}\label{Y_i}
Y_i:=\left\{
\begin{array}{ll}
x_i&\text{if}~1\leqslant  i\leqslant  l;\\
y_{i-l}&\text{if}~l+1\leqslant  i\leqslant  l+q.
\end{array}
\right.
\end{equation}
Assume that $Y_i$ belongs to ${\ggg}(m_i)$ for $1\leqslant  i\leqslant  l+q$.
In virtue of the structure theory of $(Q_\chi^{\chi})^{\text{ad}\,{\mmm}'_\mathds{k}}$ over $\mathds{k}$ obtained in \S\ref{rrw}, the same discussion as in \cite[Theorem 4.5, Theorem 4.7]{ZS2} shows that

\begin{theorem}\label{PBWC} The following PBW structural statements for $W'_\chi$ hold.
\begin{itemize}
\item[(1)] There exist homogeneous elements $\Theta_1,\cdots,\Theta_{l}\in (W'_\chi)_{\bar0}$ and $\Theta_{l+1},\cdots,\Theta_{l+q}\in (W'_\chi)_{\bar1}$ such that
\begin{equation*}
\begin{split}
\Theta_k=&(Y_k+\sum\limits_{\mbox{\tiny $\begin{array}{c}|\mathbf{a},\mathbf{b},\mathbf{c},\mathbf{d}|_e=m_k+2,\\|\mathbf{a}|
+|\mathbf{b}|+|\mathbf{c}|+|\mathbf{d}|\geqslant 2\end{array}$}}\lambda^k_{\mathbf{a},\mathbf{b},\mathbf{c},\mathbf{d}}x^{\mathbf{a}}
y^{\mathbf{b}}u^{\mathbf{c}}v^{\mathbf{d}}\\&+\sum\limits_{|\mathbf{a},\mathbf{b},\mathbf{c},\mathbf{d}|_e<m_k+2}\lambda^k_{\mathbf{a},\mathbf{b},\mathbf{c},\mathbf{d}}x^{\mathbf{a}}
y^{\mathbf{b}}u^{\mathbf{c}}v^{\mathbf{d}})\otimes1_\chi
\end{split}
\end{equation*}
for $1\leqslant  k\leqslant  l+q$, where $\lambda^k_{\mathbf{a},\mathbf{b},\mathbf{c},\mathbf{d}}\in\mathbb{Q}$, and $\lambda^k_{\mathbf{a},\mathbf{b},\mathbf{c},\mathbf{d}}=0$ if $a_{l+1}=\cdots=a_m=b_{q+1}=\cdots=b_n=c_1=\cdots=c_s=
d_1=\cdots=d_{\lfloor\frac{\sfr}{2}\rfloor}=0$.

\item[(2)] The monomials $\Theta_1^{a_1}\cdots\Theta_l^{a_l}\Theta_{l+1}^{b_1}\cdots\Theta_{l+q}^{b_{q}}$ with $a_i\in\mathbb{Z}_+,~b_j\in\Lambda'_1$ for $1\leqslant i\leqslant l$ and $1\leqslant j\leqslant q$ form a basis of $W'_\chi$ over $\mathbb{C}$.

\item[(3)] For $i,j$ satisfying $1\leqslant  i<j\leqslant  l+q$ and $l+1\leqslant  i=j\leqslant  l+q$, there exist super-polynomials
$F_{ij}\in\mathbb{Q}[X_1,\cdots,X_{l+q}]$ such that $$[\Theta_i,\Theta_{j}]=F_{ij}(\Theta_1,\cdots,\Theta_{l+q}).$$ Moreover, if the elements $Y_i,~Y_j\in {\ggg}^e$ satisfy $[Y_i,Y_j]=\sum\limits_{k=1}^{l+q}\alpha_{ij}^kY_k$ in ${\ggg}^e$, then
\begin{equation}\label{Thetaa2}
\begin{array}{llllll}
&F_{ij}[X_1,\cdots,X_{l+q}]&\equiv&\sum\limits_{k=1}^{l+q}\alpha_{ij}^k\Theta_k+q_{ij}(\Theta_1,\cdots,\Theta_{l+q})&(\mbox{mod }\text{F}_{m_i+m_j+1}W_\chi'),
\end{array}
\end{equation}where $q_{ij}$ is a super-polynomial in $l+q$ variables in $\mathbb{Q}$ whose constant term and linear part are zero, and $\text{F}_{m_i+m_j+1}W_\chi'$ denotes the component of Kazhdan filtration of $W_\chi'$ with degree $m_i+m_j+1$.

\item[(4)] The refined $W$-superalgebra $W'_\chi$ is generated by the $\mathbb{Z}_2$-homogeneous elements $\Theta_1,\cdots,\Theta_{l}\in (W'_\chi)_{\bar0}$ and $\Theta_{l+1},\cdots,\Theta_{l+q}\in (W'_\chi)_{\bar1}$ subject to the relations in \eqref{Thetaa2}
with $1\leqslant  i<j\leqslant  l+q$ and $l+1\leqslant  i=j\leqslant  l+q$.
\end{itemize}
\end{theorem}

As a direct corollary of Theorem~\ref{PBWC}, we have
\begin{corollary}\label{isogr}
There is an isomorphism between $\mathbb{C}$-algebras $$\text{gr}\,(W_\chi')\cong S(\ggg^e),$$ where $\text{gr}\,(W_\chi')$ denotes the graded algebra of $W_\chi'$ under Kazhdan grading, and $S(\ggg^e)$ is the supersymmetric algebra on $\ggg^e$.
\end{corollary}
\begin{rem}
It is worth noting that in the case when $\sfr$ is even, we have $W_\chi'\cong U(\ggg,e)$ as $\mathbb{C}$-algebras, and Corollary \ref{isogr} is just \cite[Theorem 0.1(1)]{ZS2}. But for the case when $\sfr$ is odd, $W_\chi'$ is a proper subalgebra of $U(\ggg,e)$, and the consequence of Corollary \ref{isogr} is completely different from that of \cite[Theorem 0.1(2)]{ZS2}.
\end{rem}

\section{Refined $W$-superalgebras and the dimensional lower bounds for the modular representations of basic Lie superalgebras}\label{Refined W-superalgebras and the dimensional lower bounds}
In virtue of the results on refined $W$-superalgebras in \S\ref{crw}, in this section we will refine \cite[Conjecture 1.3]{ZS4}, then discuss the accessibility of the dimensional lower bounds for the modular representations of basic Lie superalgebras based on this conjecture. In the final part we will introduce an equivalent definition of refined $W$-superalgebras, which will be applied to describe the structure of minimal $W$-superalgebras in the next section.
\subsection{The super Kac-Weisfeiler property}
Recall that the dimensional lower bounds for the irreducible representations of $U_\xi({\ggg}_{\bbk})$ were predicted by Wang-Zhao in \cite[Theorem 5.6]{WZ} as below:
\begin{prop}(\cite{WZ})\label{wzd2}
Let ${\ggg}_{\bbk}$ be a basic Lie superalgebra over ${\bbk}=\overline{\mathbb{F}}_p$, assuming that the prime $p$ satisfies the restriction imposed in \cite[Table 1]{WZ}. Let $\xi$ be arbitrary $p$-character in $({\ggg}_{\bbk})^*_{\bar0}$, corresponding to an element $\bar x\in ({\ggg}_{\bbk})_{\bar0}$ such that $\xi(\bar y)=(\bar x,\bar y)$ for any $\bar y\in{\ggg}_{\bbk}$. Set $d_0=\text{dim}\,({\ggg}_{\bbk})_{\bar 0}-\text{dim}\,({\ggg}_{\bbk}^{\bar x})_{\bar 0}$ and $d_1=\text{dim}\,({\ggg}_{\bbk})_{\bar 1}-\text{dim}\,({\ggg}_{\bbk}^{\bar x})_{\bar 1}$, where ${\ggg}_{\bbk}^{\bar x}$ denotes the centralizer of $\bar x$ in ${\ggg}_{\bbk}$. Denote by $\lfloor\frac{d_1}{2}\rfloor$ the least integer upper bound of $\frac{d_1}{2}$. Then the dimension of every $U_\xi({\ggg}_{\bbk})$-module $M$ is divisible by $p^{\frac{d_0}{2}}2^{\lfloor\frac{d_1}{2}\rfloor}$.
\end{prop}

A natural question is whether or not there exist some irreducible modules of $U_\xi(\ggg_\bbk)$ with dimension equaling $p^{\frac{d_0}{2}}2^{\lfloor\frac{d_1}{2}\rfloor}$. Conjecture \ref{conjectureold} on the minimal dimensional representations of the complex finite $W$-superalgebra $U({\ggg},e)$ was  raised in \cite{ZS4}.
Under the assumption of the conjecture, the lower bounds of the dimensions given in Proposition \ref{wzd2} are correspondingly to be indicated accessible for $p\gg0$ in \cite[Theorem 1.5]{ZS4}.

The conjecture in the general case is still open, out of scope we solved in \cite[Proposition 4.7, Proposition 5.8]{ZS4}.
\subsection{A conjecture on the representation of refined $W$-superalgebras}
In virtue of the structure theory of refined $W$-superalgebras in \S\ref{crw}, we can reformulate Conjecture \ref{conjectureold} as follows.
\begin{conj}\label{conjecture22}
Let ${\ggg}$ be a basic Lie superalgebra over ${\bbc}$, then the refined $W$-superalgebra $W_\chi'$ affords a one-dimensional representation.
\end{conj}

Now we make some explanation on the above. Let ${\ggg}$ be a basic Lie superalgebra over $\mathbb{C}$.  For the case with $\sfr$ being odd, we can obtain the following result.

\begin{prop}\label{intromain}
Assume $\sfr$ is odd. If the refined $W$-superalgebra $W_\chi'$ affords a one-dimensional representation, then the corresponding finite $W$-superalgebra $U({\ggg},e)$ admits a two-dimensional irreducible representation.
\end{prop}
\begin{proof}
As the refined $W$-superalgebra $W_\chi'$ affords a one-dimensional representation, it is immediate that $Q_\chi^{\text{ad}\,{\mmm}'}=W_\chi'$ has a two-sided ideal of codimensional $1$. Denote this ideal by $I$. In the following we will prove that $I\oplus I(v_{\frac{\sfr+1}{2}}\otimes1_\chi)$ is a two-sided ideal of codimensional $2$ in $Q_\chi^{\text{ad}\,{\mmm}}\cong U({\ggg},e)$.

For any  $\mathbb{Z}_2$-homogeneous element $h\in Q_\chi^{\text{ad}\,{\mmm}'}$, both $h\cdot I$ and $I\cdot h$ are contained in $I$, then $h\cdot I(v_{\frac{\sfr+1}{2}}\otimes1_\chi)\subseteq I(v_{\frac{\sfr+1}{2}}\otimes1_\chi)$. Let $h'$ be a $\mathbb{Z}_2$-homogeneous element in $I$, then $$h'(v_{\frac{\sfr+1}{2}}\otimes1_\chi)\cdot h=(-1)^{|h|}h'h (v_{\frac{\sfr+1}{2}}\otimes1_\chi)+h'[v_{\frac{\sfr+1}{2}}\otimes1_\chi,h].$$ Since $h'h(v_{\frac{\sfr+1}{2}}\otimes1_\chi)\in I(v_{\frac{\sfr+1}{2}}\otimes1_\chi)$, and $[v_{\frac{\sfr+1}{2}}\otimes1_\chi,h]=0$ by the definition of $Q_\chi^{\text{ad}\,{\mmm}'}$, we have $h'(v_{\frac{\sfr+1}{2}}\otimes1_\chi)\cdot h\in I(v_{\frac{\sfr+1}{2}}\otimes1_\chi)$, which yields $I(v_{\frac{\sfr+1}{2}}\otimes1_\chi)\cdot h\subseteq I(v_{\frac{\sfr+1}{2}}\otimes1_\chi)$ by the arbitrary of $h'$. Summing up, we have
\begin{equation}\label{eq0}
Q_\chi^{\text{ad}\,{\mmm}'}\cdot(I\oplus I(v_{\frac{\sfr+1}{2}}\otimes1_\chi))\subseteq I\oplus I(v_{\frac{\sfr+1}{2}}\otimes1_\chi),\quad(I\oplus I(v_{\frac{\sfr+1}{2}}\otimes1_\chi))\cdot Q_\chi^{\text{ad}\,{\mmm}'}\subseteq I\oplus I(v_{\frac{\sfr+1}{2}}\otimes1_\chi).
\end{equation}

Now consider the element $v_{\frac{\sfr+1}{2}}\otimes1_\chi\in Q_\chi^{\text{ad}\,{\mmm}}\backslash Q_\chi^{\text{ad}\,{\mmm}'}$.
Let $h$ be arbitrary $\mathbb{Z}_2$-homogeneous element in $I\subseteq Q_\chi^{\text{ad}\,{\mmm}'}$. Note that $$v_{\frac{\sfr+1}{2}}^2\otimes1_\chi=\frac{1}{2}
[v_{\frac{\sfr+1}{2}},v_{\frac{\sfr+1}{2}}]\otimes1_\chi=\frac{1}{2}\otimes1_\chi,$$
and $[v_{\frac{\sfr+1}{2}}\otimes1_\chi,h]=0$ by the definition of $Q_\chi^{\text{ad}\,{\mmm}'}$,
then we have
\begin{equation*}
\begin{split}
(v_{\frac{\sfr+1}{2}}\otimes1_\chi)\cdot h=&[v_{\frac{\sfr+1}{2}}\otimes1_\chi,h]+(-1)^{|h|}h (v_{\frac{\sfr+1}{2}}\otimes1_\chi)\\
=&(-1)^{|h|}h(v_{\frac{\sfr+1}{2}}\otimes1_\chi)\in I(v_{\frac{\sfr+1}{2}}\otimes1_\chi),\\
(v_{\frac{\sfr+1}{2}}\otimes1_\chi)\cdot h(v_{\frac{\sfr+1}{2}}\otimes1_\chi)
=&[v_{\frac{\sfr+1}{2}}\otimes1_\chi,h](v_{\frac{\sfr+1}{2}}\otimes1_\chi)+(-1)^{|h|}h(v_{\frac{\sfr+1}{2}}^2\otimes1_\chi)\\
=&(-1)^{|h|}[v_{\frac{\sfr+1}{2}}\otimes1_\chi,[v_{\frac{\sfr+1}{2}}\otimes1_\chi,h]]-(-1)^{|h|}(v_{\frac{\sfr+1}{2}}\otimes1_\chi)\cdot\\
&[v_{\frac{\sfr+1}{2}}\otimes1_\chi,h]+(-1)^{|h|}h(v_{\frac{\sfr+1}{2}}^2\otimes1_\chi)\\
=&\frac{(-1)^{|h|}}{2}h\in I,
\end{split}\end{equation*}
which shows that \begin{equation}\label{eq1}
(v_{\frac{\sfr+1}{2}}\otimes1_\chi)\cdot(I\oplus I(v_{\frac{\sfr+1}{2}}\otimes1_\chi))\subseteq I\oplus I(v_{\frac{\sfr+1}{2}}\otimes1_\chi).\end{equation}

On the other hand, we have
\begin{equation*}
\begin{split}
h\cdot(v_{\frac{\sfr+1}{2}}\otimes1_\chi)\in&I(v_{\frac{\sfr+1}{2}}\otimes1_\chi),\\
h(v_{\frac{\sfr+1}{2}}\otimes1_\chi)\cdot(v_{\frac{\sfr+1}{2}}\otimes1_\chi)=&h(v_{\frac{\sfr+1}{2}}^2\otimes1_\chi)
=\frac{h}{2}\in I,\end{split}\end{equation*}
i.e., \begin{equation}\label{eq2}(I\oplus I(v_{\frac{\sfr+1}{2}}\otimes1_\chi))\cdot(v_{\frac{\sfr+1}{2}}\otimes1_\chi)\subseteq I\oplus I(v_{\frac{\sfr+1}{2}}\otimes1_\chi).\end{equation}
Since the $\mathbb{C}$-algebra $Q_\chi^{\text{ad}\,{\mmm}}$ is generated by $v_{\frac{\sfr+1}{2}}\otimes1_\chi$ and the elements in $Q_\chi^{\text{ad}\,{\mmm}'}$, by \eqref{eq0}, \eqref{eq1} and \eqref{eq2} we can conclude that $I\oplus I(v_{\frac{\sfr+1}{2}}\otimes1_\chi)$ is a two-sided ideal of $U({\ggg},e)$.

Moreover, it is immediate from the PBW theorem of $Q_\chi^{\text{ad}\,{\mmm}}$ in \cite[Theorem 4.5(2)]{ZS2} that $Q_\chi^{\text{ad}\,{\mmm}}/(I\oplus I(v_{\frac{\sfr+1}{2}}\otimes1_\chi))\cong \bbc_\chi\oplus\bbc (v_{\frac{\sfr+1}{2}}\otimes1_\chi)$ as $\bbc$-vector spaces, then $I\oplus I(v_{\frac{\sfr+1}{2}}\otimes1_\chi)$ is codimensional $2$ in $U({\ggg},e)\cong Q_\chi^{\text{ad}\,{\mmm}}$. Thus the $\mathbb{C}$-algebra $U({\ggg},e)$ admits a two-dimensional irreducible representation, completing the proof.
\end{proof}

Recall that for the case when $\sfr$ is even, we have $W_\chi'=U(\ggg,e)$ by the definition. Then
in virtue of Proposition \ref{intromain}, we can refine Conjecture \ref{conjectureold} as Conjecture \ref{conjecture22}.
It is worth noting that in contrast to Conjecture \ref{conjectureold}, the assumption in Conjecture \ref{conjecture22} is irrelevant to the parity of $\sfr$.

\subsection{The minimal dimensional representations of reduced enveloping algebras}
We will consider the minimal dimensional representations of reduced enveloping algebra of a
basic Lie superalgebra in this subsection.

Let ${\ggg}_{\bbk}$ be a basic Lie superalgebra over ${\bbk}=\overline{\mathbb{F}}_p$. Set $d_0=\text{dim}\,({\ggg}_{\bbk})_{\bar 0}-\text{dim}\,({\ggg}_{\bbk}^{e})_{\bar 0}$ and $d_1=\text{dim}\,({\ggg}_{\bbk})_{\bar 1}-\text{dim}\,({\ggg}_{\bbk}^{e})_{\bar 1}$, where ${\ggg}_{\bbk}^{e}$ denotes the centralizer of $e$ in ${\ggg}_{\bbk}$.
Denote by $\lfloor\frac{d_1}{2}\rfloor$ the least integer upper bound of $\frac{d_1}{2}$. The main result of this section is the following theorem.

\begin{theorem}\label{intromain-2}
Let ${\ggg}_{\bbk}$ be a basic Lie superalgebra over ${\bbk}=\overline{\mathbb{F}}_p$, and let $\chi\in({\ggg}_{\bbk})^*_{\bar0}$ be the $p$-character of ${\ggg}_{\bbk}$ corresponding to a nilpotent element $e\in(\ggg_\bbk)_\bz$ such that $\chi(\bar y)=(e,\bar y)$ for any $\bar y\in{\ggg}_{\bbk}$. If Conjecture \ref{conjecture22} establishes for the refined $W$-superalgebra $W_\chi'$ associated with (${\ggg},e)$ over $\mathbb{C}$, then for $p\gg0$ the reduced enveloping algebra $U_\chi({\ggg}_{\bbk})$ admits irreducible representations of dimension $p^{\frac{d_0}{2}}2^{\lfloor\frac{d_1}{2}\rfloor}$.
\end{theorem}
\begin{proof}
The theorem readily follows from proposition \ref{intromain} and \cite[Theorem 1.6]{ZS4}.
\end{proof}

As for the general case, let $\xi\in({\ggg}_{\bbk})^*_{\bar0}$ be any $p$-character of ${\ggg}_{\bbk}$ corresponding to an element $\bar x\in ({\ggg}_{\bbk})_{\bar0}$ such that $\xi(\bar y)=(\bar x,\bar y)$ for any $\bar y\in{\ggg}_{\bbk}$. Let $d'_0=\text{dim}\,({\ggg}_{\bbk})_{\bar 0}-\text{dim}\,({\ggg}_{\bbk}^{\bar x})_{\bar 0}$ and $d'_1=\text{dim}\,({\ggg}_{\bbk})_{\bar 1}-\text{dim}\,({\ggg}_{\bbk}^{\bar x})_{\bar 1}$, where ${\ggg}_{\bbk}^{\bar x}$ denotes the centralizer of $\bar x$ in ${\ggg}_{\bbk}$. Then we have
\begin{theorem}\label{intromain-3}
Let ${\ggg}_{\bbk}$ be a basic Lie superalgebra over ${\bbk}=\overline{\mathbb{F}}_p$, and let $\xi\in({\ggg}_{\bbk})^*_{\bar0}$. If all the refined $W$-superalgebras associated with the basic Lie superalgebras over $\bbc$ excluding type $D(2,1;a)$ with $a\notin\overline{\mathbb{Q}}$ afford one-dimensional representations, then for $p\gg0$ the reduced enveloping algebra $U_\xi({\ggg}_{\bbk})$ admits irreducible representations of dimension $p^{\frac{d'_0}{2}}2^{\lfloor\frac{d'_1}{2}\rfloor}$.
\end{theorem}

\subsection{A variation on the definition of refined $W$-superalgebras}\label{variation}
In this part we will introduce a variation on the definition of refined $W$-superalgebras.

Let $\ggg$ be a basic Lie superalgebra over $\mathbb{C}$, then we have ${\ggg}=\bigoplus_{i\in{\bbz}}{\ggg}(i)$, where ${\ggg}(i)=\{x\in{\ggg}\mid[h,x]=ix\}$. Set $\mathfrak{n}:=\bigoplus_{i\leqslant-1}{\ggg}(i)$ to be a nilpotent subalgebra of $\ggg$.
Let $I^{\text{fin}}$ be the ideal of $U(\ggg)$ generated by the elements $\{x-\chi(x) \mid x\in\mathfrak{n}\}$, and let $Q^{\text{fin}}_\chi:=U(\ggg)/I^{\text{fin}}$ be an induced $\ggg$-module. In \cite[Definition 4.3]{suh} Suh introduced the quantum finite $W$-superalgebra $W^{\text{fin}}(\ggg,e)$ as follows.
\begin{defn}(\cite{suh})\label{quafi}
Define the quantum finite $W$-superalgebra $W^{\text{fin}}(\ggg,e)$ associated with $\ggg$ and $e$ by
$$W^{\text{fin}}(\ggg,e):=(Q^{\text{fin}}_\chi)^{\text{ad}\,\mathfrak{n}},$$
where $(Q^{\text{fin}}_\chi)^{\text{ad}\,\mathfrak{n}}$ denotes the invariant subspace of $Q^{\text{fin}}_\chi$ under the adjoint action of $\mathfrak{n}$, and the associative product of $W^{\text{fin}}(\ggg,e)$ is defined by $$(x+I^{\text{fin}})\cdot(y+I^{\text{fin}}):=xy+I^{\text{fin}}.$$
\end{defn}
In \cite[(3.25)]{suh}, Suh introduced a grading on $Q^{\text{fin}}_\chi$. Let $w_1,\cdots, w_c$ be a basis of $\ggg$ over $\bbc$. Let $U({\ggg})=\bigcup_{i\in{\bbz}}\text{F}_i^pU({\ggg})$ be a filtration of $U({\ggg})$, where $\text{F}_i^pU({\ggg})$ is the ${\bbc}$-span of all $w_1\cdots w_c$ with $w_1\in{\ggg}(j_1),\cdots,w_c\in{\ggg}(j_c)$ and $(j_1+1)+\cdots+(j_c+1)\leqslant  i$. The corresponding filtration on $Q_{\chi}^{\text{fin}}$ is defined by $\text{F}_i^pQ_{\chi}^{\text{fin}}:=\pi(\text{F}_i^pU({\ggg}))$ with $\pi:U({\ggg})\twoheadrightarrow U({\ggg})/I^{\text{fin}}$ being the canonical homomorphism, which makes $Q_{\chi}^{\text{fin}}$ into a filtered $U({\ggg})$-module. This grading is called $p$-grading.

Suppose $\{v_\alpha\}_{\alpha\in J}$ is a basis of $\ggg^e$ such that $v_\alpha\in\ggg(j_\alpha)$, and let $S=\{v_\alpha+a_\alpha\mid\alpha\in J\}$ be a subset of $W^{\text{fin}}(\ggg,e)$ such that the $p$-grading of $a_\alpha$ is smaller than that of $v_\alpha$. In \cite[Proposition 4.7]{suh} Suh showed that $S$ is a set of free generators of $W^{\text{fin}}(\ggg,e)$. Moreover, the proofs of \cite[Proposition 3.12]{suh} and \cite[Proposition 4.7]{suh} showed that the quantum finite $W$-superalgebra version of \cite[Proposition 3.10]{suh} is also valid, then we have
\begin{prop}\label{Suh}
There is an isomorphism between $\mathbb{C}$-algebras $$\text{gr}^p\,(W^{\text{fin}}(\ggg,e))\cong S(\ggg^e),$$ where $\text{gr}^p\,(W^{\text{fin}}(\ggg,e))$ denotes the graded algebra of $W^{\text{fin}}(\ggg,e)$ under $p$-grading, and
$S(\ggg^e)$ is the supersymmetric algebra on $\ggg^e$.
\end{prop}
\begin{rem}
It is worth noting that in \cite{suh} the quantum finite $W$-superalgebras are introduced in the context of vertex algebras and quantum reduction, and the nilpotent element chosen there is $f$, but not $e$. What is more, the original grading on $\ggg$ applied there was under the action of $\text{ad}\,\frac{h}{2}$, while the grading we used is under the action of $\text{ad}\,h$. Moreover, the $p$-grading on $Q^{\text{fin}}_\chi$ introduced there is also a little different from Kazhdan grading. Although there are so many differences, one can find that they are essentially the same, i.e., all the discussions there still go through in our settings. Therefore, the isomorphism between $\mathbb{C}$-algebras $\text{gr}\,(W^{\text{fin}}(\ggg,e))$ and $S(\ggg^e)$ in Proposition~\ref{Suh} can also be established under Kazhdan grading, i.e.,
\end{rem}

\begin{prop}\label{p-kazhdan}
There is an isomorphism between $\mathbb{C}$-algebras $$\text{gr}\,(W^{\text{fin}}(\ggg,e))\cong S(\ggg^e),$$ where $\text{gr}\,(W^{\text{fin}}(\ggg,e))$ denotes the graded algebra of $W^{\text{fin}}(\ggg,e)$ under Kazhdan grading, and
$S(\ggg^e)$ is the supersymmetric algebra on $\ggg^e$.
\end{prop}

As an immediate corollary of Corollary~\ref{isogr} and Proposition~\ref{p-kazhdan}, we have

\begin{prop}\label{p-kazhdan2}
The graded algebra of finite $W$-superalgebra $\text{gr}\,(W_\chi')$ is isomorphism to the graded algebra of the quantum finite $W$-superalgebras $\text{gr}\,(W^{\text{fin}}(\ggg,e))$ under Kazhdan grading.
\end{prop}

As $\ggg^e$ is finite-dimensional as a vector space, it is readily from Proposition \ref{p-kazhdan2} that
\begin{theorem}\label{important}
There is an isomorphism between Kazhdan filtered algebras
$$W_\chi'\cong W^{\text{fin}}(\ggg,e).$$
\end{theorem}

In virtue of Theorem \ref{important}, in the following we will take the quantum finite $W$-superalgebra $W^{\text{fin}}(\ggg,e)$ in Definition \ref{quafi} as the refined $W$-superalgebra $W_\chi'$ introduced in Definition \ref{rewcc} , which will cause no confusion.

\begin{rem}
In fact, we can generalize the definition of refined $W$-superalgebras as follows.
Fix an isotropic subspace $\mathfrak{l}$ of $\ggg(-1)$ with respect to $\langle\cdot,\cdot\rangle$, and let $\mathfrak{l}'=\{x\in\ggg(-1)\mid\langle x,\mathfrak{l}\rangle=0\}$. Obviously we have $\mathfrak{l}\subseteq \mathfrak{l}'$. Let $\mathfrak{m}_\mathfrak{l}$ and $\mathfrak{m}'_\mathfrak{l}$ be nilpotent subalgebras of $\ggg$ defined by
$${\mmm}_\mathfrak{l}=\bigoplus_{i\leqslant -2}{\ggg}(i)\oplus \mathfrak{l},\qquad {\mmm}'_{\mathfrak{l}}=\bigoplus_{i\leqslant -2}{\ggg}(i)\oplus \mathfrak{l}',$$ then we have ${\mmm}_\mathfrak{l}\subseteq{\mmm}'_{\mathfrak{l}}$.

The linear function $\chi\in\ggg^*$ restricts to a character on ${\mmm}_\mathfrak{l}$, and denote by $\mathbb{C}_\mathfrak{l}$ the corresponding $1$-dimensional representation of ${\mmm}_\mathfrak{l}$. Define the generalized Gelfand-Graev module associated with $\mathfrak{l}$ by $Q_\mathfrak{l}=U({\ggg})\otimes_{U({\mmm}_\mathfrak{l})}\mathbb{C}_\mathfrak{l}=U({\ggg})/I_\mathfrak{l}$, where $I_\mathfrak{l}$ is the ${\bbz}_2$-graded ideal of $U({\ggg})$ generated by all $x-\chi(x)$ with $x\in\mathfrak{m}_\mathfrak{l}$.

Define the $W$-superalgebra associated with isotropic subspace $\mathfrak{l}$ by
$$W_\mathfrak{l}:=(U({\ggg})/I_\mathfrak{l})^{\text{ad}\,{\mmm}'_\mathfrak{l}}\cong\{\bar x\in U({\ggg})/I_\mathfrak{l}\mid
[a,x]\in I_\mathfrak{l},~\forall a\in{\mmm}'_{\mathfrak{l}}\},$$ where $\bar x$ is the coset of $x\in U(\ggg)$, and the multiplication is given by $\bar x_1\cdot\bar x_2=\overline{x_1x_2}$ for $\bar x_1\, \bar x_2\in W_\mathfrak{l}$.

It is easy to observe that the refined $W$-superalgebra $W'_\chi$ in Definition \ref{rewcc} is the case with $\mathfrak{l}$ being chosen as a Lagrangian subspace of $\ggg(-1)$, and the algebra $W^{\text{fin}}(\ggg,e)$ in Definition \ref{quafi}
corresponds to the case with $\mathfrak{l}=0$.

Recall that in \cite[Theorem 4.1]{GG} Gan-Ginzburg showed that the definition of finite $W$-algebras is independent of the choice of an isotropic subspaces $\mathfrak{l}\subseteq\ggg'(-1)$ associated with a complex semisimple Lie algebra $\ggg'$.
For the super case, one may wonder whether this conclusion is also valid for the algebra $W_\mathfrak{l}$. In fact, for the case when $\sfr=\dim\ggg(-1)_{\bar{1}}$ is even,
Zhao has proved in \cite[Remark 3.11]{Z2} that the definition of the $W$-superalgebra $W_\mathfrak{l}$ is indeed independent of the choice of isotropic subspaces. For the case with $\sfr$ being odd, one needs to check whether the discussions in \cite[Lemma 3.9, Propostion 3.10, Theorem 4.6]{suh} go through for the $\mathbb{C}$-algebra $W_\mathfrak{l}$. If they are true, then Theorem \ref{important} follows just as a corollary.
\end{rem}

\section{The structure theory of minimal $W$-superalgebras}\label{structure}
This section is the main part of the paper. A refined $W$-superalgebra $W_\chi'$ will be called a minimal one if it is associated with a minimal nilpotent element in $\ggg$. In this section we will describe the explicit expression of the generators of minimal $W$-superalgebras over $\mathbb{C}$, and also the commutators between them. These generators and their relations completely determine the structure of minimal $W$-superalgebras. In virtue of these results, we will show that the dimensional lower bounds for the modular representations of reduced enveloping of basic Lie superalgebras associated with minimal nilpotent elements are accessible, completing the proof of Theorem \ref{intromainminnimalf}.

In \cite[Proposition 5.3]{suh} Suh has introduced the free generators of $W_\chi'$, and also the commutators between these generators in \cite[Proposition 5.4]{suh}. However, there are some errors in Suh's proof, thus also the second equation of (5.12) and the last equation of (5.17) in the results obtained by Suh. Moreover, the proof there is too brief to get the point. In order to determine the minimal dimension for the representations of reduced enveloping algebra $U_\chi(\ggg_{\mathds{k}})$, we must give an accurate description on the construction of corresponding refined $W$-superalgebra $W_\chi'$. In this section we will correct these errors and give the detailed calculation, mainly following Premet's strategy on finite $W$-algebras \cite[\S4]{P3}, with a few modifications.

This section is divided into two parts. The first part will be devoted to the general situation, where we consider refined $W$-superalgebras associated with arbitrary nilpotent elements, and the minimal $W$-superalgebras case will be dealt with in the second part.

Throughout this section, we will take $W^{\text{fin}}(\ggg,e)$ in Definition \ref{quafi} as the refined $W$-superalgebra $W_\chi'$ by Theorem \ref{important}.

\subsection{The case of arbitrary nilpotent elements}
In this part we will give the explicit formulae for the generators of lower Kazhdan degree for refined $W$-superalgebras associated with arbitrary nilpotent elements, and also some commutators between them.
\subsubsection{}\label{general}
Let ${\ggg}$ be a basic Lie superalgebra over ${\bbc}$ and $\mathfrak{h}$ be a typical Cartan subalgebra of ${\ggg}$. Let $\Phi$ be a root system of ${\ggg}$ relative to $\mathfrak{h}$ whose simple root system $\Delta=\{\alpha_1,\cdots,\alpha_l\}$ is distinguished. Let $\{e,h,f\}$ be an $\mathfrak{sl}_2$-triple in ${\ggg}$ such that
$$[e,f]=h\in\mathfrak{h},\quad[h,e]=2e,\quad[h,f]=-2f.$$
Normalize the invariant bilinear form $(\cdot,\cdot)$ on $\ggg$ by the condition $(e,f)=1$. This entails $(h,h)=2$.

Recall that there exists a symplectic (resp. symmetric) bilinear form $\langle\cdot,\cdot\rangle$ on the ${\bbz}_2$-graded subspace ${\ggg}(-1)_{\bar{0}}$ (resp. ${\ggg}(-1)_{\bar{1}}$) given by $\langle x,y\rangle=(e,[x,y])=\chi([x,y])$ for all $x,y\in{\ggg}(-1)_{\bar0}~(\text{resp.}\,x,y\in{\ggg}(-1)_{\bar1})$. Set $s=\text{dim}\,\ggg(-1)_{\bar0}$ (note that $s$ is an even number), and $\sfr=\text{dim}\,\ggg(-1)_{\bar1}$. In \S\ref{0.1.1} we choose $\mathbb{Z}_2$-homogenenous bases $\{u_1,\cdots,u_{s}\}$ of ${\ggg}(-1)_{\bar0}$ and $\{v_1,\cdots,v_\sfr\}$ of ${\ggg}(-1)_{\bar1}$  contained in ${\ggg}$ such that $\langle u_i, u_j\rangle =i^*\delta_{i+j,s+1}$ for $1\leqslant i,j\leqslant s$, where $i^*=\left\{\begin{array}{ll}-1&\text{if}~1\leqslant i\leqslant \frac{s}{2};\\ 1&\text{if}~\frac{s}{2}+1\leqslant i\leqslant s\end{array}\right.$, and $\langle v_i,v_j\rangle=\delta_{i+j,\sfr+1}$ for $1\leqslant i,j\leqslant \sfr$.
Set $z_\alpha:=u_\alpha$ for $1\leqslant\alpha\leqslant s$, and $z_{\alpha+s}:=v_\alpha$ for $1\leqslant\alpha\leqslant\sfr$. Let $\{z_\alpha\mid\alpha\in S(-1)\}$ denote the union of $\{z_\alpha\mid\alpha\in S(-1)_{\bar0}\}$ and $\{z_{\alpha+s}\mid\alpha\in S(-1)_{\bar1}\}$, which is a base of ${\ggg}(-1)$. Denote by $A_e$ the associative algebra generated by $z_\alpha$ with $\alpha\in S(-1)$ subject to the relations given above.

Set $z_\alpha^*:=\alpha^\natural z_{s+1-\alpha}$ for $1\leqslant\alpha\leqslant s$, where $\alpha^\natural=\left\{\begin{array}{ll}1&\text{if}~1\leqslant \alpha\leqslant \frac{s}{2};\\ -1&\text{if}~\frac{s}{2}+1\leqslant \alpha\leqslant s\end{array}\right.$, and $z_{\alpha+s}^*:=z_{\sfr+1-\alpha+s}$ for $1\leqslant\alpha\leqslant \sfr$. Then $\{z_\alpha^*\mid\alpha\in S(-1)\}$ is dual base of $\{z_\alpha\mid\alpha\in S(-1)\}$ such that $\langle z^*_\alpha,z_\beta\rangle =\delta_{\alpha,\beta}$ for $\alpha,\beta\in S(-1)$. From the assumption above we can conclude that $\langle z_\alpha,z_\beta\rangle\neq0$ if and only if $z_\alpha=\mathbb{C}z_\beta^*$ and $z_\beta=\mathbb{C}z_\alpha^*$, and $\langle z_\alpha,z_\beta\rangle=0$ otherwise. Moreover, $\langle z_\alpha,z_\beta\rangle=1$ if and only if $z_\alpha=z_\beta^*$ and $z_\beta=-(-1)^{|\alpha|}z_\alpha^*$. The same conclusion also establishes for the pair $(z_\alpha^*,z_\beta^*)$.  We can further assume that $\{z_\alpha\mid\alpha\in S(-1)\}$ and $\{z_\alpha^*\mid\alpha\in S(-1)\}$ are root vectors for $\mathfrak{h}$.

In the following discussion, for any $\alpha\in S(-1)$ we will denote the parity of $z_\alpha$ by $|\alpha|$ for simplicity. It is straightforward that $z_\alpha$ and $z_\alpha^*$ have the same parity.
\subsubsection{}
Let $Y_1,\cdots,Y_{l+q}$ be a basis of $\ggg^e$ as defined in \eqref{Y_i}. Recall that in Theorem \ref{PBWC} we choose $\Theta_{k}$ for $1\leqslant k\leqslant l+q$ as the generators of the refined $W$-superalgebra, with $Y_k$ being the leading term. In what follows we will introduce explicit formulae for the generators $\Theta_k$ of lower Kazhdan degree for a refined $W$-superalgebra $W_\chi'$.

For the case with $v\in\ggg^e(0)$, Suh introduced in \cite[Proposition 5.3]{suh} the corresponding generators $\Theta_v$ of $W_\chi'$ as follows:
\begin{prop}(\cite{suh})\label{v}
If $v\in\ggg^e(0)$ then it can be assumed that
\begin{equation*}
\Theta_v=(v-\frac{1}{2}\sum\limits_{\alpha\in S(-1)}z_\alpha[z_\alpha^*,v])\otimes1_\chi.
\end{equation*}
\end{prop}
For the generators of $W_\chi'$ introduced in Proposition \ref{v}, Suh introduced the commutators between them in \cite[Proposition 5.4]{suh}, i.e.,
\begin{prop}(\cite{suh})\label{v1v2}
Let $v_1, v_2$ be elements in $\ggg^e(0)$. The commutator between the generators introduced in Proposition \ref{v} is
\begin{equation}
[\Theta_{v_1},\Theta_{v_2}]=\Theta_{[v_1,v_2]}.
\end{equation}
\end{prop}
For the case with $w\in\ggg^e(1)$, Suh introduced the corresponding generators $\Theta_w$ of $W_\chi'$ in \cite[Proposition 5.3]{suh}. However, there are some errors in the proof there. Now we will put forward a new set of generators.
\begin{prop}\label{w}
If $w\in\ggg^e(1)$ then the generator $\Theta_w\in W_\chi'$ has the following property:
\begin{equation*}
\Theta_w=(w-\sum\limits_{\alpha\in S(-1)}z_\alpha[z_\alpha^*,w]+\frac{1}{3}(\sum\limits_{\alpha,\beta\in S(-1)}z_\alpha z_\beta[z_\beta^*,[z_\alpha^*,w]]-2[w,f]))\otimes1_\chi.
\end{equation*}
\end{prop}
\begin{proof}
By the definition of $W_\chi'$, it is enough to show that $[z_\gamma^*,\Theta_w]=0$ for any $\gamma\in S(-1)$. First note that
\begin{equation}\label{wxyuv}
\begin{split}
&\sum\limits_{\alpha\in S(-1)}[z_\gamma^*,z_\alpha[z_\alpha^*,w]]\otimes1_\chi\\
=&\sum\limits_{\alpha\in S(-1)}[z^*_\gamma,z_\alpha][z_\alpha^*,w]\otimes1_\chi
+\sum\limits_{\alpha\in S(-1)}(-1)^{|\alpha||\gamma|}z_\alpha[z_\gamma^*,[z_\alpha^*,w]]\otimes1_\chi\\
=&[z_\gamma^*,w]\otimes1_\chi+[f,[z_\gamma^*,w]]\otimes1_\chi+\sum\limits_{\alpha\in S(-1)}(-1)^{|\alpha||\gamma|}z_\alpha[z_\gamma^*,[z_\alpha^*,w]]\otimes1_\chi\\
=&[z_\gamma^*,w]\otimes1_\chi-[z_\gamma^*,[w,f]]\otimes1_\chi+\sum\limits_{\alpha\in S(-1)}(-1)^{|\alpha||\gamma|}z_\alpha[z_\gamma^*,[z_\alpha^*,w]]\otimes1_\chi.
\end{split}
\end{equation}

Recall that $[z_\alpha^*,z_\beta]=\delta_{\alpha,\beta}f$ for all $\alpha,\beta\in S(-1)$, and $\chi(f)=(e,f)=1$, then we have
\begin{equation}\label{3.5}
\begin{split}
&\sum\limits_{\alpha,\beta\in S(-1)}[z_\gamma^*,z_\alpha z_\beta[z_\beta^*,[z_\alpha^*,w]]]\otimes1_\chi\\
=&\sum\limits_{\alpha,\beta\in S(-1)}[z_\gamma^*,z_\alpha]z_\beta[z_\beta^*,[z_\alpha^*,w]]\otimes1_\chi
+\sum\limits_{\alpha,\beta\in S(-1)}(-1)^{|\alpha||\gamma|}z_\alpha[z_\gamma^*,z_\beta][z_\beta^*,[z_\alpha^*,w]]\otimes1_\chi\\
&+\sum\limits_{\alpha,\beta\in S(-1)}(-1)^{(|\alpha|+|\beta|)|\gamma|}z_\alpha z_\beta[z_\gamma^*,[z_\beta^*,[z_\alpha^*,w]]]\otimes1_\chi\\
=&\sum\limits_{\beta\in S(-1)}z_\beta[z_\beta^*,[z_\gamma^*,w]]\otimes1_\chi+\sum\limits_{\alpha\in S(-1)}(-1)^{|\alpha||\gamma|}z_\alpha[z_\gamma^*,[z_\alpha^*,w]]\otimes1_\chi\\
&+\sum\limits_{\alpha,\beta\in S(-1)}(-1)^{(|\alpha|+|\beta|)|\gamma|}z_\alpha z_\beta[z_\gamma^*,[z_\beta^*,[z_\alpha^*,w]]]\otimes1_\chi.
\end{split}
\end{equation}
Since $\langle z_\beta^*,z_\gamma^*\rangle\neq0$ if and only if $z_\beta=z_\gamma^*$ and $z_\gamma=-(-1)^{|\beta|}z_\beta^*$, and $\langle z_\beta^*,z_\gamma^*\rangle=1$ in this case, then
\begin{equation}\label{3.6}
\begin{split}
&\sum\limits_{\beta\in S(-1)}z_\beta[z_\beta^*,[z_\gamma^*,w]]\otimes1_\chi\\=&\sum\limits_{\beta\in S(-1)}(z_\beta[[z_\beta^*,z_\gamma^*],w]+(-1)^{|\beta||\gamma|}z_\beta[z_\gamma^*,[z_\beta^*,w]])\otimes1_\chi\\
=&-z_\gamma^*[w,f]\otimes1_\chi+\sum\limits_{\beta\in S(-1)}(-1)^{|\beta||\gamma|}z_\beta[z_\gamma^*,[z_\beta^*,w]]\otimes1_\chi.
\end{split}
\end{equation}
Moreover, it is obvious that $z=\sum\limits_{\alpha\in S(-1)}[z_\alpha^*,z]z_\alpha$ for all $z\in \ggg(-1)$,
thus
\begin{equation}\label{3.7}
\begin{split}
&\sum\limits_{\alpha,\beta\in S(-1)}(-1)^{(|\alpha|+|\beta|)|\gamma|}z_\alpha z_\beta[z_\gamma^*,[z_\beta^*,[z_\alpha^*,w]]]\otimes1_\chi\\
=&\sum\limits_{\alpha,\beta\in S(-1)}(-1)^{(|\alpha|+|\beta|)|\gamma|}(z_\alpha z_\beta[[z_\gamma^*,z_\beta^*],[z_\alpha^*,w]]+(-1)^{|\beta||\gamma|}z_\alpha z_\beta[z_\beta^*,[z_\gamma^*,[z_\alpha^*,w]]])\otimes1_\chi\\
=&\sum\limits_{\alpha\in S(-1)}(-(-1)^{|\alpha||\gamma|}z_\alpha z_\gamma^*[f,[z_\alpha^*,w]]+(-1)^{|\alpha||\gamma|}z_\alpha[z_\gamma^*,[z_\alpha^*,w]])\otimes1_\chi\\
=&\sum\limits_{\alpha\in S(-1)}((-1)^{|\alpha||\gamma|}z_\alpha z_\gamma^*[z_\alpha^*,[w,f]]+(-1)^{|\alpha||\gamma|}z_\alpha[z_\gamma^*,[z_\alpha^*,w]])\otimes1_\chi\\
=&\sum\limits_{\alpha\in S(-1)}(-1)^{|\alpha||\gamma|}([z_\alpha, z_\gamma^*][z_\alpha^*,[w,f]]+(-1)^{|\alpha||\gamma|}z_\gamma^*z_\alpha[z_\alpha^*,[w,f]]+z_\alpha[z_\gamma^*,[z_\alpha^*,w]])\otimes1_\chi\\
=&-[z_\gamma^*,[w,f]]\otimes1_\chi+z_\gamma^*[w,f]\otimes1_\chi+\sum\limits_{\alpha\in S(-1)}(-1)^{|\alpha||\gamma|}z_\alpha[z_\gamma^*,[z_\alpha^*,w]]\otimes1_\chi.
\end{split}
\end{equation}

By \eqref{wxyuv}, \eqref{3.5}, \eqref{3.6} and \eqref{3.7} we conclude that $[z_\gamma^*,\Theta_w]=0$ for any $\gamma\in S(-1)$, completing the proof.
\end{proof}
\begin{rem}
It is worth noting that similar result as in Proposition \ref{w} has been obtained by Suh in \cite[Proposition 5.3]{suh}. However, by careful inspection one can find that there are some errors in the first equation of \cite[(5.13)]{suh} and also the equation \cite[(5.16)]{suh}. Thus there is an error in the last part of the second equation of (5.12) in \cite[Proposition 5.3]{suh}, which is a little different from the consequence we obtained in Proposition \ref{w}.
\end{rem}
\subsubsection{}
Let $v\in\ggg^e(0)$ and $w\in\ggg^e(1)$. To determine the commutator between $\Theta_{v}$ and $\Theta_{w}$ defined in Proposition \ref{v} and Proposition \ref{w} respectively, we need to introduce an automorphism on $W_\chi'$.

Similar to the Lie algebra case in \cite[\S2.1]{GG}, we can define a linear action of $\mathbb{C}^*$ on $\ggg$. Given $\mathfrak{sl}_2$-triple $e,h,f\in\ggg_{\bar0}$, consider
the Lie algebra homomorphism $\mathfrak{sl}_2(\mathbb{C})\rightarrow\ggg_{\bar0}$ defined by
$$\left(\begin{array}{cc} 0 & 1\\0 & 0\end{array}\right)\mapsto e,
\qquad \left(\begin{array}{@{\hspace{0pt}}c@{\hspace{8pt}}c@{\hspace{0pt}}} 1 & 0\\0 & -1\end{array}\right)\mapsto h,
\qquad \left(\begin{array}{cc} 0 & 0\\1 & 0\end{array}\right)\mapsto f.$$
This Lie algebra homomorphism exponentiates to a rational homomorphism $$\tilde{\gamma}:~SL_2\rightarrow G_{\text{ev}}\rightarrow G.$$ We put
$$\gamma:\mathbb{C}^*\rightarrow G,\qquad \gamma(t)=\tilde{\gamma}
\left(\begin{array}{@{\hspace{0pt}}c@{\hspace{8pt}}c@{\hspace{0pt}}} t & 0\\0 & t^{-1}\end{array}\right),
\qquad \forall\,t\in\mathbb{C}^*.$$

Define $\sigma:=\text{Ad}\gamma(-1)$, an element of order $\leqslant2$ in $\text{Ad}\,G_{\text{ev}}$. Since the grading $\ggg=\bigoplus_{i\in \mathbb{Z}}\ggg(i)$ is obtained under the action of $\text{ad}\,h$, then for $x\in\ggg(i)$ we have $\sigma(x)=(-1)^{|i|}x$. Obviously $\sigma$ acts on $U(\ggg)$ as algebra automorphism. As $\sigma$ preserves the left ideal $I_\chi$ of $U(\ggg)$, it then acts on $Q_\chi$. As $\sigma$ preserves $\mathfrak{n}$ too, it acts on $W_\chi'\cong Q_\chi^{\text{ad}\,\mathfrak{n}}$ as algebra automorphism. One can easily conclude that
\begin{equation}\label{edegree}
\sigma(x^\mathbf{a}y^\mathbf{b}u^\mathbf{c}v^\mathbf{d}\otimes1_\chi)=(-1)^{|(\mathbf{a},\mathbf{b},\mathbf{c},\mathbf{d})|_e}
x^\mathbf{a}y^\mathbf{b}u^\mathbf{c}v^\mathbf{d}\otimes1_\chi
\end{equation}
for all $(\mathbf{a},\mathbf{b},\mathbf{c},\mathbf{d})\in{\bbz}_+^m\times\Lambda'_n\times{\bbz}_+^{s}\times\Lambda'_{\sfr}$.

We continue to consider the elements $\Theta_{v}$ and $\Theta_{w}$ for $v\in\ggg^e(0)$ and $w\in\ggg^e(1)$. In fact, we have the following result.
\begin{prop}\label{vw}
Let $v$ be an element in $\ggg^e(0)$ and $w$ be an element in $\ggg^e(1)$. Then the commutator between $\Theta_{v}$ and $\Theta_{w}$ in Proposition \ref{v} and Proposition \ref{w} is
\begin{equation}
[\Theta_{v},\Theta_{w}]=\Theta_{[v,w]}.
\end{equation}
\end{prop}
\begin{proof}

In virtue of \eqref{edegree}, we have $\sigma(\Theta_{v})=\Theta_{v}$ and $\sigma(\Theta_{w})=-\Theta_{w}$ by the definition of $\Theta_{v}$ and $\Theta_{w}$. Set $\Theta:=[\Theta_{v},\Theta_{w}]-\Theta_{[v,w]}$, an element in $W_\chi'$. Note that
\begin{equation*}
\begin{split}
\sigma(\Theta)=&\sigma([\Theta_{v},\Theta_{w}]-\Theta_{[v,w]})\\=&\sigma(\Theta_{v})\sigma(\Theta_{w})-(-1)^{|v||w|}
\sigma(\Theta_{w})\sigma(\Theta_{v})-\sigma(\Theta_{[v,w]})\\
=&-\Theta_{v}\Theta_{w}+(-1)^{|v||w|}\Theta_{w}\Theta_{v}+\Theta_{[v,w]}\\
=&-([\Theta_{v},\Theta_{w}]-\Theta_{[v,w]})\\=&-\Theta.
\end{split}
\end{equation*}
On the other hand, it follows from Theorem \ref{PBWC}(3) that $\Theta$ is a super-polynomial in $\Theta_{x_i}$ with $x_i\in\ggg^e(0)$. From all the discussion above, one knows that this super-polynomial must be zero. So $[\Theta_{v},\Theta_{w}]=\Theta_{[v,w]}$ necessarily holds, completing the proof.
\end{proof}
\begin{rem}
In fact, one can also prove Proposition \ref{vw} by computing the commutator between $\Theta_{v}$ and $\Theta_{w}$ for $v\in\ggg^e(0)$ and $w\in\ggg^e(1)$ directly. However, that proof is much complicate than the one we have applied above.
\end{rem}

\subsection{The case of minimal $W$-superalgebras}
Now we come to the case of minimal $W$-superalgebras. In this subsection the structure and the existence of $1$-dimensional representations of minimal $W$-superalgebras will be determined.

\subsubsection{}\label{3.1.1}
Retain the notations as in \S\ref{general}. A root $-\theta$ is called minimal if it is even and there exists an additive function $\varphi:\Phi\rightarrow\mathbb{R}$ such that $\varphi_{\mid{\Phi}}\neq0$ and $\varphi(\theta)>\varphi(\eta)$ for all $\eta\in\Phi\backslash\{\theta\}$. It is obvious that a minimal root $-\theta$ is the lowest root of one of the simple components of $\ggg_{\bar0}$ (in the ordering defined by $\varphi$).

Fix a minimal root $-\theta$ of $\ggg$. We may choose root vectors $e:=e_\theta$ and $f:=e_{-\theta}$ such
that $$[e,f]=h\in\mathfrak{h},\quad[h,e]=2e,\quad[h,f]=-2f.$$Then $e$ is a minimal nilpotent element in $\ggg$. Due to the minimality of $-\theta$, the eigenspace decomposition of $\text{ad}\,h$ gives rise to a short $\mathbb{Z}$-grading
$$\ggg=\ggg(-2)\oplus\ggg(-1)\oplus\ggg(0)\oplus\ggg(1)\oplus\ggg(2).$$
Moreover, $\ggg(2)=\mathbb{C}e, \ggg(-2)=\mathbb{C}f$.
We thus have a bijective correspondence between minimal gradings (up to an automorphism of $\ggg$) and minimal roots (up to the
action of the Weyl group). Furthermore, one has
$$\ggg^e=\ggg(0)^\sharp\oplus\ggg(1)\oplus\ggg(2),$$where $\ggg^e$ denotes the centralizer of $e$ in $\ggg$, and $\ggg(0)^\sharp=\{x\in\ggg(0)\mid[x,e]=0\}$.
Note that $\ggg(0)^\sharp$ is the centralizer of the triple $\{e,f,h\}$ by $\mathfrak{sl}_2$-theory.
Moreover, $\ggg(0)^\sharp$ is the orthogonal complement to $\mathbb{C}h$ in $\ggg(0)$, and coincides with the image of the Lie
superalgebra endomorphism
\begin{equation}\label{xh}
\sharp: \ggg(0)\rightarrow\ggg(0), x\mapsto x-\frac{1}{2}(h,x)h.
\end{equation} Obviously $\ggg(0)^{\sharp}$ is an ideal of codimensional $1$ in the Levi subalgebra $\ggg(0)$.

\subsubsection{The proof of Proposition \ref{ge}}\label{3.1.3}
Let $C$ denote the Casimir element of $U(\ggg)$ corresponding to the invariant form $(\cdot,\cdot)$. Since $C$ induces a $\ggg$-endomorphism of $Q_\chi$, it can be viewed as a central element of $W_\chi'$. Now we will determine the explicit formula of $C$.

Let $\{a_i\mid i\in I\}$ and $\{b_i\mid i\in I\}$ be dual bases of $\ggg^e(0)$ with respect to the restriction of the invariant form $(\cdot,\cdot)$ to $\ggg^e(0)$, then
$$\{e,h,f\}\cup\{a_i\mid i\in I\}\cup\{[e,z_\alpha^*]\mid\alpha\in S(-1)\}\cup\{z_\alpha\mid \alpha\in S(-1)\}$$
and
$$\{f,\frac{h}{2},e\}\cup\{b_i\mid i\in I\}\cup\{z_\alpha\mid \alpha\in S(-1)\}\cup\{(-1)^{|\alpha|}[e,z_\alpha^*]\mid\alpha\in S(-1)\}$$
are dual bases of $\ggg$ with respect to $(\cdot,\cdot)$. Since $a_i$ and $b_i$ have the same parity for  given $i\in I$, we will denote this parity by $|i|$. Recall that if $z_\alpha=z_\beta^*$ for some $\alpha,\beta\in S(-1)$, then the assumption on dual bases of $\ggg(-1)$ in \S\ref{general} implies that $z_\alpha^*=-(-1)^{|\beta|}z_\beta$. In virtue of this, we have\begin{equation}\label{zalphazalpha}
\sum\limits_{\alpha\in S(-1)}[[z_\alpha,e],z_\alpha^*]\otimes1_\chi=-\sum\limits_{\alpha\in S(-1)}(-1)^{|\alpha|}[[z_\alpha^*,e],z_\alpha]\otimes1_\chi.\end{equation}Recall that $s=\text{dim}\,\ggg(-1)_{\bar0}$ and $\sfr=\text{dim}\,\ggg(-1)_{\bar1}$. Since
\begin{equation}\label{3.9}
\begin{split}
\sum\limits_{\alpha\in S(-1)}[z_\alpha,[e,z_\alpha^*]]\otimes1_\chi
=&\sum\limits_{\alpha\in S(-1)}([[z_\alpha,e],z_\alpha^*]+[e,[z_\alpha,z_\alpha^*]])\otimes1_\chi\\
=&-\sum\limits_{\alpha\in S(-1)}(-1)^{|\alpha|}([[z_\alpha^*,e],z_\alpha]
+[e,[z_\alpha^*,z_\alpha]])\otimes1_\chi\\=&-\sum\limits_{\alpha\in S(-1)}[z_\alpha,[e,z_\alpha^*]]\otimes1_\chi+(\sfr-s)h\otimes1_\chi,
\end{split}
\end{equation} where the second equation in \eqref{3.9} follows from \eqref{zalphazalpha}. Then we have $$\sum\limits_{\alpha\in S(-1)}[z_\alpha,[e,z_\alpha^*]]\otimes1_\chi=\frac{(\sfr-s)}{2}h\otimes1_\chi.$$
As a consequence,
\begin{equation}\label{C}
\begin{split}
C=&(2e+\frac{h^2}{2}-h+\sum\limits_{i\in I}(-1)^{|i|}a_ib_i+\sum\limits_{\alpha\in S(-1)}(-1)^{|\alpha|}[e,z_\alpha^*]z_\alpha
+\sum\limits_{\alpha\in S(-1)}z_\alpha[e,z_\alpha^*])\otimes1_\chi\\
=&(2e+\frac{h^2}{2}-h+\sum\limits_{i\in I}(-1)^{|i|}a_ib_i+\sum\limits_{\alpha\in S(-1)}(-1)^{|\alpha|}[e,z_\alpha^*]z_\alpha
+\sum\limits_{\alpha\in S(-1)}[z_\alpha,[e,z_\alpha^*]]\\
&+\sum\limits_{\alpha\in S(-1)}(-1)^{|\alpha|}[e,z_\alpha^*]z_\alpha)\otimes1_\chi\\
=&(2e+\frac{h^2}{2}-(1+\frac{s-\sfr}{2})h+\sum\limits_{i\in I}(-1)^{|i|}a_ib_i
+2\sum\limits_{\alpha\in S(-1)}(-1)^{|\alpha|}[e,z_\alpha^*]z_\alpha)\otimes1_\chi.
\end{split}
\end{equation}

Combining Proposition \ref{v}, Proposition \ref{w} and the discussion in \S\ref{3.1.1}, we now obtain a generating set (as an associative superalgebra) of a refined $W$-superalgebra associated with a minimal nilpotent element as presented in Proposition \ref{ge}.

Now we can easily make an account of Proposition \ref{ge} as below.
\begin{proof}
The proposition follows from Theorem \ref{PBWC}, Proposition \ref{v}, Proposition \ref{w}, \eqref{C} and the discussion in \S\ref{3.1.1}.
\end{proof}

\subsubsection{}
In this part, we make further preparation for Theorem \ref{maiin1}, by determining the commutators between the generators of $W_\chi'$ given in Proposition \ref{ge}.

Recall that in \S\ref{3.1.1} and \S\ref{3.1.3} we showed that $\ggg^e(0)=\ggg(0)^\sharp$ is an ideal of codimensional $1$ in the Levi subalgebra $\ggg(0)$ of $\ggg$, with $\{a_i\mid i\in I\}$ and $\{b_i\mid i\in I\}$ being dual bases of $\ggg^e(0)$ with respect to the restriction of the invariant form $(\cdot,\cdot)$ to $\ggg^e(0)$. Now let $C_0:=\sum_{i\in I}a_ib_i$ be the corresponding Casimir element of $U(\ggg^e(0))$, and set $\Theta_{\text{Cas}}:=\sum_{i\in I}(-1)^{|i|}\Theta_{a_i}\Theta_{b_i}$ to be an element of $W_\chi'$. Although $\Theta_{\text{Cas}}$ is not central in $W_\chi'$, we have the following result.
\begin{prop}\label{vcommutate}
The element $\Theta_{\text{Cas}}$ commutes with all operators $\Theta_v$ for $v\in\ggg^e(0)$.
\end{prop}
\begin{proof}
Since $\{a_i\mid i\in I\}$ and $\{b_i\mid i\in I\}$ are dual bases of $\ggg^e(0)$ with respect to the invariant form $(\cdot,\cdot)$, we have
\begin{equation}\label{3.46}
\begin{split}
[a_i,v]=&\sum\limits_{j\in I}([a_i,v],b_j)a_j=\sum\limits_{j\in I}(a_i,[v,b_j])a_j\\
=&-\sum\limits_{j\in I}(-1)^{|v||j|}(a_i,[b_j,v])a_j,\\
[b_i,v]=&\sum\limits_{j\in I}(a_j,[b_i,v])b_j
\end{split}
\end{equation}for any $i\in I$.
Then one can obtain from  \eqref{3.46} and Proposition \ref{v1v2} that
\begin{equation}\label{3.47}
\begin{split}
&\Big[\sum\limits_{i\in I}(-1)^{|i|}\Theta_{a_i}\Theta_{b_i},\Theta_v\Big]\\
=&\sum\limits_{i\in I}((-1)^{(|v|+1)|i|}[\Theta_{a_i},\Theta_v]\Theta_{b_i}+(-1)^{|i|}\Theta_{a_i}[\Theta_{b_i},\Theta_v])\\
=&\sum\limits_{i\in I}((-1)^{(|v|+1)|i|}\Theta_{[{a_i},v]}\Theta_{b_i}+(-1)^{|i|}\Theta_{a_i}\Theta_{[{b_i},v]})\\
=&-\sum\limits_{i,j\in I}(-1)^{|v|(|i|+|j|)+|i|}(a_i,[b_j,v])\Theta_{a_j}\Theta_{b_i}
+\sum\limits_{i,j\in I}(-1)^{|i|}(a_j,[b_i,v])\Theta_{a_i}\Theta_{b_j}\\
=&-\sum\limits_{i,j\in I}(-1)^{|v|(|i|+|j|)+|j|}(a_j,[b_i,v])\Theta_{a_i}\Theta_{b_j}
+\sum\limits_{i,j\in I}(-1)^{|i|}(a_j,[b_i,v])\Theta_{a_i}\Theta_{b_j}.
\end{split}
\end{equation}

As the bilinear form $(\cdot,\cdot)$ is even, the first term of the last equation in \eqref{3.47} is nonzero if and only if
\begin{equation*}
|v|(|i|+|j|)+|j|=(|i|+|j|)^2+|j|=|i|+|j|+|j|=|i|.
\end{equation*}
Combining this with \eqref{3.47}, we have
$$\Big[\sum\limits_{i\in I}(-1)^{|i|}\Theta_{a_i}\Theta_{b_i},\Theta_v\Big]=0,$$completing the proof.
\end{proof}

Now we can calculate the remaining commutators of the generators of a minimal $W$-superalgebra $W_\chi'$ in Proposition \ref{ge}.

  First note that the Casimir element $C$ is a central element of $W_\chi'$, then we have
\begin{equation}\label{casimircom}
[C,W_\chi']=0.
\end{equation}
Then we have finished the arguments for all the items of relations in  Theorem \ref{maiin1} except (3). Next we deal with the exceptional item.

\begin{prop}\label{1commutator}
Let $w_1, w_2\in\ggg^e(1)$. Then the following relation holds in $W_\chi'$:
\begin{equation}
\begin{split}
[\Theta_{w_1},\Theta_{w_2}]=&\frac{1}{2}([w_1,w_2],f)(C-\Theta_{\text{Cas}}-c_0)-\frac{1}{2}\sum\limits_{\alpha\in S(-1)}(\Theta_{[w_1,z_\alpha]^{\sharp}}\Theta_{[z_\alpha^*, w_2]^{\sharp}}\\
&-(-1)^{|w_1||w_2|}\Theta_{[w_2,z_\alpha]^{\sharp}}\Theta_{[z_\alpha^*, w_1]^{\sharp}}),
\end{split}
\end{equation}
where $c_0\in\mathbb{C}$ is a constant depending on $\ggg$. To be explicit,
the constant $c_0$ satisfies the following equation:
\begin{equation*}
\begin{split}
c_0([w_1,w_2],f)=&\frac{1}{12}\sum\limits_{\alpha,\beta\in S(-1)}(-1)^{|\alpha||w_1|+|\beta||w_1|+|\alpha||\beta|}\otimes[[z_\beta,[z_\alpha,w_1]],[z_\beta^*,[z_\alpha^*,w_2]]]\\
&-\frac{3(s-\sfr)+4}{12}([w_1,w_2],f),
\end{split}
\end{equation*}
where $s=\text{dim}\,\ggg(-1)_{\bar0}$ and $\sfr=\text{dim}\,\ggg(-1)_{\bar1}$.
\end{prop}
Since the proof of Proposition \ref{1commutator} is rather lengthy, we will postpone it till \S\ref{proof}.
\begin{rem}
Compared with the result obtained by Suh in \cite[Proposition 5.4]{suh}, one can find that there are subtle differences between the last equation of (5.17) obtained by Suh in \cite{suh} and the one we introduced in Theorem \ref{maiin1}(3), i.e., the term $-\frac{1}{2}\sum\limits_{\alpha\in S(-1)}\Theta_{[w_1,z_\alpha]^{\sharp}}\Theta_{[z_\alpha^*, w_2]^{\sharp}}$ in Theorem \ref{maiin1}(3) did not appear in \cite[Proposition 5.4]{suh}. One reason may lie in the fact that the expression in the second equation of \cite[(5.12)]{suh} is error. On the other hand, from the detailed calculation in the proof of \eqref{3.57}, one will observe that without the term we mentioned above, the term $\frac{1}{4}\sum\limits_{\alpha\in S(-1)}(-1)^{|\alpha||w_2|}[[w_1,z_\alpha],[w_2,z_\alpha^*]]\otimes1$ (which is not a constant in general) cannot be eliminated  by other terms.
\end{rem}

We are now in a position to prove the main results of the present paper.

\subsubsection{The Proof of Theorem \ref{maiin1}} \label{5.2.4}
This theorem readily follows from Proposition \ref{ge}, Proposition \ref{v1v2}, Proposition \ref{vw}, \eqref{casimircom}, Proposition \ref{1commutator}, and finally Theorem \ref{PBWC}.

\subsubsection{The Proof of Theorem \ref{intromainminnimalf}}\label{5.2.5}

In virtue of the results we obtained above, we will show that Conjecture \ref{conjecture22} is true for minimal $W$-superalgebras, and the lower bounds for the representations of basic Lie superalgebras with minimal nilpotent $p$-characters in Theorem \ref{intromain-2} are accessible. We mainly follow the method used in \cite[Corollary 4.1]{P3} for the finite $W$-algebra case.

Let $\ggg$ be a basic Lie superalgebra over $\mathbb{C}$, and $e$ be a minimal nilpotent element in $\ggg$. Denote by $\ggg^e$ the centralizer of $e$ in $\ggg$. Set $\{x_1,\cdots,x_m\}$, $\{y_1,\cdots,y_n\}$, $\{u_1,\cdots,u_{s}\}$ and $\{v_1,\cdots,v_t\}$ to be bases of $\ggg^e(0)_{\bar0}$, $\ggg^e(0)_{\bar1}$, $\ggg^e(1)_{\bar0}$ and $\ggg^e(1)_{\bar1}$, respectively. Let $W_\chi'$ be the minimal $W$-superalgebra associated with the pair ($\ggg$, $e$). Denote by $(W_\chi')^+$ the $\mathbb{C}$-span of the monomials
$$\Theta_{x_1}^{i_1}\cdots\Theta_{x_m}^{i_m}\cdot
\Theta_{y_1}^{j_1}\cdots\Theta_{y_n}^{j_n}\cdot
\Theta_{u_1}^{k_1}\cdots\Theta_{u_{s}}^{k_{s}}\cdot
\Theta_{v_1}^{l_1}\cdots\Theta_{v_t}^{l_t}\cdot(C-c_0)^q$$
with $i_r, k_r, q\in\mathbb{Z}_+$ and $j_r, l_r\in\{0,1\}$ such that $\sum i_r+\sum j_r+\sum k_r+\sum l_r+q\geqslant1$. It follows from Theorem \ref{PBWC}(2) that $(W_\chi')^+$ is a subspace of codimensional $1$ in $W_\chi'$. We now prove that $(W_\chi')^+$ is a two-sided ideal of $W_\chi'$, i.e.,
\begin{prop}\label{minimal nilpotent}
The subspace $(W_\chi')^+$ is a two-sided ideal of codimension $1$ in the refined $W$-superalgebra $W_\chi'$ associated with minimal nilpotent element $e$. Moreover, if $\ggg$ is not of type $A(m|n)$ with $m, n\geqslant0$, then $(W_\chi')^+$ is the only ideal of codimensional $1$ in $W_\chi'$.
\end{prop}
\begin{proof}
We first show that $h\cdot h'\in (W_\chi')^+$ for all $h, h'\in (W_\chi')^+$. Since $C-c_0$ is in the center of $W_\chi'$, we have $(C-c_0)\cdot (W_\chi')^+\subseteq (W_\chi')^+$. Therefore, we just need to prove that $\Theta_x\cdot (W_\chi')^+\subseteq(W_\chi')^+$ for all $x\in\ggg^e(0)\cup\ggg^e(1)$.

For the case with $x\in\ggg^e(0)$, we can conclude from Proposition \ref{v1v2} that the span of all $\Theta_{x_1}^{i_1}\cdots\Theta_{x_m}^{i_m}\cdot\Theta_{y_1}^{j_1}\cdots\Theta_{y_n}^{j_n}$ with  $i_r\in\mathbb{Z}_+, j_r\in\{0,1\}$ and $\sum i_r+\sum j_r\geqslant1$ is stable under the left multiplications of $\Theta_{x}$ by $x\in\ggg^e(0)$. Now we assume that $x\in\ggg^e(1)$. Since the monomials $\Theta_{x_1}^{i_1}\cdots\Theta_{x_m}^{i_m}\cdot
\Theta_{y_1}^{j_1}\cdots\Theta_{y_n}^{j_n}\cdot
\Theta_{u_1}^{k_1}\cdots\Theta_{u_{s}}^{k_{s}}\cdot
\Theta_{v_1}^{l_1}\cdots\Theta_{v_t}^{l_t}\cdot(C-c_0)^q$ with $i_r, k_r\in\mathbb{Z}_+$ and $j_r, l_r\in\{0,1\}$ form a basis of the refined $W$-superalgebra $W_\chi'$, by Theorem \ref{PBWC} and Proposition \ref{vw} it suffices to prove that $\Theta_{x}\cdot\Theta_{u_1}^{k_1}\cdots\Theta_{u_{k_M}}^{k_M}\cdot
\Theta_{v_1}^{l_1}\cdots\Theta_{v_N}^{l_N}\in (W_\chi')^+$ for all $k_1,\cdots,k_M\in\{1,\cdots,s\}$, $l_1,\cdots,l_N\in\{1,\cdots,t\}$ and all $x\in\ggg^e(1)$. In fact, in view of Proposition \ref{vw} and Proposition \ref{1commutator} we can reach this conclusion by induction on $M$ and $N$.

Let $\ggg$ be not of type $A(m|n)$ with $m, n\geqslant0$, then it follows from \cite[Table 1-Table 3]{KW} that $\ggg^e(0)$ is a semisimple Lie (super)algebra (which is denoted by $\ggg^\natural$ in the settings there). Let $I$ be any ideal of codimensional $1$ in the refined $W$-superalgebra $W_\chi'$. One can easily verify that $[x,y]=xy-(-1)^{|x||y|}yx\in I$ for all $x, y\in W_\chi'$. Since $\ggg^e(0)=[\ggg^e(0),\ggg^e(0)]$ by the above remark, Proposition \ref{v1v2} implies that $\Theta_x\in I$ for $x\in\ggg^e(0)$. Moreover, \cite[Table 1-Table 3]{KW} also showed that $\ggg^e(1)=\ggg(1)$ (which is denoted by $\ggg_{\frac{1}{2}}$ there) is an irreducible $\ggg^e(0)$-module, then we have $[\ggg^e(0),\ggg^e(1)]=\ggg^e(1)$. Thus $\Theta_x\in I$ for all $x\in \ggg^e(1)$ by Proposition \ref{vw}. As $I$ is a subalgebra of $W_\chi'$ containing $\ggg^e(0)\cup\ggg^e(1)$, Proposition \ref{1commutator} implies that $C-c_0$ also in $I$. All the above shows that $(W_\chi')^+\subseteq I$ by definition. Since the codimension of $(W_\chi')^+$ is $1$ in $W_\chi'$, we conclude that $I=(W_\chi')^+$, as desired.
\end{proof}

It is immediate from Proposition \ref{minimal nilpotent} that Conjecture \ref{conjecture22} is true for minimal $W$-superalgebras.
As a direct corollary of Theorem \ref{intromain-2} and Proposition \ref{minimal nilpotent}, we complete the proof of Theorem \ref{intromainminnimalf}.

\section{The proof of Proposition \ref{1commutator}}\label{proof}
This section is contributed to the proof of Proposition \ref{1commutator}. We mainly follow Premet's strategy on finite $W$-algebras \cite[\S4]{P3}, with a few modifications.

 In the first two subsections, we will make some necessary preparation. To simplify calculation, in \S\ref{improve} we rewrite  generators of minimal $W$-superalgebras appearing in Proposition \ref{ge} in another way. Then in \S\ref{6.2} we will present relations of the generators  by lots of computations. The final proof of Proposition \ref{1commutator} will be given in \S\ref{6.4}.

Retain all the notations as in \S\ref{structure} throughout this section.

\subsection{}\label{improve} In this subsection,  we first express  the generators $\Theta_w$ with $w\in\ggg^e(1)$  of $W_\chi'$ in Proposition \ref{ge} in some more computable way, which will be fulfilled  in (\ref{rewrite}) via Proposition \ref{translate}.

Let us begin with the following equation
\begin{equation}\label{zzwzzw}
\sum\limits_{\alpha\in S(-1)}[z_\alpha,[z_\alpha^*,w]]\otimes1_\chi=-\sum\limits_{\alpha\in S(-1)}(-1)^{|\alpha|}[z_\alpha^*,[z_\alpha,w]])\otimes1_\chi,
\end{equation} then we have
\begin{equation*}\label{zzw}
\begin{split}
\sum\limits_{\alpha\in S(-1)}[z_\alpha,[z_\alpha^*,w]]\otimes1_\chi
&=\sum\limits_{\alpha\in S(-1)}([[z_\alpha,z_\alpha^*],w]+(-1)^{|\alpha|}[z_\alpha^*,[z_\alpha,w]])\otimes1_\chi\\
&=((s-\sfr)[w,f]-\sum\limits_{\alpha\in S(-1)}[z_\alpha,[z_\alpha^*,w]])\otimes1_\chi,
\end{split}
\end{equation*} thus $\sum\limits_{\alpha\in S(-1)}[z_\alpha,[z_\alpha^*,w]]\otimes1_\chi=\frac{(s-\sfr)}{2}[w,f]\otimes1_\chi$.
It follows that
\begin{equation*}
\begin{split}
\sum\limits_{\alpha\in S(-1)}z_\alpha[z_\alpha^*,w]\otimes1_\chi&=\sum\limits_{\alpha\in S(-1)}([z_\alpha,[z_\alpha^*,w]]+(-1)^{|w||\alpha|+|\alpha|}[z_\alpha^*,w]z_\alpha)\otimes1_\chi\\
&=\Big(\frac{(s-\sfr)}{2}[w,f]-\sum\limits_{\alpha\in S(-1)}(-1)^{|\alpha|}[w,z_\alpha^*]z_\alpha\Big)\otimes1_\chi.
\end{split}
\end{equation*}

Define \begin{equation}\label{varphi}
\varphi_w:=\frac{1}{3}(\sum\limits_{\alpha,\beta\in S(-1)}z_\alpha z_\beta[z_\beta^*,[z_\alpha^*,w]]-\frac{3(s-\sfr)+4}{2}[w,f]),\end{equation}
then we can write
\begin{equation}\label{anotherw}
\Theta_w=(w+\sum\limits_{\alpha\in S(-1)}(-1)^{|\alpha|}[w,z_\alpha^*]z_\alpha+\varphi_w)\otimes1_\chi.
\end{equation}

Now we come to the case with $w_1, w_2\in\ggg^e(1)$. To simplify calculation, we need the following settings.

Since each vector $h\in W_\chi'$ can be uniquely expressed as $h=\sum\limits_{(\mathbf{i},\mathbf{j})\in\mathbb{Z}_+^{s}
\times\Lambda'_{\sfr}}x_\mathbf{i}y_\mathbf{j}u^\mathbf{i}v^\mathbf{j}\otimes1_\chi$ with $x_\mathbf{i}y_\mathbf{j}\in U(\mathfrak{p})$, one can define a natural linear injection
\begin{equation*}
\begin{array}{llll}
\mu:&W_\chi'&\longrightarrow&U(\mathfrak{p})\otimes A_e^{\text{op}}\\
&h&\mapsto&\sum\limits_{(\mathbf{i},\mathbf{j})\in\mathbb{Z}_+^{s}\times\Lambda'_{\sfr}}x_\mathbf{i}y_\mathbf{j}\otimes u^\mathbf{i}v^\mathbf{j},
\end{array}
\end{equation*}where $A_e^{\text{op}}$ denotes the opposite algebra of $A_e$ as in \S\ref{general}.

Obviously the mapping $\mu$ is injective. Moreover, we have the following result.
\begin{prop}\label{translate}
The map $\mu: W_\chi' \longrightarrow U(\mathfrak{p})\otimes A_e^{\text{op}}$ is an algebra homomorphism.
\end{prop}
\begin{proof}
Denote by $\mathscr{Z}$ the linear span of all $u^\mathbf{i}v^\mathbf{j}$ with $(\mathbf{i},\mathbf{j})\in\mathbb{Z}_+^{s}\times\Lambda'_{\sfr}$. Then we can identify $\mathscr{Z}$ with the span of the left regular representation of $A_e$ via $u^\mathbf{i}v^\mathbf{j}\otimes1_\chi\mapsto u^\mathbf{i}v^\mathbf{j}$. Let $\rho_\chi$ denote the left regular representation of $U(\mathfrak{n})$ in $\text{End}(\mathscr{Z})$. Recall that we have assumed that $\{z_\alpha\mid\alpha\in S(-1)\}=\{u_\alpha\mid\alpha\in S(-1)_{\bar0}\}\bigcup\{v_{\alpha}\mid\alpha\in S(-1)_{\bar1}\}$ in \S\ref{general}. As $\mathfrak{g}(-1)\subseteq\mathfrak{n}$, and $\mathfrak{g}(i)\subseteq\text{Ker}\,\chi$ for all $i\leqslant-3$, the definition of $W_\chi'$ and induction on $k$ show that
\begin{equation}\label{uvz}
\rho_\chi(z_1\cdots z_k)(h)=(-1)^{|h|(|z_1|+\cdots+|z_k|)}\sum\limits_{(\mathbf{i},\mathbf{j})\in\mathbb{Z}_+^{s}
\times\Lambda'_{\sfr}}x_\mathbf{i}y_\mathbf{j}\cdot\rho_\chi(u^\mathbf{i}v^\mathbf{j}\cdot z_1\cdots z_k)(1_\chi)
\end{equation}
for all $z_1,\cdots,z_k\in\mathfrak{g}(-1)$.

Let $h'$ be another element in $W_\chi'$ such that $h'=\sum\limits_{(\mathbf{i},\mathbf{j})\in\mathbb{Z}_+^{s}
\times\Lambda'_{\sfr}}
x'_\mathbf{i}y'_\mathbf{j}(u')^{\mathbf{i}}(v')^{\mathbf{j}}\otimes1_\chi$ with $x'_\mathbf{i}y'_\mathbf{j}\in U(\mathfrak{p})$, then it follows from \eqref{uvz} that \[
\begin{split}
h\cdot h'=&\sum\limits_{(\mathbf{i},\mathbf{j})}
\rho_\chi(x_\mathbf{i}y_\mathbf{j})\cdot\rho_\chi(u^{\mathbf{i}}v^{\mathbf{j}})(h')\\
=&\sum\limits_{(\mathbf{i},\mathbf{j})}\sum\limits_{(\mathbf{k},\mathbf{l})}(-1)^{|h'||\mathbf{j}|}x_\mathbf{i}y_\mathbf{j}\cdot
x'_\mathbf{k}y'_\mathbf{l}\cdot\rho_\chi((u')^\mathbf{k}(v')^\mathbf{l}\cdot u^{\mathbf{i}}v^{\mathbf{j}})\otimes1_\chi,
\end{split}\]
thus
\begin{equation*}
\mu(h\cdot h')=\sum\limits_{(\mathbf{i},\mathbf{j})}\sum\limits_{(\mathbf{k},\mathbf{l})}(-1)^{|h'||\mathbf{j}|}x_\mathbf{i}y_\mathbf{j}\cdot
x'_\mathbf{k}y'_\mathbf{l}\otimes(u')^\mathbf{k}(v')^\mathbf{l}\cdot u^{\mathbf{i}}v^{\mathbf{j}}.
\end{equation*}
It remains to note that the map $u^\mathbf{i}v^\mathbf{j}\otimes1_\chi\mapsto u^\mathbf{i}v^\mathbf{j}$ mentioned above identifies $\rho_\chi(U(\mathfrak{n}))$ with the image of $A_e$ in its left regular representation, then the proof is completed.
\end{proof}

In virtue of Proposition \ref{translate}, we can rewrite the generators of $W_\chi'$ in Proposition \ref{ge} and \eqref{anotherw} as follows:
\begin{equation}\label{rewrite}
\begin{split}
\Theta_v=&v\otimes1-\sum\limits_{\alpha\in S(-1)}\frac{1}{2}\otimes z_\alpha[z_\alpha^*,v],\\
\Theta_w=&w\otimes1+\sum\limits_{\alpha\in S(-1)}(-1)^{|\alpha|}[w,z_\alpha^*]\otimes z_\alpha+1\otimes\varphi_w,\\
C=&2e\otimes1+\frac{h^2}{2}\otimes1-(1+\frac{s-r}{2})h\otimes1+\sum\limits_{i\in I}(-1)^{|i|}a_ib_i\otimes1\\
&+2\sum\limits_{\alpha\in S(-1)}(-1)^{|\alpha|}[e,z_\alpha^*]\otimes z_\alpha,
\end{split}
\end{equation} where $\{z_\alpha^*\mid\alpha\in S(-1)\}$ and $\{z_\alpha\mid\alpha\in S(-1)\}$ are dual bases of $\ggg(-1)$ with respect to $\langle\cdot,\cdot\rangle$, $\{a_i\mid i\in I\}$ and $\{b_i\mid i\in I\}$ be dual bases of $\ggg^e(0)$ with respect $(\cdot,\cdot)$, and $\varphi_w=\frac{1}{3}(\sum\limits_{\alpha,\beta\in S(-1)}z_\alpha z_\beta[z_\beta^*,[z_\alpha^*,w]]-\frac{3(s-\sfr)+4}{2}[w,f])$.

\subsection{}\label{6.2} In this  subsection, we will present the relations of the generators described just above, which will be given in (\ref{3.57}). The process will be long and tedious.

First, notice that in $U(\mathfrak{p})\otimes A_e^{\text{op}}$ we have
\begin{equation}\label{aop}
\begin{split}
[a\otimes f,b\otimes g]&=(-1)^{|b||f|+|f||g|}ab\otimes gf-(-1)^{(|a|+|f|)(|b|+|g|)+|a||g|+|f||g|}ba\otimes fg\\
&=(-1)^{|f|(|b|+|g|)}[a,b]\otimes gf-(-1)^{|b|(|a|+|f|)}ba\otimes[f,g]
\end{split}
\end{equation}
for all $a,b\in U(\mathfrak{p})$ and $f,g\in A_e$.
For all $w_1,w_2\in\ggg^e(1)$, we will compute the commutators between $\Theta_{w_1}$ and $\Theta_{w_2}$ in this subsection.
\subsubsection{}
Keeping \eqref{aop} in mind, we have
\begin{equation}\label{theta12}
\begin{split}
[\Theta_{w_1},\Theta_{w_2}]
=&[w_1\otimes1+\sum\limits_{\alpha\in S(-1)}(-1)^{|\alpha|}[w_1,z_\alpha^*]\otimes z_\alpha+1\otimes\varphi_{w_1},w_2\otimes1\\
&+\sum\limits_{\alpha\in S(-1)}(-1)^{|\alpha|}[w_2,z_\alpha^*]\otimes z_\alpha+1\otimes\varphi_{w_2}]\\
=&[w_1,w_2]\otimes1+\sum\limits_{\alpha\in S(-1)}(-1)^{|\alpha|+|\alpha||w_2|}[[w_1,z_\alpha^*],w_2]\otimes z_\alpha\\
&+\sum\limits_{\alpha\in S(-1)}(-1)^{|\alpha|}[w_1,[w_2,z_\alpha^*]]\otimes z_\alpha\\
&+\sum\limits_{\alpha,\beta\in S(-1)}(-1)^{|\alpha|+|\beta|+|\alpha||w_2|}[[w_1,z_\alpha^*],[w_2,z_\beta^*]]\otimes z_\beta z_\alpha\\&-\sum\limits_{\alpha,\beta\in S(-1)}(-1)^{|\alpha|+|\beta|+|\beta||w_1|+|w_1||w_2|}[w_2,z_\beta^*][w_1,z_\alpha^*]\otimes[z_\alpha,z_\beta]\\
&-\sum\limits_{\alpha\in S(-1)}(-1)^{|\alpha|+|\alpha||w_1|+|w_1||w_2|}[w_2,z_\alpha^*]\otimes[\varphi_{w_1},z_\alpha]\\
&-\sum\limits_{\alpha\in S(-1)}(-1)^{|\alpha|}[w_1,z_\alpha^*]\otimes[z_\alpha,\varphi_{w_2}]-1\otimes[\varphi_{w_1},\varphi_{w_2}]\\
=&[w_1,w_2]\otimes1+\sum\limits_{\alpha\in S(-1)}(-1)^{|\alpha|}[[w_1,w_2],z_\alpha^*]\otimes z_\alpha\\
&+\sum\limits_{\alpha,\beta\in S(-1)}(-1)^{|\alpha|+|\beta|+|\alpha||w_2|}[[w_1,z_\alpha^*],[w_2,z_\beta^*]]\otimes z_\beta z_\alpha\\
&+\sum\limits_{\alpha\in S(-1)}(-1)^{|\alpha|+|\alpha||w_1|+|w_1||w_2|}[w_2,z_\alpha^*][w_1,z_\alpha]\otimes1\\
&-\sum\limits_{\alpha\in S(-1)}(-1)^{|w_1||w_2|}
[w_2,z_\alpha]\otimes[z_\alpha^*,\varphi_{w_1}]+\sum\limits_{\alpha\in S(-1)}[w_1,z_\alpha]\otimes[z_\alpha^*,\varphi_{w_2}]\\
&-1\otimes[\varphi_{w_1},\varphi_{w_2}].
\end{split}
\end{equation}

For any $y\in\ggg^e(1)$ and $\alpha\in S(-1)$, it follows from the definition of $\varphi_y$ and \eqref{3.5}, \eqref{3.6}, \eqref{3.7} that
\begin{equation}\label{zw}
\begin{split}
[z_\alpha^*,\varphi_y]=&[z_\alpha^*,\frac{1}{3}(\sum\limits_{\beta,\gamma\in S(-1)}z_\beta z_\gamma[z_\gamma^*,[z_\beta^*,y]]-\frac{3(s-\sfr)+4}{2}[y,f])]\\
=&\frac{\sfr-s-2}{2}[z_\alpha^*,[y,f]]+\sum\limits_{\beta\in S(-1)}(-1)^{|\alpha||\beta|}z_\beta[z_\alpha^*,[z_\beta^*,y]].
\end{split}
\end{equation}
It is worth noting that for any $u\in\ggg(-1)$, we have
\begin{equation}\label{zuz}
u=\sum\limits_{\alpha\in S(-1)}[z_\alpha^*,u]z_\alpha=-\sum\limits_{\alpha\in S(-1)}(-1)^{|\alpha|}[z_\alpha,u]z_\alpha^*.
\end{equation} One can conclude from \eqref{zzwzzw}, \eqref{zw} and \eqref{zuz} that
\begin{equation}\label{xzzy}
\begin{split}
&\sum\limits_{\alpha\in S(-1)}[x,z_\alpha]\otimes[z_\alpha^*,\varphi_y]\\
=&\sum\limits_{\alpha,\beta\in S(-1)}(-1)^{|\alpha||\beta|}[x,z_\alpha]\otimes z_\beta[z_\alpha^*,[z_\beta^*,y]]+\frac{\sfr-s-2}{2}\sum\limits_{\alpha\in S(-1)}[x,z_\alpha]\otimes[z_\alpha^*,[y,f]]\\
=&\sum\limits_{\alpha,\beta,\gamma\in S(-1)}(-1)^{|\alpha||\beta|}[x,z_\alpha]\otimes z_\beta [z_\gamma^*,[z_\alpha^*,[z_\beta^*,y]]]z_\gamma\\&+\frac{\sfr-s-2}{2}\sum\limits_{\alpha\in S(-1)}[x,z_\alpha]\otimes[z_\alpha^*,[y,f]]\\
=&\sum\limits_{\alpha,\beta,\gamma\in S(-1)}(-1)^{|\alpha||\beta|}[x,z_\alpha]\otimes z_\beta [[z_\gamma^*,z_\alpha^*],[z_\beta^*,y]]z_\gamma\\
&+\sum\limits_{\alpha,\beta,\gamma\in S(-1)}(-1)^{|\alpha|(|\beta|+|\gamma|)}[x,z_\alpha]\otimes z_\beta[z_\alpha^*,[z_\gamma^*,[z_\beta^*,y]]]z_\gamma\\&+\frac{\sfr-s-2}{2}\sum\limits_{\alpha\in S(-1)}[x,z_\alpha]\otimes[z_\alpha^*,[y,f]]\\
=&\sum\limits_{\alpha\in S(-1)}(-1)^{|\alpha||y|+|\alpha|}[x,z_\alpha^*]\otimes[y,f]z_\alpha\\
&+\sum\limits_{\beta,\gamma\in S(-1)}(-1)^{(|\beta|+|\gamma|)(1+|y|)}[x,[z_\gamma^*,[z_\beta^*,y]]]\otimes z_\beta z_\gamma\\
&+\frac{\sfr-s-2}{2}\sum\limits_{\alpha\in S(-1)}[x,z_\alpha]\otimes[z_\alpha^*,[y,f]]\\
=&-\sum\limits_{\alpha\in S(-1)}[x,z_\alpha]\otimes z_\alpha^*[y,f]+\sum\limits_{\beta,\gamma\in S(-1)}(-1)^{(|\beta|+|\gamma|)(1+|y|)}[x,[z_\gamma^*,[z_\beta^*,y]]]\otimes z_\beta z_\gamma\\
&+\frac{\sfr-s}{2}\sum\limits_{\alpha\in S(-1)}[x,z_\alpha]\otimes[z_\alpha^*,[y,f]]\\
=&-\sum\limits_{\alpha\in S(-1)}[x,z_\alpha]\otimes z_\alpha^*[y,f]+\sum\limits_{\beta,\gamma\in S(-1)}(-1)^{(|\beta|+|\gamma|)(1+|y|)}[x,[z_\gamma^*,[z_\beta^*,y]]]\otimes z_\beta z_\gamma\\
&+\frac{\sfr-s}{2}[x,[y,f]]\otimes1
\end{split}
\end{equation} for all $x,y\in\ggg^e(1)$.

It is worth noting  that the assumption in \S\ref{general} shows that the following relations hold in $U(\ggg)$:
\begin{equation}\label{zz*}
\begin{array}{lllll}
\sum\limits_{\alpha\in S(-1)_{\bar0}}z_\alpha z_\alpha^*&=-\sum\limits_{\alpha\in S(-1)_{\bar0}}z_\alpha^* z_\alpha&\equiv -\frac{\text{dim}\,\ggg(-1)_{\bar0}}{2}&=-\frac{s}{2}&(\text{mod}\,I_\chi),\\
\sum\limits_{\alpha\in S(-1)_{\bar1}}z_\alpha z_\alpha^*&=\sum\limits_{\alpha\in S(-1)_{\bar1}}z_\alpha^* z_\alpha&\equiv \frac{\text{dim}\,\ggg(-1)_{\bar1}}{2}&=\frac{\sfr}{2}&(\text{mod}\,I_\chi).
\end{array}
\end{equation}

Now we consider the second term of the last equation in \eqref{xzzy}.
In virtue of \eqref{zz*}, we can obtain
\begin{equation}\label{wzzw}
\begin{split}
&\sum\limits_{\alpha,\beta\in S(-1)}((-1)^{(|\alpha|+|\beta|)(1+|w_2|)}[w_1,[z_\beta^*,[z_\alpha^*,w_2]]]\otimes z_\alpha z_\beta\\
&-(-1)^{(|\alpha|+|\beta|)(1+|w_1|)+|w_1||w_2|}[w_2,[z_\beta^*,[z_\alpha^*,w_1]]]\otimes z_\alpha z_\beta)\\
=&\sum\limits_{\alpha,\beta\in S(-1)}((-1)^{(|\alpha|+|\beta|)(1+|w_2|)}([[w_1,z_\beta^*],[z_\alpha^*,w_2]]\otimes z_\alpha z_\beta\\
&+(-1)^{|\beta||w_1|}[z_\beta^*,[w_1,[z_\alpha^*,w_2]]]\otimes z_\alpha z_\beta)\\
&-(-1)^{(|\alpha|+|\beta|)(1+|w_1|)+|w_1||w_2|}([[w_2,z_\beta^*],[z_\alpha^*,w_1]]\otimes z_\alpha z_\beta\\
&+(-1)^{|\beta||w_2|}[z_\beta^*,[w_2,[z_\alpha^*,w_1]]]\otimes z_\alpha z_\beta))\\
=&\sum\limits_{\alpha,\beta\in S(-1)}(-(-1)^{|\alpha|+|\beta|+|\alpha||w_2|}[[w_1,z_\alpha^*],[w_2,z_\beta^*]]\otimes z_\beta z_\alpha\\
&+(-1)^{|\alpha|+|\beta|+|\alpha||\beta|}[[w_1,[w_2,z_\alpha^*]],z_\beta^*]\otimes z_\alpha z_\beta\\
&-(-1)^{|\alpha|+|\beta|+|\alpha||\beta|+|\alpha||w_2|}[[w_1,z_\alpha^*],[w_2,z_\beta^*]]\otimes z_\alpha z_\beta\\
&+(-1)^{|\alpha|+|\beta|+|\alpha||\beta|+|\alpha||w_2|}[[[w_1,z_\alpha^*],w_2],z_\beta^*]\otimes z_\alpha z_\beta)\\
=&-2\sum\limits_{\alpha,\beta\in S(-1)}(-1)^{|\alpha|+|\beta|+|\alpha||w_2|}[[w_1,z_\alpha^*],[w_2,z_\beta^*]]\otimes z_\beta z_\alpha\\
&+\sum\limits_{\alpha\in S(-1)}(-1)^{|\alpha||w_2|}[[w_1,z_\alpha],[w_2,z_\alpha^*]]\otimes1\\
&+\sum\limits_{\alpha,\beta\in S(-1)}(-1)^{|\alpha|+|\beta|+|\alpha||\beta|}([w_1,w_2],f)[e,[z_\alpha^*,z_\beta^*]]\otimes z_\alpha z_\beta\\
&+\sum\limits_{\alpha,\beta\in S(-1)}(-1)^{|\alpha|+|\beta|}([w_1,w_2],f)[[e,z_\beta^*],z_\alpha^*]\otimes z_\alpha z_\beta\\
=&-2\sum\limits_{\alpha,\beta\in S(-1)}(-1)^{|\alpha|+|\beta|+|\alpha||w_2|}[[w_1,z_\alpha^*],[w_2,z_\beta^*]]\otimes z_\beta z_\alpha\\
&+\sum\limits_{\alpha\in S(-1)}(-1)^{|\alpha||w_2|}[[w_1,z_\alpha],[w_2,z_\alpha^*]]\otimes1+\frac{s-\sfr}{2}([w_1,w_2],f)h\otimes1\\
&+\sum\limits_{\alpha,\beta\in S(-1)}(-1)^{|\alpha|+|\beta|}([w_1,w_2],f)[[e,z_\beta^*],z_\alpha^*]\otimes z_\alpha z_\beta.
\end{split}
\end{equation}

As \begin{equation*}\begin{split}
&[w_1,[w_2,f]]\otimes1-(-1)^{|w_1||w_2|}[w_2,[w_1,f]]\otimes1\\=&[[w_1,w_2],f]\otimes1=([w_1,w_2],f)[e,f]\otimes1=([w_1,w_2],f)h\otimes1,
\end{split}\end{equation*}
combining this with \eqref{theta12}, \eqref{xzzy} and \eqref{wzzw} we have
\begin{equation}\label{5.22}
\begin{split}
[\Theta_{w_1},\Theta_{w_2}]=&([w_1,w_2],f)(e\otimes1+\sum\limits_{\alpha\in S(-1)}(-1)^{|\alpha|}[e,z_\alpha^*]\otimes z_\alpha\\
&+\sum\limits_{\alpha,\beta\in S(-1)}(-1)^{|\alpha|+|\beta|}[[e,z_\beta^*],z_\alpha^*]\otimes z_\alpha z_\beta)\\&-\sum\limits_{\alpha,\beta\in S(-1)}(-1)^{|\alpha|+|\beta|+|\alpha||w_2|}[[w_1,z_\alpha^*],[w_2,z_\beta^*]]\otimes z_\beta z_\alpha\\&+\sum\limits_{\alpha\in S(-1)}
(-1)^{|\alpha||w_2|}[w_1,z_\alpha][w_2,z_\alpha^*]\otimes1+\frac{s-\sfr}{2}([w_1,w_2],f)h\otimes1\\
&+\frac{\sfr-s}{2}[w_1,[w_2,f]]\otimes1-(-1)^{|w_1||w_2|}\frac{\sfr-s}{2}[w_2,[w_1,f]]\otimes1\\
&-\sum\limits_{\alpha\in S(-1)}[w_1,z_\alpha]\otimes z_\alpha^*[w_2,f]\\
&+\sum\limits_{\alpha\in S(-1)}(-1)^{|w_1||w_2|}[w_2,z_\alpha]\otimes z_\alpha^*[w_1,f]-1\otimes[\varphi_{w_1},\varphi_{w_2}]\\
=&([w_1,w_2],f)(e\otimes1+\sum\limits_{\alpha\in S(-1)}(-1)^{|\alpha|}[e,z_\alpha^*]\otimes z_\alpha\\
&+\sum\limits_{\alpha,\beta\in S(-1)}(-1)^{|\alpha|+|\beta|}[[e,z_\beta^*],z_\alpha^*]\otimes z_\alpha z_\beta)\\
&-\sum\limits_{\alpha,\beta\in S(-1)}(-1)^{|\alpha|+|\beta|+|\alpha||w_2|}[[w_1,z_\alpha^*],[w_2,z_\beta^*]]\otimes z_\beta z_\alpha\\
&+\sum\limits_{\alpha\in S(-1)}
(-1)^{|\alpha||w_2|}[w_1,z_\alpha][w_2,z_\alpha^*]\otimes1-\sum\limits_{\alpha\in S(-1)}[w_1,z_\alpha]\otimes z_\alpha^*[w_2,f]\\
&+\sum\limits_{\alpha\in S(-1)}(-1)^{|w_1||w_2|}[w_2,z_\alpha]\otimes z_\alpha^*[w_1,f]-1\otimes[\varphi_{w_1},\varphi_{w_2}].
\end{split}
\end{equation}
\subsubsection{}
For all $x\in\ggg^e(1)$, we have defined the Lie superalgebra endomorphism $\sharp: \ggg(0)\rightarrow\ggg(0)$ in \eqref{xh} by
\begin{equation}\label{3.25}
[x,z_\alpha]^{\sharp}=[x,z_\alpha]-\frac{1}{2}(h,[x,z_\alpha])h.
\end{equation}
Moreover, we have
\begin{equation}\label{3.26}
\begin{split}
\sum\limits_{\alpha\in S(-1)}(h,[x,z_\alpha^*])z_\alpha=&\sum\limits_{\alpha\in S(-1)}(e,[f,[x,z_\alpha^*]])z_\alpha
=\sum\limits_{\alpha\in S(-1)}(-1)^{|\alpha||x|}\langle z_\alpha^*,[x,f] \rangle z_\alpha\\
=&(-1)^{|x|}[x,f]
\end{split}
\end{equation}
and \begin{equation}\label{3.27}
\sum\limits_{\alpha\in S(-1)}(h,[x,z_\alpha])z_\alpha^*=-\sum\limits_{\alpha\in S(-1)}(-1)^{|\alpha|}(h,[x,z_\alpha^*])z_\alpha=-[x,f].
\end{equation}

As $[w_2,z_\alpha]^{\sharp}\in\ggg^e(0)$, it follows from \eqref{rewrite}, \eqref{zz*} and \eqref{3.25} that
\begin{equation}\label{wzsharp}
\begin{split}
\Theta_{[w_2,z_\alpha]^{\sharp}}=&[w_2,z_\alpha]^{\sharp}\otimes1-\sum\limits_{\beta\in S(-1)}\frac{1}{2}\otimes z_\beta[z_\beta^*,[w_2,z_\alpha]^{\sharp}]\\
=&[w_2,z_\alpha]\otimes1-\frac{1}{2}(h,[w_2,z_\alpha])h\otimes1\\&-\sum\limits_{\beta\in S(-1)}\frac{1}{2}\otimes z_\beta[z_\beta^*,([w_2,z_\alpha]-\frac{1}{2}(h,[w_2,z_\alpha])h)]\\
=&[w_2,z_\alpha]\otimes1-\frac{1}{2}(h,[w_2,z_\alpha])h\otimes1\\&-\sum\limits_{\beta\in S(-1)}\frac{1}{2}\otimes z_\beta[z_\beta^*,[w_2,z_\alpha]]+\frac{1}{4}\sum\limits_{\beta\in S(-1)}(h,[w_2,z_\alpha])\otimes z_\beta z_\beta^*\\
=&[w_2,z_\alpha]\otimes1-\frac{1}{2}(h,[w_2,z_\alpha])h\otimes1\\&-\sum\limits_{\beta\in S(-1)}\frac{1}{2}\otimes z_\beta[z_\beta^*,[w_2,z_\alpha]]+\frac{\sfr-s}{8}(h,[w_2,z_\alpha]).
\end{split}
\end{equation}
The same discussion as in \eqref{wzsharp} shows that
\begin{equation}\label{zwsharp}
\begin{split}
\Theta_{[z_\alpha^*, w_1]^{\sharp}}=&[z_\alpha^*, w_1]\otimes1-\frac{1}{2}(h,[z_\alpha^*, w_1])h\otimes1-\sum\limits_{\beta\in S(-1)}\frac{1}{2}\otimes z_\beta[z_\beta^*,[z_\alpha^*, w_1]]\\&+\frac{\sfr-s}{8}(h,[z_\alpha^*, w_1]).
\end{split}
\end{equation}

Taking \eqref{wzsharp} and \eqref{zwsharp} into account, we can deduce that
\begin{equation}\label{3.30}
\begin{split}
&\sum\limits_{\alpha\in S(-1)}\Theta_{[w_2,z_\alpha]^{\sharp}}\Theta_{[z_\alpha^*, w_1]^{\sharp}}\\
=&\sum\limits_{\alpha\in S(-1)}(([w_2,z_\alpha]\otimes1-\frac{1}{2}(h,[w_2,z_\alpha])h\otimes1-\frac{1}{2}\sum\limits_{\beta\in S(-1)}1\otimes z_\beta[z_\beta^*,[w_2,z_\alpha]]\\
&+\frac{\sfr-s}{8}(h,[w_2,z_\alpha]))([z_\alpha^*, w_1]\otimes1-\frac{1}{2}(h,[z_\alpha^*, w_1])h\otimes1\\
&-\frac{1}{2}\sum\limits_{\beta\in S(-1)}1\otimes z_\beta[z_\beta^*,[z_\alpha^*, w_1]]+\frac{\sfr-s}{8}(h,[z_\alpha^*, w_1])))\\
=&\sum\limits_{\alpha\in S(-1)}[w_2,z_\alpha][z_\alpha^*, w_1]\otimes1-\frac{1}{2}\sum\limits_{\alpha\in S(-1)}(h,[z_\alpha^*, w_1])[w_2,z_\alpha]h\otimes1\\
&-\frac{1}{2}\sum\limits_{\alpha,\beta\in S(-1)}[w_2,z_\alpha]\otimes z_\beta[z_\beta^*,[z_\alpha^*, w_1]]+\frac{\sfr-s}{8}\sum\limits_{\alpha\in S(-1)}(h,[z_\alpha^*, w_1])[w_2,z_\alpha]\otimes1\\
&-\frac{1}{2}\sum\limits_{\alpha\in S(-1)}(h,[w_2,z_\alpha])[z_\alpha^*, w_1]h\otimes1+\frac{1}{4}\sum\limits_{\alpha\in S(-1)}(h,[w_2,z_\alpha])(h,[z_\alpha^*, w_1])h^2\otimes1\\
&+\frac{1}{4}\sum\limits_{\alpha,\beta\in S(-1)}(h,[w_2,z_\alpha])h\otimes z_\beta[z_\beta^*,[z_\alpha^*, w_1]]\\
&+\frac{s-\sfr}{16}\sum\limits_{\alpha\in S(-1)}(h,[z_\alpha^*, w_1])(h,[w_2,z_\alpha])h\otimes1\\&
-\frac{1}{2}\sum\limits_{\alpha,\beta\in S(-1)}(-1)^{(|w_1|+|\alpha|)(|w_2|+|\alpha|)}[z_\alpha^*, w_1]\otimes z_\beta[z_\beta^*,[w_2,z_\alpha]]\\
&+\frac{1}{4}\sum\limits_{\alpha,\beta\in S(-1)}(h,[z_\alpha^*, w_1])h\otimes z_\beta[z_\beta^*,[w_2,z_\alpha]]\\
&+\frac{1}{4}\sum\limits_{\alpha,\beta,\gamma\in S(-1)}(-1)^{(|w_1|+|\alpha|)(|w_2|+|\alpha|)}\otimes z_\beta[z_\beta^*,[z_\alpha^*, w_1]]z_\gamma[z_\gamma^*,[w_2,z_\alpha]]\\
&+\frac{s-\sfr}{16}\sum\limits_{\alpha,\beta\in S(-1)}(h,[z_\alpha^*, w_1])\otimes z_\beta[z_\beta^*,[w_2,z_\alpha]]
+\frac{\sfr-s}{8}\sum\limits_{\alpha\in S(-1)}(h,[w_2,z_\alpha])[z_\alpha^*, w_1]\otimes1\\
&+\frac{s-\sfr}{16}\sum\limits_{\alpha\in S(-1)}(h,[w_2,z_\alpha])(h,[z_\alpha^*, w_1])h\otimes1\\
&+\frac{s-\sfr}{16}\sum\limits_{\alpha,\beta\in S(-1)}(h,[w_2,z_\alpha])\otimes z_\beta[z_\beta^*,[z_\alpha^*, w_1]]\\
&+\frac{(s-\sfr)^2}{64}\sum\limits_{\alpha\in S(-1)}(h,[w_2,z_\alpha])(h,[z_\alpha^*, w_1]).
\end{split}
\end{equation}

Now we will deal with the terms in \eqref{3.30}. First note that
\begin{equation}\label{3.31}
\begin{split}
&\sum\limits_{\alpha\in S(-1)}[w_2,z_\alpha][z_\alpha^*, w_1]\otimes1
=-\sum\limits_{\alpha\in S(-1)}(-1)^{|\alpha||w_1|}[w_2,z_\alpha][w_1,z_\alpha^*]\otimes1\\
=&\sum\limits_{\alpha\in S(-1)}(-1)^{|\alpha||w_1|+|\alpha|+(|w_1|+|\alpha|)(|w_2|+|\alpha|)}(-[[w_1,z_\alpha],[w_2,z_\alpha^*]]+[w_1,z_\alpha][w_2,z_\alpha^*])\otimes1\\
=&-\sum\limits_{\alpha\in S(-1)}(-1)^{|\alpha||w_2|+|w_1||w_2|}[[w_1,z_\alpha],[w_2,z_\alpha^*]]\otimes1\\
&+\sum\limits_{\alpha\in S(-1)}(-1)^{|\alpha||w_2|+|w_1||w_2|}[w_1,z_\alpha][w_2,z_\alpha^*]\otimes1.
\end{split}
\end{equation}

Moreover, it follows from \eqref{zz*}, \eqref{3.26} and \eqref{3.27} that
\begin{equation}\label{3.32}
\begin{split}
&\frac{1}{4}\sum\limits_{\alpha,\beta\in S(-1)}(h,[w_2,z_\alpha])h\otimes z_\beta[z_\beta^*,[z_\alpha^*, w_1]]
+\frac{1}{4}\sum\limits_{\alpha,\beta\in S(-1)}(h,[z_\alpha^*, w_1])h\otimes z_\beta[z_\beta^*,[w_2,z_\alpha]]\\
=&\frac{1}{4}\sum\limits_{\alpha,\beta\in S(-1)}h\otimes z_\beta[z_\beta^*,[(h,[w_2,z_\alpha])z_\alpha^*, w_1]]\\
&-\frac{1}{4}\sum\limits_{\alpha,\beta\in S(-1)}(-1)^{|\alpha||w_1|}h\otimes z_\beta[z_\beta^*,[w_2,(h,[w_1,z_\alpha^*])z_\alpha]]\\
=&-\frac{1}{4}\sum\limits_{\alpha\in S(-1)}h\otimes z_\alpha[z_\alpha^*,[[w_2,f], w_1]]
-\frac{1}{4}\sum\limits_{\alpha\in S(-1)}h\otimes z_\alpha[z_\alpha^*,[w_2,[w_1,f]]]\\
=&\frac{1}{4}(-1)^{|w_1||w_2|}\sum\limits_{\alpha\in S(-1)}h\otimes z_\alpha[z_\alpha^*,[[w_1,w_2],f]]\\
=&\frac{1}{4}(-1)^{|w_1||w_2|}([w_1,w_2],f)\sum\limits_{\alpha\in S(-1)}h\otimes z_\alpha[z_\alpha^*,h]\\
=&\frac{1}{4}(-1)^{|w_1||w_2|}([w_1,w_2],f)\sum\limits_{\alpha\in S(-1)}h\otimes z_\alpha z_\alpha^*
=\frac{\sfr-s}{8}(-1)^{|w_1||w_2|}([w_1,w_2], f)h\otimes1.
\end{split}
\end{equation}
By the same discussion as in \eqref{3.32} we can obtain that
\begin{equation}
\begin{split}
&-\frac{1}{2}\sum\limits_{\alpha\in S(-1)}(h,[z_\alpha^*, w_1])[w_2,z_\alpha]h\otimes1-\frac{1}{2}\sum\limits_{\alpha\in S(-1)}(h,[w_2,z_\alpha])[z_\alpha^*, w_1]h\otimes1\\
=&-\frac{(-1)^{|w_1||w_2|}}{2}([w_1,w_2],f)h^2\otimes1,
\end{split}
\end{equation}
\begin{equation}
\begin{split}
&\frac{\sfr-s}{8}\sum\limits_{\alpha\in S(-1)}(h,[z_\alpha^*, w_1])[w_2,z_\alpha]\otimes1+\frac{\sfr-s}{8}\sum\limits_{\alpha\in S(-1)}(h,[w_2,z_\alpha])[z_\alpha^*, w_1]\otimes1\\
=&\frac{\sfr-s}{8}(-1)^{|w_1||w_2|}([w_1,w_2], f)h\otimes1,
\end{split}
\end{equation}
\begin{equation}
\begin{split}
&\frac{1}{4}\sum\limits_{\alpha\in S(-1)}(h,[w_2,z_\alpha])(h,[z_\alpha^*, w_1])h^2\otimes1
=\frac{1}{4}\sum\limits_{\alpha\in S(-1)}(h,[(h,[w_2,z_\alpha])z_\alpha^*, w_1])h^2\otimes1\\
=&\frac{1}{4}(-1)^{|w_1||w_2|}([w_1,w_2], f)h^2\otimes1,
\end{split}
\end{equation}
\begin{equation}
\begin{split}
&\frac{s-\sfr}{16}(\sum\limits_{\alpha\in S(-1)}(h,[z_\alpha^*, w_1])(h,[w_2,z_\alpha])h+\sum\limits_{\alpha\in S(-1)}(h,[w_2,z_\alpha])(h,[z_\alpha^*, w_1])h)\otimes1\\
=&\frac{s-\sfr}{8}(-1)^{|w_1||w_2|}([w_1,w_2],f)h\otimes1,
\end{split}
\end{equation}
\begin{equation}
\begin{split}
&\frac{s-\sfr}{16}(\sum\limits_{\alpha,\beta\in S(-1)}(h,[z_\alpha^*, w_1])\otimes z_\beta[z_\beta^*,[w_2,z_\alpha]]+\sum\limits_{\alpha,\beta\in S(-1)}(h,[w_2,z_\alpha])\otimes z_\beta[z_\beta^*,[z_\alpha^*, w_1]])\\
=&\frac{s-\sfr}{16}\sum\limits_{\alpha\in S(-1)}(-1)^{|w_1||w_2|}([w_1,w_2],f)\otimes z_\alpha z_\alpha^*=-\frac{(s-\sfr)^2}{32}(-1)^{|w_1||w_2|}([w_1,w_2],f),
\end{split}
\end{equation}
\begin{equation}\label{3.38}
\frac{(s-\sfr)^2}{64}\sum\limits_{\alpha\in S(-1)}(h,[w_2,z_\alpha])(h,[z_\alpha^*, w_1])=\frac{(s-\sfr)^2}{64}(-1)^{|w_1||w_2|}([w_1,w_2],f).
\end{equation}

On the other hand, it follows from the procedure of \eqref{xzzy} that
\begin{equation}\label{3.39}
\begin{split}
&\sum\limits_{\alpha,\beta\in S(-1)}[w_2,z_\alpha]\otimes z_\beta[z_\beta^*,[z_\alpha^*, w_1]]\\
=&\sum\limits_{\alpha,\beta\in S(-1)}[w_2,z_\alpha]\otimes z_\beta[[z_\beta^*,z_\alpha^*], w_1]+\sum\limits_{\alpha,\beta\in S(-1)}(-1)^{|\alpha||\beta|}[w_2,z_\alpha]\otimes z_\beta[z_\alpha^*,[z_\beta^*,w_1]]\\
=&-\sum\limits_{\alpha\in S(-1)}(-1)^{|\alpha|}[w_2,z_\alpha^*]\otimes z_\alpha[f,w_1]+[w_2,[w_1,f]]\otimes1-\sum\limits_{\alpha\in S(-1)}[w_2,z_\alpha]\otimes z_\alpha^*[w_1,f]\\
&+\sum\limits_{\beta,\gamma\in S(-1)}(-1)^{(|\beta|+|\gamma|)(1+|w_1|)}[w_2,[z_\gamma^*,[z_\beta^*, w_1]]]\otimes z_\beta z_\gamma\\
=&-2\sum\limits_{\alpha\in S(-1)}[w_2,z_\alpha]\otimes z_\alpha^*[w_1,f]+[w_2,[w_1,f]]\otimes1\\
&+\sum\limits_{\beta,\gamma\in S(-1)}(-1)^{(|\beta|+|\gamma|)(1+|w_1|)}[w_2,[z_\gamma^*,[z_\beta^*, w_1]]]\otimes z_\beta z_\gamma.
\end{split}
\end{equation}
By the same discussion as in \eqref{3.39}, one can conclude that
\begin{equation}\label{3.40}
\begin{split}
&\sum\limits_{\alpha,\beta\in S(-1)}(-1)^{(|w_1|+|\alpha|)(|w_2|+|\alpha|)}[z_\alpha^*, w_1]\otimes z_\beta[z_\beta^*,[w_2,z_\alpha]]\\
=&(-1)^{|w_1||w_2|}\sum\limits_{\alpha,\beta\in S(-1)}(-1)^{|\alpha|}[w_1,z_\alpha^*]\otimes z_\beta[z_\beta^*,[z_\alpha,w_2]]\\
=&-(-1)^{|w_1||w_2|}\sum\limits_{\alpha,\beta\in S(-1)}[w_1,z_\alpha]\otimes z_\beta[z_\beta^*,[z_\alpha^*,w_2]]\\
=&2\sum\limits_{\alpha\in S(-1)}(-1)^{|w_1||w_2|}[w_1,z_\alpha]\otimes z_\alpha^*[w_2,f]-(-1)^{|w_1||w_2|}[w_1,[w_2,f]]\otimes1\\
&-\sum\limits_{\beta,\gamma\in S(-1)}(-1)^{(|\beta|+|\gamma|)(1+|w_2|)+|w_1||w_2|}[w_1,[z_\gamma^*,[z_\beta^*, w_2]]]\otimes z_\beta z_\gamma.
\end{split}
\end{equation}
Combining \eqref{3.39} with \eqref{3.40}, it is immediate from \eqref{wzzw} that
\begin{equation}\label{3.41}
\begin{split}
&\sum\limits_{\alpha,\beta\in S(-1)}[w_2,z_\alpha]\otimes z_\beta[z_\beta^*,[z_\alpha^*, w_1]]\\
&+\sum\limits_{\alpha,\beta\in S(-1)}(-1)^{(|w_1|+|\alpha|)(|w_2|+|\alpha|)}[z_\alpha^*, w_1]\otimes z_\beta[z_\beta^*,[w_2,z_\alpha]]\\
=&-2\sum\limits_{\alpha\in S(-1)}[w_2,z_\alpha]\otimes z_\alpha^*[w_1,f]+2\sum\limits_{\alpha\in S(-1)}(-1)^{|w_1||w_2|}[w_1,z_\alpha]\otimes z_\alpha^*[w_2,f]\\
&-(-1)^{|w_1||w_2|}([w_1,w_2],f)h\otimes1-\frac{s-\sfr}{2}(-1)^{|w_1||w_2|}([w_1,w_2],f)h\otimes1\\
&+2\sum\limits_{\alpha,\beta\in S(-1)}(-1)^{|\alpha|+|\beta|+|\alpha||w_2|+|w_1||w_2|}[[w_1,z_\alpha^*],[w_2,z_\beta^*]]\otimes z_\beta z_\alpha\\
&-\sum\limits_{\alpha\in S(-1)}(-1)^{|\alpha||w_2|+|w_1||w_2|}[[w_1,z_\alpha],[w_2,z_\alpha^*]]\otimes1\\
&-([w_1,w_2],f)\sum\limits_{\alpha,\beta\in S(-1)}(-1)^{|\alpha|+|\beta|+|w_1||w_2|}[[e,z_\beta^*],z_\alpha^*]\otimes z_\alpha z_\beta.
\end{split}
\end{equation}

Now we can deduce from \eqref{3.30}-\eqref{3.38} and \eqref{3.41} that
\begin{equation}\label{3.42}
\begin{split}
&\sum\limits_{\alpha\in S(-1)}(-1)^{|w_1||w_2|}\Theta_{[w_2,z_\alpha]^{\sharp}}\Theta_{[z_\alpha^*, w_1]^{\sharp}}\\
=&\sum\limits_{\alpha\in S(-1)}(-1)^{|\alpha||w_2|}[w_1,z_\alpha][w_2,z_\alpha^*]\otimes1-\frac{1}{2}\sum\limits_{\alpha\in S(-1)}(-1)^{|\alpha||w_2|}[[w_1,z_\alpha],[w_2,z_\alpha^*]]\otimes1\\
&-\frac{1}{2}([w_1,w_2],f)h^2\otimes1+\frac{\sfr-s}{8}([w_1,w_2],f)h\otimes1+\frac{1}{4}([w_1,w_2],f)h^2\otimes1\\
&+\frac{\sfr-s}{8}([w_1,w_2],f)h\otimes1+\frac{s-\sfr}{8}([w_1,w_2],f)h\otimes1-\frac{(s-\sfr)^2}{32}([w_1,w_2],f)\\
&+\frac{(s-\sfr)^2}{64}([w_1,w_2],f)+\sum\limits_{\alpha\in S(-1)}(-1)^{|w_1||w_2|}[w_2,z_\alpha]\otimes z_\alpha^*[w_1,f]\\
&-\sum\limits_{\alpha\in S(-1)}[w_1,z_\alpha]\otimes z_\alpha^*[w_2,f]+\frac{1}{2}([w_1,w_2],f)h\otimes1+\frac{s-\sfr}{4}([w_1,w_2],f)h\otimes1\\
&-\sum\limits_{\alpha,\beta\in S(-1)}(-1)^{|\alpha|+|\beta|+|\alpha||w_2|}[[w_1,z_\alpha^*],[w_2,z_\beta^*]]\otimes z_\beta z_\alpha\\
&+\frac{1}{2}([w_1,w_2],f)\sum\limits_{\alpha,\beta\in S(-1)}(-1)^{|\alpha|+|\beta|}[[e,z_\beta^*],z_\alpha^*]\otimes z_\alpha z_\beta\\
&+\frac{1}{4}\sum\limits_{\alpha,\beta,\gamma\in S(-1)}(-1)^{|\alpha||w_1|+|\alpha||w_2|+|\alpha|}\otimes z_\beta[z_\beta^*,[z_\alpha^*, w_1]]z_\gamma[z_\gamma^*,[w_2,z_\alpha]]\\
=&\sum\limits_{\alpha\in S(-1)}(-1)^{|\alpha||w_2|}[w_1,z_\alpha][w_2,z_\alpha^*]\otimes1-\frac{1}{2}\sum\limits_{\alpha\in S(-1)}(-1)^{|\alpha||w_2|}[[w_1,z_\alpha],[w_2,z_\alpha^*]]\otimes1\\
&-\frac{1}{4}([w_1,w_2],f)h^2\otimes1+\frac{s-\sfr}{8}([w_1,w_2],f)h\otimes1-\frac{(s-\sfr)^2}{64}([w_1,w_2],f)\\
&+\sum\limits_{\alpha\in S(-1)}(-1)^{|w_1||w_2|}[w_2,z_\alpha]\otimes z_\alpha^*[w_1,f]-\sum\limits_{\alpha\in S(-1)}[w_1,z_\alpha]\otimes z_\alpha^*[w_2,f]\\
&+\frac{1}{2}([w_1,w_2],f)h\otimes1-\sum\limits_{\alpha,\beta\in S(-1)}(-1)^{|\alpha|+|\beta|+|\alpha||w_2|}[[w_1,z_\alpha^*],[w_2,z_\beta^*]]\otimes z_\beta z_\alpha\\
&+\frac{1}{2}([w_1,w_2],f)\sum\limits_{\alpha,\beta\in S(-1)}(-1)^{|\alpha|+|\beta|}[[e,z_\beta^*],z_\alpha^*]\otimes z_\alpha z_\beta\\
&+\frac{1}{4}\sum\limits_{\alpha,\beta,\gamma\in S(-1)}(-1)^{|\alpha||w_1|+|\alpha||w_2|+|\alpha|}\otimes z_\beta[z_\beta^*,[z_\alpha^*, w_1]]z_\gamma[z_\gamma^*,[w_2,z_\alpha]].
\end{split}
\end{equation}

Interchanging the roles of $w_1$ and $w_2$ in \eqref{3.42}, we have
\begin{equation}\label{3.43}
\begin{split}
&\sum\limits_{\alpha\in S(-1)}\Theta_{[w_1,z_\alpha]^{\sharp}}\Theta_{[z_\alpha^*, w_2]^{\sharp}}\\
=&\sum\limits_{\alpha\in S(-1)}(-1)^{|\alpha||w_1|+|w_1||w_2|}[w_2,z_\alpha][w_1,z_\alpha^*]\otimes1\\
&-\frac{1}{2}\sum\limits_{\alpha\in S(-1)}(-1)^{|\alpha||w_1|+|w_1||w_2|}[[w_2,z_\alpha],[w_1,z_\alpha^*]]\otimes1\\
&-\frac{1}{4}(-1)^{|w_1||w_2|}([w_2,w_1],f)h^2\otimes1\\
&+\frac{s-\sfr}{8}(-1)^{|w_1||w_2|}([w_2,w_1],f)h\otimes1-\frac{(s-\sfr)^2}{64}(-1)^{|w_1||w_2|}([w_2,w_1],f)\\
&+\sum\limits_{\alpha\in S(-1)}[w_1,z_\alpha]\otimes z_\alpha^*[w_2,f]-\sum\limits_{\alpha\in S(-1)}(-1)^{|w_1||w_2|}[w_2,z_\alpha]\otimes z_\alpha^*[w_1,f]\\
&+\frac{1}{2}(-1)^{|w_1||w_2|}([w_2,w_1],f)h\otimes1\\
&-\sum\limits_{\alpha,\beta\in S(-1)}(-1)^{|\alpha|+|\beta|+|\alpha||w_1|+|w_1||w_2|}[[w_2,z_\alpha^*],[w_1,z_\beta^*]]\otimes z_\beta z_\alpha\\
&+\frac{1}{2}(-1)^{|w_1||w_2|}([w_2,w_1],f)\sum\limits_{\alpha,\beta\in S(-1)}(-1)^{|\alpha|+|\beta|}[[e,z_\beta^*],z_\alpha^*]\otimes z_\alpha z_\beta\\
&+\frac{1}{4}\sum\limits_{\alpha,\beta,\gamma\in S(-1)}(-1)^{|\alpha||w_1|+|\alpha||w_2|+|\alpha|+|w_1||w_2|}\otimes z_\beta[z_\beta^*,[z_\alpha^*, w_2]]z_\gamma[z_\gamma^*,[w_1,z_\alpha]]\\
=&-\sum\limits_{\alpha\in S(-1)}(-1)^{|\alpha||w_2|}[w_1,z_\alpha][w_2,z_\alpha^*]\otimes1\\
&-\frac{1}{2}\sum\limits_{\alpha\in S(-1)}(-1)^{|\alpha||w_2|}[[w_1,z_\alpha],[w_2,z_\alpha^*]]\otimes1+\frac{1}{4}([w_1,w_2],f)h^2\otimes1\\
&-\frac{s-\sfr}{8}([w_1,w_2],f)h\otimes1
+\frac{(s-\sfr)^2}{64}([w_1,w_2],f)+\sum\limits_{\alpha\in S(-1)}[w_1,z_\alpha]\otimes z_\alpha^*[w_2,f]\\
&-\sum\limits_{\alpha\in S(-1)}(-1)^{|w_1||w_2|}[w_2,z_\alpha]\otimes z_\alpha^*[w_1,f]-\frac{1}{2}([w_1,w_2],f)h\otimes1\\
&+\sum\limits_{\alpha,\beta\in S(-1)}(-1)^{|\alpha|+|\beta|+|\alpha||w_2|}[[w_1,z_\alpha^*],[w_2,z_\beta^*]]\otimes z_\beta z_\alpha\\
&-\frac{1}{2}([w_1,w_2],f)\sum\limits_{\alpha,\beta\in S(-1)}(-1)^{|\alpha|+|\beta|}[[e,z_\beta^*],z_\alpha^*]\otimes z_\alpha z_\beta\\
&+\frac{1}{4}\sum\limits_{\alpha,\beta,\gamma\in S(-1)}(-1)^{(|\alpha|+|w_1|)(|\alpha|+|w_2|)}\otimes z_\beta[z_\beta^*,[z_\alpha^*, w_2]]z_\gamma[z_\gamma^*,[w_1,z_\alpha]].
\end{split}
\end{equation}
In \eqref{3.43} we used the fact that
\begin{equation}
\begin{split}
&-\sum\limits_{\alpha,\beta\in S(-1)}(-1)^{|\alpha|+|\beta|+|\alpha||w_1|+|w_1||w_2|}[[w_2,z_\alpha^*],[w_1,z_\beta^*]]\otimes z_\beta z_\alpha\\
=&-\sum\limits_{\alpha,\beta\in S(-1)}(-1)^{|\alpha|+|\beta|+|\beta||w_1|+|w_1||w_2|}[[w_2,z_\beta^*],[w_1,z_\alpha^*]]\otimes z_\alpha z_\beta\\
=&\sum\limits_{\alpha,\beta\in S(-1)}(-1)^{|\alpha|+|\beta|+|\alpha||w_2|}[[w_1,z_\alpha^*],[w_2,z_\beta^*]]\otimes z_\beta z_\alpha\\
&+\sum\limits_{\alpha,\beta\in S(-1)}(-1)^{|\alpha|+|\beta|+|\alpha||\beta|+|\alpha||w_2|}[[w_1,z_\alpha^*],[w_2,z_\beta^*]]\otimes [z_\alpha,z_\beta]\\
=&\sum\limits_{\alpha,\beta\in S(-1)}(-1)^{|\alpha|+|\beta|+|\alpha||w_2|}[[w_1,z_\alpha^*],[w_2,z_\beta^*]]\otimes z_\beta z_\alpha\\
&-\sum\limits_{\alpha\in S(-1)}(-1)^{|\alpha||w_2|}[[w_1,z_\alpha],[w_2,z_\alpha^*]]\otimes1.
\end{split}
\end{equation}

We now combine \eqref{5.22} with \eqref{3.42} and \eqref{3.43} to deduce that
\begin{equation}\label{3.45}
\begin{split}
&[\Theta_{w_1},\Theta_{w_2}]+\frac{1}{2}\sum\limits_{\alpha\in S(-1)}\Theta_{[w_1,z_\alpha]^{\sharp}}\Theta_{[z_\alpha^*, w_2]^{\sharp}}-\frac{1}{2}\sum\limits_{\alpha\in S(-1)}(-1)^{|w_1||w_2|}\Theta_{[w_2,z_\alpha]^{\sharp}}\Theta_{[z_\alpha^*, w_1]^{\sharp}}\\
=&([w_1,w_2],f)((e+\frac{h^2}{4}-(\frac{1}{2}+\frac{(s-\sfr)}{8})h)\otimes1+\sum\limits_{\alpha\in S(-1)}(-1)^{|\alpha|}[e,z_\alpha^*]\otimes z_\alpha\\
&+\frac{1}{2}\sum\limits_{\alpha,\beta\in S(-1)}(-1)^{|\alpha|+|\beta|}[[e,z_\alpha^*],z_\beta^*]\otimes z_\beta z_\alpha)\\
&+\frac{1}{8}\sum\limits_{\alpha,\beta,\gamma\in S(-1)}(-1)^{(|\alpha|+|w_1|)(|\alpha|+|w_2|)}\otimes z_\beta[z_\beta^*,[z_\alpha^*, w_2]]z_\gamma[z_\gamma^*,[w_1,z_\alpha]]\\
&-\frac{1}{8}\sum\limits_{\alpha,\beta,\gamma\in S(-1)}(-1)^{|\alpha||w_1|+|\alpha||w_2|+|\alpha|}\otimes z_\beta[z_\beta^*,[z_\alpha^*, w_1]]z_\gamma[z_\gamma^*,[w_2,z_\alpha]]\\
&+\frac{(s-\sfr)^2}{64}([w_1,w_2],f)-1\otimes[\varphi_{w_1},\varphi_{w_2}].
\end{split}
\end{equation}
\subsubsection{}
Recall that in \eqref{C} we introduced $C$, which is the central element of $W_\chi'$, and showed that the element $\Theta_{\text{Cas}}=\sum_{i\in I}(-1)^{|i|}\Theta_{a_i}\Theta_{b_i}$
commutes with all operators $\Theta_v$ for $v\in\ggg^e(0)$ in Proposition \ref{vcommutate}, where $\{a_i\mid i\in I\}$ and $\{b_i\mid i\in I\}$ are dual bases of $\ggg^e(0)$ with respect to the restriction of the invariant form $(\cdot,\cdot)$ to $\ggg^e(0)$. Now we will calculate $C-\Theta_{\text{Cas}}$ in this part.

Since $[z_\beta^*,b_i],\,[z_\beta^*,a_i]\in\ggg(-1)$ for any $\beta\in S(-1)$ and $i\in I$, it is immediate from \eqref{zuz} that
\begin{equation}\label{3.48}
\begin{split}
[z_\beta^*,b_i]=&\sum\limits_{\alpha\in S(-1)}[z_\alpha^*,[z_\beta^*,b_i]]z_\alpha=\sum\limits_{\alpha\in S(-1)}(e,[z_\alpha^*,[z_\beta^*,b_i]])z_\alpha,\\
[z_\beta^*,a_i]=&\sum\limits_{\alpha\in S(-1)}[z_\alpha^*,[z_\beta^*,a_i]]z_\alpha=\sum\limits_{\alpha\in S(-1)}(e,[z_\alpha^*,[z_\beta^*,a_i]])z_\alpha.
\end{split}
\end{equation}
As the even bilinear form $(\cdot,\cdot)$ is invariant and $\mathbb{C}h$ is orthogonal to $\ggg(0)^\sharp$ with respect to $(\cdot,\cdot)$, it follows from \eqref{3.48} that
\begin{equation}\label{3.49}
\begin{split}
&\sum\limits_{\alpha,\beta\in S(-1)}(-1)^{|\alpha|+|\beta|}[[e,z_\alpha^*],z_\beta^*]^\sharp\otimes z_\beta z_\alpha\\
&=\sum\limits_{\alpha,\beta\in S(-1),i\in I}(-1)^{|i|}([[e,z_\alpha^*],z_\beta^*]^\sharp,b_i)a_i\otimes z_\beta z_\alpha\\
&=\sum\limits_{\alpha,\beta\in S(-1),i\in I}(-1)^{|i|}a_i\otimes z_\beta (e,[z_\alpha^*,[z_\beta^*,b_i]])z_\alpha\\
&=\sum\limits_{\beta\in S(-1),i\in I}(-1)^{|i|}a_i\otimes  z_\beta[z_\beta^*,b_i].
\end{split}
\end{equation}
Interchanging the roles of $\{a_i\}$ and $\{b_i\}$, we can obtain
\begin{equation}\label{3.50}
\begin{split}
&\sum\limits_{\alpha,\beta\in S(-1)}(-1)^{|\alpha|+|\beta|}[[e,z_\alpha^*],z_\beta^*]^\sharp\otimes z_\beta z_\alpha\\
=&\sum\limits_{\alpha,\beta\in S(-1),i\in I}(-1)^{|i|}(a_i,[[e,z_\alpha^*],z_\beta^*]^\sharp)b_i\otimes z_\beta z_\alpha\\
=&\sum\limits_{\alpha,\beta\in S(-1),i\in I}(-1)^{|i|}(a_i,[[e,z_\alpha^*],z_\beta^*])b_i\otimes z_\beta z_\alpha\\
=&\sum\limits_{\alpha,\beta\in S(-1),i\in I}b_i\otimes z_\beta(e,[z_\alpha^*,[z_\beta^*,a_i]])z_\alpha\\
=&\sum\limits_{\beta\in S(-1),i\in I}b_i\otimes z_\beta[z_\beta^*,a_i],
\end{split}
\end{equation}
where the third equation comes from the supersymmetry of the invariant bilinear form $(\cdot,\cdot)$.

On the other hand, it can be observed that
\begin{equation}\label{3.51}
\begin{split}
&\sum\limits_{\alpha,\beta\in S(-1)}(-1)^{|\alpha|+|\beta|}[[e,z_\alpha^*],z_\beta^*]^\sharp\otimes z_\beta z_\alpha\\
=&\sum\limits_{\alpha,\beta\in S(-1)}(-1)^{|\alpha|+|\beta|}([[e,z_\alpha^*],z_\beta^*]\otimes z_\beta z_\alpha-\frac{1}{2}(h,[[e,z_\alpha^*],z_\beta^*])h\otimes z_\beta z_\alpha)\\
=&\sum\limits_{\alpha,\beta\in S(-1)}(-1)^{|\alpha|+|\beta|}([[e,z_\alpha^*],z_\beta^*]\otimes z_\beta z_\alpha-\frac{1}{2}(-1)^{|\alpha||\beta|}([[h,z_\beta^*],z_\alpha^*],e)h\otimes z_\beta z_\alpha)\\
=&\sum\limits_{\alpha,\beta\in S(-1)}(-1)^{|\alpha|+|\beta|}[[e,z_\alpha^*],z_\beta^*]\otimes z_\beta z_\alpha+\frac{1}{2}\sum\limits_{\alpha,\beta\in S(-1)}(-1)^{|\alpha|+|\beta|}([[h,z_\beta^*],z_\alpha],e)h\otimes z_\beta z_\alpha^*\\
=&\sum\limits_{\alpha,\beta\in S(-1)}(-1)^{|\alpha|+|\beta|}[[e,z_\alpha^*],z_\beta^*]\otimes z_\beta z_\alpha+\frac{(s-\sfr)}{4}h\otimes 1,
\end{split}
\end{equation}
where the last equation in \eqref{3.51} follows from \eqref{zz*} and the fact that
\begin{equation*}
\begin{split}
&\sum\limits_{\alpha,\beta\in S(-1)}(-1)^{|\alpha|+|\beta|}([[h,z_\beta^*],z_\alpha],e)h\otimes z_\beta z_\alpha^*\\
=&-\sum\limits_{\alpha,\beta\in S(-1)}(-1)^{|\alpha|+|\beta|}([z_\beta^*,z_\alpha],e)h\otimes z_\beta z_\alpha^*\\
=&-\sum\limits_{\alpha\in S(-1)}h\otimes z_\alpha z_\alpha^*=\frac{(s-\sfr)}{2}h\otimes1.
\end{split}
\end{equation*}

As a result, it follows from \eqref{3.49}, \eqref{3.50} and \eqref{3.51} that
\begin{equation}\label{3.52}
\begin{split}
&\sum\limits_{\beta\in S(-1),i\in I}((-1)^{|i|}a_i\otimes  z_\beta[z_\beta^*,b_i]+b_i\otimes  z_\beta[z_\beta^*,a_i])\\&=2\sum\limits_{\alpha,\beta\in S(-1)}(-1)^{|\alpha|+|\beta|}[[e,z_\alpha^*],z_\beta^*]\otimes z_\beta z_\alpha+\frac{(s-\sfr)}{2}h\otimes 1.
\end{split}
\end{equation}

As $\mathbb{C}h$ is orthogonal to $\ggg(0)^\sharp$ with respect to $(\cdot,\cdot)$, \eqref{3.48} yields
\begin{equation}\label{3.53}
\begin{split}
&\sum\limits_{\alpha,\beta\in S(-1),i\in I}1\otimes z_\alpha[z_\alpha^*,b_i]z_\beta[z_\beta^*,a_i]\\
=&\sum\limits_{\alpha,\beta,\gamma,\delta\in S(-1),i\in I}1\otimes z_\alpha(e,[z_\gamma^*,[z_\alpha^*,b_i]])z_\gamma z_\beta(e,[z_\delta^*,[z_\beta^*,a_i]])z_\delta\\
=&\sum\limits_{\alpha,\beta,\gamma,\delta\in S(-1),i\in I}([[e,z_\delta^*],z_\beta^*],a_i)([[e,z_\gamma^*],z_\alpha^*],b_i)\otimes z_\alpha z_\gamma z_\beta z_\delta\\
=&\sum\limits_{\alpha,\beta,\gamma,\delta\in S(-1),i\in I}([[e,z_\delta^*],z_\beta^*]^\sharp,a_i)([[e,z_\gamma^*],z_\alpha^*]^\sharp,b_i)\otimes z_\alpha z_\gamma z_\beta z_\delta\\
=&\sum\limits_{\alpha,\beta,\gamma,\delta\in S(-1),i\in I}([[e,z_\delta^*],z_\beta^*]^\sharp,([[e,z_\gamma^*],z_\alpha^*]^\sharp,b_i)a_i)\otimes z_\alpha z_\gamma z_\beta z_\delta\\
=&\sum\limits_{\alpha,\beta,\gamma,\delta\in S(-1)}([[e,z_\delta^*],z_\beta^*]^\sharp,[[e,z_\gamma^*],z_\alpha^*]^\sharp)\otimes z_\alpha z_\gamma z_\beta z_\delta.
\end{split}
\end{equation}

Since $a_i$ and $b_i$ are in $\ggg^e(0)$ for all $i\in I$, we can conclude from \eqref{rewrite}, \eqref{3.52} and \eqref{3.53} that
\begin{equation}\label{3.54}
\begin{split}
&\Theta_{\text{Cas}}=\sum\limits_{i\in I}(-1)^{|i|}\Theta_{a_i}\Theta_{b_i}\\
=&\sum\limits_{i\in I}(-1)^{|i|}(a_i\otimes1-\frac{1}{2}\otimes\sum\limits_{\alpha\in S(-1)}z_\alpha[z_\alpha^*,a_i])(b_i\otimes1-\frac{1}{2}\otimes\sum\limits_{\alpha\in S(-1)}z_\alpha[z_\alpha^*,b_i])\\
=&\sum\limits_{i\in I}(-1)^{|i|}a_ib_i\otimes1-\frac{1}{2}\sum\limits_{\alpha\in S(-1),i\in I}((-1)^{|i|}a_i\otimes z_\alpha[z_\alpha^*,b_i]+b_i\otimes z_\alpha[z_\alpha^*,a_i])\\
&+\frac{1}{4}\sum\limits_{\alpha,\beta\in S(-1),i\in I}z_\alpha[z_\alpha^*,b_i]z_\beta[z_\beta^*,a_i]\\
=&\sum\limits_{i\in I}(-1)^{|i|}a_ib_i\otimes1-\sum\limits_{\alpha,\beta\in S(-1)}(-1)^{|\alpha|+|\beta|}[[e,z_\alpha^*],z_\beta^*]\otimes z_\beta z_\alpha-\frac{(s-\sfr)}{4}h\otimes1\\
&+\frac{1}{4}\sum\limits_{\alpha,\beta,\gamma,\delta\in S(-1)}([[e,z_\delta^*],z_\beta^*]^\sharp,[[e,z_\gamma^*],z_\alpha^*]^\sharp)\otimes z_\alpha z_\gamma z_\beta z_\delta.
\end{split}
\end{equation}

Recall that in \eqref{rewrite} we introduced a central element of $W_\chi'$:
\begin{equation*}
\begin{split}
C=&2e\otimes1+\frac{h^2}{2}\otimes1-(1+\frac{s-\sfr}{2})h\otimes1+\sum\limits_{i\in I}(-1)^{|i|}a_ib_i\otimes1\\
&+2\sum\limits_{\alpha\in S(-1)}(-1)^{|\alpha|}[e,z_\alpha^*]\otimes z_\alpha.
\end{split}
\end{equation*}
In view of \eqref{3.54} we have
\begin{equation}\label{3.56}
\begin{split}
C-\Theta_{\text{Cas}}=&(2e+\frac{h^2}{2}-(1+\frac{(s-\sfr)}{4})h)\otimes1+2\sum\limits_{\alpha\in S(-1)}(-1)^{|\alpha|}[e,z_\alpha^*]\otimes z_\alpha\\
&+\sum\limits_{\alpha,\beta\in S(-1)}(-1)^{|\alpha|+|\beta|}[[e,z_\alpha^*],z_\beta^*]\otimes z_\beta z_\alpha\\
&-\frac{1}{4}\sum\limits_{\alpha,\beta,\gamma,\delta\in S(-1)}([[e,z_\delta^*],z_\beta^*]^\sharp,[[e,z_\gamma^*],z_\alpha^*]^\sharp)\otimes z_\alpha z_\gamma z_\beta z_\delta.
\end{split}
\end{equation}
We finally combine \eqref{3.45} and \eqref{3.56} to deduce that
\begin{equation}\label{3.57}
\begin{split}
&[\Theta_{w_1},\Theta_{w_2}]+\frac{1}{2}\sum\limits_{\alpha\in S(-1)}\Theta_{[w_1,z_\alpha]^{\sharp}}\Theta_{[z_\alpha^*, w_2]^{\sharp}}-\frac{1}{2}\sum\limits_{\alpha\in S(-1)}(-1)^{|w_1||w_2|}\Theta_{[w_2,z_\alpha]^{\sharp}}\Theta_{[z_\alpha^*, w_1]^{\sharp}}\\
&-\frac{1}{2}([w_1,w_2],f)(C-\Theta_{\text{Cas}})\\
=&\frac{1}{8}\sum\limits_{\alpha,\beta,\gamma\in S(-1)}(-1)^{(|\alpha|+|w_1|)(|\alpha|+|w_2|)}\otimes z_\beta[z_\beta^*,[z_\alpha^*, w_2]]z_\gamma[z_\gamma^*,[w_1,z_\alpha]]\\
&-\frac{1}{8}\sum\limits_{\alpha,\beta,\gamma\in S(-1)}(-1)^{|\alpha||w_1|+|\alpha||w_2|+|\alpha|}\otimes z_\beta[z_\beta^*,[z_\alpha^*, w_1]]z_\gamma[z_\gamma^*,[w_2,z_\alpha]]\\
&+\frac{1}{8}([w_1,w_2],f)\sum\limits_{\alpha,\beta,\gamma,\delta\in S(-1)}([[e,z_\delta^*],z_\gamma^*]^\sharp,[[e,z_\beta^*],z_\alpha^*]^\sharp)\otimes z_\alpha z_\beta z_\gamma z_\delta\\
&+\frac{(s-\sfr)^2}{64}([w_1,w_2],f)-1\otimes[\varphi_{w_1},\varphi_{w_2}],
\end{split}
\end{equation}
where $\varphi_{w_1}$ and $\varphi_{w_2}$ are defined as in \eqref{varphi}, taking $w_1$ and $w_2$ in place of $w$, respectively.

By a direct but rather lengthy calculation, we can obtain that the right hand of \eqref{3.57} equals
\begin{equation}\label{3.58}
\begin{split}
&-\frac{1}{24}\sum\limits_{\alpha,\beta\in S(-1)}(-1)^{|\alpha||w_1|+|\beta||w_1|+|\alpha||\beta|}\otimes[[z_\beta,[z_\alpha,w_1]],[z_\beta^*,[z_\alpha^*,w_2]]]\\
&+\frac{3(s-\sfr)+4}{24}([w_1,w_2],f).
\end{split}
\end{equation}
As \eqref{3.58} is computed by brute force, the completely elementary yet tedious proof will be omitted.

\subsection{}\label{6.4} Now we are in a position to give the proof of Proposition \ref{1commutator}. The arguments will be divided into two steps, according to whether or not $\mathfrak{g}=\mathfrak{sl}(2|2)/\mathbb{C}I$.

(Step 1) Set \begin{equation}\label{bw}
\begin{split}
B(w_1,w_2):=&[\Theta_{w_1},\Theta_{w_2}]-\frac{1}{2}([w_1,w_2],f)(C-\Theta_{\text{Cas}})+\frac{1}{2}\sum\limits_{\alpha\in S(-1)}\Theta_{[w_1,z_\alpha]^{\sharp}}\Theta_{[z_\alpha^*, w_2]^{\sharp}}\\
&-\frac{1}{2}\sum\limits_{\alpha\in S(-1)}(-1)^{|w_1||w_2|}\Theta_{[w_2,z_\alpha]^{\sharp}}\Theta_{[z_\alpha^*, w_1]^{\sharp}},
\end{split}
\end{equation}
which is an element in $W_\chi'$. Moreover, the discussion in \eqref{3.57} shows that $B(w_1,w_2)=1\otimes b(w_1,w_2)$ for some
$b(w_1,w_2)\in A_e$. In conjunction with Lemma \ref{hwc2} this shows that $b(w_1,w_2)\in\mathbb{C}$ for all $w_1, w_2\in\ggg^e(1)$. Then
\begin{equation*}
\begin{array}{lcll}
b:&\ggg^e(1)\times\ggg^e(1)&\rightarrow&\mathbb{C}\\
&(w_1,w_2)&\mapsto &b(w_1,w_2)
\end{array}\end{equation*} is a bilinear form on $\ggg^e(1)$. Moreover, it is immediate from \eqref{3.57} and \eqref{3.58} that this bilinear form is even, i.e., $b(w_1,w_2)=0$ if $w_1$ and $w_2$ have different parities. In the following discussion  we will show that this bilinear form is invariant under the adjoint action of $\ggg^e(0)_{\bar0}$, i.e., for any $v\in\ggg^e(0)_{\bar0}$, we will prove
\begin{equation}\label{5.}
b([w_1,v],w_2)-b(w_1,[v,w_2])=0.
\end{equation}

First note that
\begin{equation}\label{invariantform}
\begin{split}
&b([w_1,v],w_2)-b(w_1,[v,w_2])\\
=&[\Theta_{[{w_1},v]},\Theta_{w_2}]-[\Theta_{w_1},\Theta_{[v,{w_2}]}]-\frac{1}{2}([[{w_1},v],w_2],f)(C-\Theta_{\text{Cas}})\\
&+\frac{1}{2}([w_1,[v,{w_2}]],f)(C-\Theta_{\text{Cas}})
+\frac{1}{2}\sum\limits_{\alpha\in S(-1)}\Theta_{[[{w_1},v],z_\alpha]^{\sharp}}\Theta_{[z_\alpha^*, w_2]^{\sharp}}\\
&-\frac{1}{2}\sum\limits_{\alpha\in S(-1)}\Theta_{[w_1,z_\alpha]^{\sharp}}\Theta_{[z_\alpha^*,[v,{w_2}]]^{\sharp}}
-\frac{1}{2}\sum\limits_{\alpha\in S(-1)}(-1)^{|w_1||w_2|}\Theta_{[w_2,z_\alpha]^{\sharp}}\Theta_{[z_\alpha^*,[{w_1},v]]^{\sharp}}\\
&+\frac{1}{2}\sum\limits_{\alpha\in S(-1)}(-1)^{|w_1||w_2|}\Theta_{[[v,{w_2}],z_\alpha]^{\sharp}}\Theta_{[z_\alpha^*, w_1]^{\sharp}}.
\end{split}
\end{equation}
In virtue of Proposition \ref{vw}, we have
\begin{equation}\label{key1}
\begin{split}
[\Theta_{[w_1,v]},\Theta_{w_2}]-[\Theta_{w_1},\Theta_{[v,w_2]}]
=&[[\Theta_{w_1},\Theta_v],\Theta_{w_2}]
-[\Theta_{w_1},[\Theta_v,\Theta_{w_2}]]\\
=&[[\Theta_{w_1},\Theta_{w_2}],\Theta_v].
\end{split}
\end{equation}
Moveover, it is immediate from $v\in\ggg^e$ that
\begin{equation}\label{key2}
\begin{split}
-([[{w_1},v],w_2],f)+([w_1,[v,{w_2}]],f)=&-([[w_1,w_2],v],f)\\
=&-([w_1,w_2],f)([e,v],f)=0.
\end{split}
\end{equation}

For the remaining terms in \eqref{invariantform}, first note that
\begin{equation}\label{thegoal}
\begin{split}
&\sum\limits_{\alpha\in S(-1)}\Theta_{[[w_1,v],z_\alpha]^{\sharp}}\Theta_{[z_\alpha^*, w_2]^{\sharp}}-\sum\limits_{\alpha\in S(-1)}\Theta_{[w_1,z_\alpha]^{\sharp}}\Theta_{[z_\alpha^*,[v,w_2]]^{\sharp}}\\
=&\sum\limits_{\alpha\in S(-1)}\Theta_{[[w_1,z_\alpha],v]^{\sharp}}\Theta_{[z_\alpha^*, w_2]^{\sharp}}+\sum\limits_{\alpha\in S(-1)}\Theta_{[w_1,[v,z_\alpha]]^{\sharp}}\Theta_{[z_\alpha^*, w_2]^{\sharp}}\\
&-\sum\limits_{\alpha\in S(-1)}\Theta_{[w_1,z_\alpha]^{\sharp}}\Theta_{[[z_\alpha^*,v],w_2]^{\sharp}}-\sum\limits_{\alpha\in S(-1)}\Theta_{[w_1,z_\alpha]^{\sharp}}\Theta_{[v,[z_\alpha^*,w_2]]^{\sharp}}.
\end{split}
\end{equation}

For any $\alpha\in S(-1)$, it is worth noting that both $[v,z_\alpha]$ and $[z_\alpha^*,v]$ are in $\ggg(-1)$, then we have
\begin{equation}\label{extend}
[v,z_\alpha]=\sum\limits_{\beta\in S(-1)}\langle z_\beta^*,[v,z_\alpha]\rangle z_\beta,\qquad
[z_\alpha^*,v]=-\sum\limits_{\beta\in S(-1)}(-1)^{|\beta|}\langle z_\beta,[z_\alpha^*,v]\rangle z_\beta^*.
\end{equation}
We claim that the bilinear form $\langle\cdot,\cdot\rangle$ is $\ggg^e(0)_{\bar0}$-invariant. In fact, for any $v\in\ggg^e(0)_{\bar0}$ and $\alpha, \beta\in S(-1)$, we have
\begin{equation}\label{angle}
\begin{split}
\langle[z_\alpha,v],z_\beta\rangle=&(e,[[z_\alpha,v],z_\beta])
=(e,[z_\alpha,[v,z_\beta]])+(e,[[z_\alpha,z_\beta],v])\\
=&(e,[z_\alpha,[v,z_\beta]])-(e,[v,[z_\alpha,z_\beta]])\\
=&(e,[z_\alpha,[v,z_\beta]])-([e,v],[z_\alpha,z_\beta])\\
=&(e,[z_\alpha,[v,z_\beta]])=\langle z_\alpha,[v,z_\beta]\rangle.
\end{split}
\end{equation}

Taking \eqref{extend} and \eqref{angle} into account, one can conclude that
\begin{equation}\label{3.66angle}
\begin{split}
&\sum\limits_{\alpha\in S(-1)}\Theta_{[w_1,[v,z_\alpha]]^{\sharp}}\Theta_{[z_\alpha^*, w_2]^{\sharp}}-\sum\limits_{\alpha\in S(-1)}\Theta_{[w_1,z_\alpha]^{\sharp}}\Theta_{[[z_\alpha^*,v],w_2]^{\sharp}}\\
=&\sum\limits_{\alpha,\beta\in S(-1)}\langle z_\beta^*,[v,z_\alpha]\rangle \Theta_{[w_1,z_\beta]^{\sharp}}\Theta_{[z_\alpha^*, w_2]^{\sharp}}+\sum\limits_{\alpha,\beta\in S(-1)}(-1)^{|\beta|}\langle z_\beta,[z_\alpha^*,v]\rangle \Theta_{[w_1,z_\alpha]^{\sharp}}\Theta_{[z_\beta^*,w_2]^{\sharp}}\\
=&\sum\limits_{\alpha,\beta\in S(-1)}\langle z_\alpha^*,[v,z_\beta]\rangle \Theta_{[w_1,z_\alpha]^{\sharp}}\Theta_{[z_\beta^*, w_2]^{\sharp}}-\sum\limits_{\alpha,\beta\in S(-1)}\langle [z_\alpha^*,v],z_\beta\rangle \Theta_{[w_1,z_\alpha]^{\sharp}}\Theta_{[z_\beta^*,w_2]^{\sharp}}\\
=&\sum\limits_{\alpha,\beta\in S(-1)}\langle z_\alpha^*,[v,z_\beta]\rangle \Theta_{[w_1,z_\alpha]^{\sharp}}\Theta_{[z_\beta^*, w_2]^{\sharp}}-\sum\limits_{\alpha,\beta\in S(-1)}\langle z_\alpha^*,[v,z_\beta]\rangle \Theta_{[w_1,z_\alpha]^{\sharp}}\Theta_{[z_\beta^*,w_2]^{\sharp}}\\
=&0.
\end{split}
\end{equation}

For any $w\in\ggg^e(1)$ and $\alpha\in S(-1)$, it follows from the definition of ${\sharp}$ in \S\ref{3.1.1} that $[w,z_\alpha]^{\sharp}\in\ggg^e(0)$, then we have
\begin{equation*}
\begin{split}
[[w,z_\alpha]^{\sharp},v]^{\sharp}=&[[w,z_\alpha]^{\sharp},v]-\frac{1}{2}(h,[[w,z_\alpha]^{\sharp},v])h\\
=&[[w,z_\alpha]^{\sharp},v]-\frac{1}{2}([h,[w,z_\alpha]],v)h\\
=&[[w,z_\alpha]^{\sharp},v].
\end{split}
\end{equation*}
As $v\in\ggg^e(0)_{\bar0}$, and $\mathbb{C}h$ is orthogonal to $\ggg^e(0)$ with respect to $(\cdot,\cdot)$, then
\begin{equation*}
[[w,z_\alpha],v]^{\sharp}=[([w,z_\alpha]-\frac{1}{2}(h,[w,z_\alpha])h),v]^{\sharp}=[[w,z_\alpha]^{\sharp},v]^{\sharp}=[[w,z_\alpha]^{\sharp},v].
\end{equation*}
By the same discussion we have $[v,[z_\alpha^*,w_2]]^{\sharp}=[v,[z_\alpha^*,w_2]^{\sharp}]$.
In view of Proposition \ref{v1v2} this yields
\begin{equation}\label{thethird}
\begin{split}
&\sum\limits_{\alpha\in S(-1)}\Theta_{[[w_1,z_\alpha],v]^{\sharp}}\Theta_{[z_\alpha^*, w_2]^{\sharp}}-\sum\limits_{\alpha\in S(-1)}\Theta_{[w_1,z_\alpha]^{\sharp}}\Theta_{[v,[z_\alpha^*,w_2]]^{\sharp}}\\
=&\sum\limits_{\alpha\in S(-1)}\Theta_{[[w_1,z_\alpha]^{\sharp},v]}\Theta_{[z_\alpha^*, w_2]^{\sharp}}-\sum\limits_{\alpha\in S(-1)}\Theta_{[w_1,z_\alpha]^{\sharp}}\Theta_{[v,[z_\alpha^*,w_2]^{\sharp}]}\\
=&\sum\limits_{\alpha\in S(-1)}[\Theta_{[w_1,z_\alpha]^{\sharp}},\Theta_v]\Theta_{[z_\alpha^*, w_2]^{\sharp}}-\sum\limits_{\alpha\in S(-1)}\Theta_{[w_1,z_\alpha]^{\sharp}}[\Theta_v,\Theta_{[z_\alpha^*,w_2]^{\sharp}}]\\
=&\sum\limits_{\alpha\in S(-1)}[\Theta_{[w_1,z_\alpha]^{\sharp}}\Theta_{[z_\alpha^*, w_2]^{\sharp}},\Theta_v].
\end{split}
\end{equation}

Combining \eqref{thegoal} with \eqref{3.66angle} and \eqref{thethird}, we obtain
\begin{equation}\label{reachthegoal1}
\sum\limits_{\alpha\in S(-1)}\Theta_{[[w_1,v],z_\alpha]^{\sharp}}\Theta_{[z_\alpha^*, w_2]^{\sharp}}-\sum\limits_{\alpha\in S(-1)}\Theta_{[w_1,z_\alpha]^{\sharp}}\Theta_{[z_\alpha^*,[v,w_2]]^{\sharp}}
=\sum\limits_{\alpha\in S(-1)}[\Theta_{[w_1,z_\alpha]^{\sharp}}\Theta_{[z_\alpha^*, w_2]^{\sharp}},\Theta_v].
\end{equation}
Interchanging the roles of $w_1$ and $w_2$ in \eqref{reachthegoal1}, we have
\begin{equation}\label{reachthegoal2}
\begin{split}
&\sum\limits_{\alpha\in S(-1)}(-1)^{|w_1||w_2|}\Theta_{[[v,w_2],z_\alpha]^{\sharp}}\Theta_{[z_\alpha^*, w_1]^{\sharp}}-\sum\limits_{\alpha\in S(-1)}(-1)^{|w_1||w_2|}\Theta_{[w_2,z_\alpha]^{\sharp}}\Theta_{[z_\alpha^*,[w_1,v]]^{\sharp}}\\
=&-(-1)^{|w_1||w_2|}\sum\limits_{\alpha\in S(-1)}[\Theta_{[w_2,z_\alpha]^{\sharp}}\Theta_{[z_\alpha^*, w_1]^{\sharp}},\Theta_v].
\end{split}
\end{equation}
Moreover, it is worth noting that $[C-\Theta_{\text{Cas}},\Theta_v]=0$ by Proposition \ref{vcommutate}.
As an immediate consequence of \eqref{invariantform}, \eqref{key1}, \eqref{key2}, \eqref{reachthegoal1} and \eqref{reachthegoal2}, we have
\begin{equation*}
\begin{split}
&b([w_1,v],w_2)-b(w_1,[v,w_2])\\
=&[[\Theta_{w_1},\Theta_{w_2}],\Theta_v]+\frac{1}{2}\sum\limits_{\alpha\in S(-1)}[\Theta_{[w_1,z_\alpha]^{\sharp}}\Theta_{[z_\alpha^*, w_2]^{\sharp}},\Theta_v]\\
&-\frac{1}{2}\sum\limits_{\alpha\in S(-1)}(-1)^{|w_1||w_2|}[\Theta_{[w_2,z_\alpha]^{\sharp}}\Theta_{[z_\alpha^*, w_1]^{\sharp}},\Theta_v]\\
=&[([\Theta_{w_1},\Theta_{w_2}]-\frac{1}{2}([w_1,w_2],f)(C-\Theta_{\text{Cas}})+\frac{1}{2}\sum\limits_{\alpha\in S(-1)}\Theta_{[w_1,z_\alpha]^{\sharp}}\Theta_{[z_\alpha^*, w_2]^{\sharp}}\\
&-\frac{1}{2}\sum\limits_{\alpha\in S(-1)}(-1)^{|w_1||w_2|}\Theta_{[w_2,z_\alpha]^{\sharp}}\Theta_{[z_\alpha^*, w_1]^{\sharp}}),\Theta_v]\\
=&[1\otimes b(w_1,w_2),\Theta_v]=0,
\end{split}
\end{equation*}
where the last equation comes from the fact that $b(w_1,w_2)\in\mathbb{C}$. We complete the proof of \eqref{5.}.

On the other hand, Kac-Roan-Wakimoto gave a description of the $\ggg(0)^\sharp$-module $\ggg(1)$ in \cite[Proposition 4.1]{KRW} (see also \cite[Table 1-Table 3]{KW}), which showed that except the case with $\ggg=\mathfrak{sl}(2|2)/\mathbb{C}I$ (where $I$ denotes the unitary matrix), either $\ggg^e(1)$ is an irreducible $\ggg(0)^\sharp$-module, or $\ggg^e(1)\cong M\oplus M^*$ with $M$ and $M^*$ being irreducible $\ggg(0)^\sharp$-modules such that $M\ncong M^*$ (note that the grading on $\ggg$ applied there was under the action of $\text{ad}\,\frac{h}{2}$, thus $\ggg_{\frac{1}{2}}$ there is just $\ggg(1)$ in our case). If just consider the even part of $\ggg(0)^\sharp$, one can readily conclude from \cite[Table 1-Table 3]{KW} that $\ggg(0)^\sharp_{\bar0}$-module $\ggg(1)$ has the same property as described above,
except the case $\mathfrak{sl}(2|2)/\mathbb{C}I$.

Since $\ggg(0)^\sharp_{\bar0}=\ggg^e(0)_{\bar0}$ and $\ggg(1)=\ggg^e(1)$ in this situation, from all the discussion above one can conclude that $b=c_0([\cdot,\cdot],f)$ for some $c_0\in\mathbb{C}$ except  $\ggg=\mathfrak{sl}(2|2)/\mathbb{C}I$.

(Step 2) Now we turn to  the case when $\ggg=\mathfrak{sl}(2|2)/\mathbb{C}I$. In this case, it follows from \cite[Table 2]{KW}
that the $\ggg^e(0)_{\bar0}$-module $\ggg^e(1)$ is isomorphic to $\mathfrak{sl}_2$-module $\mathbb{C}^2\oplus\mathbb{C}^2$, thus the discussion in the end of (Step 1) can not be applied. So we need to calculate the value of \eqref{3.58} with $w_1,w_2\in\ggg^e(1)$.

Recall that $\mathfrak{sl}(2|2)\subseteq\mathfrak{gl}(2|2)$ consists of $4\times4$ matrices in the following $(2|2)$-block form
\begin{center}
$\begin{matrix}
\begin{pmatrix} A & B\\C & D \end{pmatrix}
\end{matrix}$,
\end{center}
where $A, B, C, D$ are all $2\times2$ matrices, and $\text{tr}\,A-\text{tr}\,D=0$. Then $\ggg$ is a quotient of  $\mathfrak{sl}(2|2)$ by the scalars of unitary matrix $I$. Denote by $e_{\bar i\bar j}, e_{\bar ij}, e_{i\bar j}, e_{ij}\in\mathfrak{gl}(2|2)$ the matrix with $1$ in $ij$-entry of $A, B, C, D$ respectively, and $0$ other entries.
It is a direct consequence from PBW theorem that $\ggg$ has a basis
\begin{equation*}
\begin{split}
&h=e_{\bar1\bar1}-e_{\bar2\bar2}, H_1=e_{\bar2\bar2}+e_{11}, H_2=e_{11}-e_{22},\\
&e_{\bar1\bar2}, e_{\bar2\bar1}, e_{\bar11},e_{1\bar1},e_{\bar12}, e_{1\bar2}, e_{\bar21}, e_{2\bar1}, e_{\bar22}, e_{2\bar2}, e_{12}, e_{21}.
\end{split}
\end{equation*}

Consider the $\mathfrak{sl}_2$-triple $(e,h,f)$ with $e=e_{\bar 1\bar 2}, h=e_{\bar 1\bar 1}-e_{\bar 2\bar 2}, f=e_{\bar 2\bar 1}$. It is obvious that $e$ is a minimal nilpotent element in $\ggg$. Set $str(\cdot,\cdot)$ to be the nondegenerate supersymmetric invariant bilinear form on $\ggg$, then we have $(e,f)=\frac{1}{2}(h,h)=1$.
Let ${\ggg}(i)=\{x\in{\ggg}\mid[h,x]=ix\}$, then ${\ggg}=\bigoplus_{i\in{\bbz}}{\ggg}(i)$.  It can be observed that
\begin{equation*}
\begin{array}{rlrl}
h, H_1, H_2, e_{12}, e_{21}&\in\ggg(0)_{\bar0},&&\\
e_{\bar11}, e_{\bar12}, e_{1\bar2}, e_{2\bar2}&\in\ggg(1)_{\bar1},&e_{\bar1\bar2}&\in\ggg(2),\\
e_{\bar21}, e_{\bar22}, e_{1\bar1}, e_{2\bar1}&\in\ggg(-1)_{\bar1},&e_{\bar2\bar1}&\in\ggg(-2).
\end{array}
\end{equation*}

As
\begin{equation*}
h+2H_1, H_2, e_{12}, e_{21}, e_{\bar1\bar2}, e_{\bar11}, e_{\bar12}, e_{1\bar2}, e_{2\bar2}
\end{equation*}
constitute a basis of $\ggg^e$, by Proposition \ref{v1v2} and Proposition \ref{vw} we can choose
\begin{equation*}
h+2H_1, e_{12}, e_{21}\in\ggg(0)_{\bar0},\qquad e_{\bar11}, e_{1\bar2}\in\ggg(1)_{\bar1}
\end{equation*}
as the generators of $\ggg^e$. Note that
\begin{equation*}1\otimes e_{\bar21}^2=\frac{1}{2}\otimes[e_{\bar21},e_{\bar21}]=0,\end{equation*}
and by the same discussion we have
\begin{equation*}
1\otimes e_{\bar22}^2=1\otimes e_{1\bar1}^2=1\otimes e_{2\bar1}^2=0.
\end{equation*}
Moreover, $e_{\bar21}$ and $e_{\bar22}$ are dual bases of $e_{1\bar1}$ and $e_{2\bar1}$ with respect to $(e,[\cdot,\cdot])$. 

Let $w_1, w_2$ be any of $e_{\bar11}$ and $e_{1\bar2}$. By case-by-case calculations one can obtain that
\begin{equation*}
\sum\limits_{\alpha,\beta\in S(-1)}(-1)^{|\alpha||w_1|+|\beta||w_1|+|\alpha||\beta|}[[z_\beta,[z_\alpha,w_1]],[z_\beta^*,[z_\alpha^*,w_2]]]=4([w_1,w_2],f).
\end{equation*}
Taking $ s=0, \sfr=4$ into account, it is immediate from \eqref{3.57} and \eqref{3.58} that
\begin{equation*}
\begin{split}
[\Theta_{w_1},\Theta_{w_2}]=&\frac{1}{2}([w_1,w_2],f)(C-\Theta_{\text{Cas}}-1)-\frac{1}{2}\sum\limits_{\alpha\in S(-1)}(\Theta_{[w_1,z_\alpha]^{\sharp}}\Theta_{[z_\alpha^*, w_2]^{\sharp}}\\
&-(-1)^{|w_1||w_2|}\Theta_{[w_2,z_\alpha]^{\sharp}}\Theta_{[z_\alpha^*, w_1]^{\sharp}}).
\end{split}
\end{equation*}
Thus $c_0=1$ in this case.

Summing up both results in Steps (1) and (2), and also \eqref{3.57}, \eqref{3.58}, we complete the proof of Proposition \ref{1commutator}.

\begin{rem}
For the Lie algebra version of Proposition \ref{1commutator}, Premet made a crucial use of the machinery of associated varieties and Joseph ideal for Lie algebras, by which Premet calculated the exact values of $c_0$ for each type of simple Lie algebras, respectively. Unfortunately, we are in lack of such a powerful tool for Lie superalgebras. On the other hand, it is still a hard work to compute the exact values of $c_0$ in Proposition \ref{1commutator} with \eqref{3.58}. Since our finial result in Theorem \ref{intromainminnimalf} does not have much to do with the exact values of $c_0$, we  give up the calculation of $c_0$ in the present paper.
\end{rem}

\vskip0.2cm

{\bf Acknowledgements} \; The authors got much help from Suh, who explained some results in her paper \cite{suh}. The authors express great thanks to her.

\end{document}